\documentclass[a4paper,11pt]{amsart}

\usepackage[margin=3.0cm]{geometry}
\usepackage{amsthm,amsmath,amssymb}
\usepackage{mathtools}
\usepackage[utf8]{inputenc}
\usepackage[english]{babel}
\usepackage[linktocpage]{hyperref}
\hypersetup{colorlinks=true,linkcolor=blue,citecolor=blue,filecolor=magenta,urlcolor=blue}
\usepackage{nicematrix}
\usepackage{tikz}
\usepackage{pgf}
\usetikzlibrary{calc,cd}
\usepackage{placeins}
\usepackage{enumitem}

\theoremstyle{plain}
\newtheorem{theorem0}{Theorem}
\newtheorem{theorem}{Theorem}[section]
\newtheorem{prop}[theorem]{Proposition}
\newtheorem{cor}[theorem]{Corollary}
\newtheorem{lemma}[theorem]{Lemma}

\theoremstyle{definition}
\newtheorem{defi}[theorem]{Definition}
\newtheorem{ex}[theorem]{Example}
\newtheorem{red}[theorem]{Reduction}
\newtheorem{step}{Step}
\newtheorem{remdef}[theorem]{Remark-Definition}

\theoremstyle{remark}
\newtheorem{remark}[theorem]{Remark}

\def\ZZ{\mathbb{Z}}
\def\RR{\mathbb{R}}
\def\QQ{\mathbb{Q}}
\def\CC{\mathbb{C}}
\def\PP{\mathbb{P}}
\def\cO{\mathcal{O}}

\def\b{\ell}
\def\sv{c_D}
\def\sf{f}
\def\l{l}
\def\t{\tau}
\def\m{\mu}
\def\Hv{\mathbb{H}_v}
\def\Hh{\mathbb{H}_h}
\def\Hm{\mathbb{H}_m}
\def\p{{w_x}}
\def\q{{w_y}}
\def\r{{w_z}}
\def\h{\varphi_D}

\def\Si{S}
\def\s{\sigma}

\newcommand{\floor}[1]{\left \lfloor #1 \right \rfloor}

\newcommand{\ai}[1]{A_{i,(#1)}}
\newcommand{\fd}{F_{(d)}}
\newcommand{\dcover}{S_d}
\newcommand{\orb}{\mathcal{O}}
\newcommand{\vdeg}{\kappa}
\newcommand{\pmult}{\nu}
\newcommand{\newmult}{m_\omega}
\newcommand{\mk}{\delta}
\newcommand{\xp}{e}
\newcommand{\mq}{m}
\newcommand{\mdet}{c}
\newcommand{\mcd}{\delta_\omega}
\newcommand{\bzt}{\beta}
\newcommand{\mcdhv}{\nu}

\DeclareMathOperator{\mult}{mult}
\DeclareMathOperator*{\lcm}{lcm}
\DeclareMathOperator{\Div}{Div}
\DeclareMathOperator{\tor}{Tor}
\DeclareMathOperator{\Weil}{Weil}
\DeclareMathOperator{\Cart}{Cart}
\DeclareMathOperator{\ord}{ord}
\DeclareMathOperator{\exc}{Exc}

\DeclareMathOperator{\sing}{Sing}
\DeclareMathOperator{\Proj}{Proj}
\DeclareMathOperator{\Jac}{Jac}

\DeclareMathOperator{\cl}{Cl}

\let\underbrace\LaTeXunderbrace
\let\overbrace\LaTeXoverbrace

\title[Cyclic coverings of rational normal surfaces]
{Cyclic coverings of rational normal surfaces which are quotients of a product of curves}

\author[E.~Artal]{Enrique Artal Bartolo}
\address[E.~Artal Bartolo]{Departamento de Matem\'{a}ticas, IUMA \\
Universidad de Zaragoza \\
C.~Pedro Cerbuna 12, 50009, Zaragoza, Spain}
\urladdr{\url{http://riemann.unizar.es/~artal}}
\email{\href{mailto:artal@unizar.es}{artal@unizar.es}} 

\author[J.I.~Cogolludo]{Jos{\'e} Ignacio Cogolludo-Agust{\'i}n}
\address[J.I.~Cogolludo-Agustín]{Departamento de Matem\'aticas, IUMA \\ 
Universidad de Za\-ra\-go\-za \\ 
C.~Pedro Cerbuna 12 \\ 
50009 Zaragoza, Spain} 
\urladdr{\url{http://riemann.unizar.es/~jicogo}}
\email{\href{mailto:jicogo@unizar.es}{jicogo@unizar.es}} 

\author[J.~Mart\'{\i}n-Morales]{Jorge Mart\'{\i}n-Morales}
\address[J.~Martín-Morales]{Departamento de Matem\'aticas, IUMA \\ 
Universidad de Za\-ra\-go\-za \\ 
C.~Pedro Cerbuna 12 \\ 
50009 Zaragoza, Spain} 
\urladdr{\url{http://riemann.unizar.es/~jorge}}
\email{\href{mailto:jorge@unizar.es}{jorge@unizar.es}}

\thanks{\noindent 
Partially supported by MCIN/AEI/10.13039/501100011033 (grant code: PID2020-114750GB-C31)
and by Departamento de Ciencia, Universidad y Sociedad del Conocimiento del Gobierno de Arag{\'o}n 
(grant code: E22\_20R: ``{\'A}lgebra y Geometr{\'i}a'').
The third named author is supported by MCIN/AEI/10.13039/501100011033 and the European Union 
\emph{NextGenerationEU/PRTR} (grant code: RYC2021-034300-I) and also by Junta de Andaluc{\'\i}a (FQM-333).}

\newcounter{relctr}
\everydisplay\expandafter{\the\everydisplay\setcounter{relctr}{0}}

\AtBeginDocument{}

\begin{document}

\begin{abstract}
This paper deals with cyclic covers of a large family of rational normal surfaces that can also 
be described as quotients of a product, where the factors are cyclic covers of algebraic curves.
We use a generalization of Esnault-Viehweg method to show that the action of the monodromy on 
the first Betti group of the covering (and its Hodge structure) splits as a direct sum of the 
same data for some specific cyclic covers over~$\PP^1$.

This has applications to the study of L{\^e}-Yomdin surface singularities, in particular to the 
action of the monodromy on the Mixed Hodge Structure, as well as to isotrivial fibered surfaces.
\end{abstract}

\maketitle
\tableofcontents

\section*{Introduction}

The general framework of this paper is the study of the cohomology of cyclic covers of normal projective surfaces
ramified along a curve, i.e.~a Weil divisor. 

In the smooth case, for 
instance the projective plane, cyclic coverings of $\PP^2$ ramified along a curve have been intensively studied since 
Zariski~\cite{Zariski-problem,Zariski-irregularity}. The first approaches used the degree of the Alexander polynomial of the 
complement of the ramification locus to calculate the irregularity of the covering ramified along the curve with order the 
degree of the curve. In the 80's a series of papers 
(\cite{Esnault-Viehweg82,es:82,Libgober-alexander,Sabbah-Alexander,Loeser-Vaquie-Alexander,Artal94})
allowed for a computation of the irregularity of the covering that was independent of its fundamental group.

The problem of computing the irregularity and structure of cyclic (or more generally, abelian) covers 
is relevant in its own right, e.g. in~\cite{ap:12,pardini} for the non-cyclic case, with a focus on their global structure (including their singularities whether the base is smooth or singular). Note that
the first Betti numbers (i.e twice the irregularity) of arbitrary finite abelian coverings can
be retrieved from those of finite cyclic covers.

Our main motivation for this work, however, stems from the study of surface singularities.
An important invariant of a singular surface in $\mathbb{C}^3$ is the mixed Hodge structure of the 
cohomology of its Milnor fiber~\cite{Deligne71a, Deligne71b, Deligne74}. In the isolated case this 
structure can be described by using Steenbrink spectral sequence~\cite{Steenbrink77} associated with 
the semistable reduction of a resolution of the singularity~\cite{Mumford-topology}.
A crucial ingredient to understand this spectral sequence is the fact that the restriction of the semistable reduction to each of the exceptional divisors of the resolution is a cyclic branched covering ramified on a normal crossing
divisor in a smooth surface. Esnault-Viehweg's theory~\cite{Esnault-Viehweg82} can be used to compute the equivariant
first Betti numbers of these coverings (by the Hodge symmetries only these are needed).

In practice, however, either the embedded resolutions are too difficult to compute or their structure is too complicated,
and only a few explicit examples are known in the literature. 

The first named author computed in~\cite{ea:mams} the Hodge structure of superisolated surface singularities~\cite{Luengo87},
providing a counterexample to Yau's conjecture~\cite{Yau}. 

In order to simplify this combinatorial problem, the third named author used in~\cite{jorge-Semistable} embedded 
$\mathbb{Q}$-resolutions instead which are less complicated.
For these resolutions, one needs to deal with cyclic branched coverings of surfaces with quotient singularities.
This motivated us to develop a generalization of Esnault-Viehweg's theory to this setting~\cite{ACM19}.
An explicit embedded $\mathbb{Q}$-resolution is constructed for L{\^e}-Yomdin surface singularities in~\cite{jorge-Israel}. 
In~\cite{jorge:Yomdin}, it is shown that only cyclic covers ramified along $\mathbb{Q}$-normal crossing divisors are required 
both for weighted projective planes and for the type of surfaces studied in this paper.

The construction of the surfaces studied in this work involves three cyclic branched covers:
\[
\begin{tikzcd}[row sep=-5pt]
\PP^1\rar["m_\vdeg"]&\PP^1,&G\rar["\tau", "\vdeg:1" below]&\PP^1,&\fd\rar["\pi_F", "d:1" below]&\PP^1. \\
z\rar[mapsto]&z^\vdeg
\end{tikzcd}
\]
The cover $\tau$ will be interpreted as an orbifold map $\tau:G\to\orb$, where $\orb$ is an orbifold whose underlying
manifold is $\PP^1$ and has $r$ orbifold points (the images of the branching points); in the same way $m_\vdeg$ induces
an orbifold structure $\PP^1_{d,d}$ with two orbifold points of order~$d$ at $0,\infty$.
We consider a surface $S$ as a diagonal quotient of $G\times\PP^1$ by the action of $\ZZ/\vdeg$, see 
\eqref{eq:diag} for details. The cover $\pi_F$ appears as the restriction to the second factor in the 
vertical part of the pull-back of $\pi$ and $\tau_2$, see diagram~\eqref{eq:diagconm}. 
\begin{equation}\label{eq:diagconm}
\begin{tikzcd}
G\times\fd\ar[r]\ar[d,"\mathbf{1}_G\times\pi_F" left]
&\dcover\ar[d,"\pi"]\\
G\times\PP^1\ar[r,"\tau_2"]\ar[dr, "\tau\times m_\kappa" below left]
&
S\ar[d,"\tau_3"]\\
&
\PP^1\times\PP^1
\end{tikzcd}
\end{equation}
The normal surface $S$ has $2r$ cyclic quotient points. There are several isotrivial
fibrations hidden in~\eqref{eq:diagconm}. The composition of $\tau_3$ with the
first projection can be seen as a ruled surface $S\to\orb$;
the composition of $\tau_3$ with the
second projection is an isotrivial fibration $S\to\PP^1_{d,d}$.

The surface $S_d$ inherits an $\fd$-isotrivial fibration structure $\dcover\to\orb$
and the main goal of the paper is to compute its cohomology of degree~$1$,
more precisely its eigen-decomposition
by the monodromy of~$\pi$.
The surface $S_d$ is also the finite quotient of
of $G\times\fd$ by a non-free action of $\ZZ/\vdeg$.

The case $r=2$ was studied in~\cite{campillo} and some results will be used here. The orbifolds in~\cite{campillo}
are rational, while in this work they are arbitrary orbifolds supported in $\PP^1$. 

This family can also be constructed by a series of weighted blow-ups and blow-downs of the Hirzebruch surface $\Sigma_\alpha$.
This fact allows us to determine the class groups of those surfaces~$S$ starting
from particular presentations of Picard groups of Hirzebruch surfaces.

Our strategy to describe the cyclic coverings is to use a generalization of the Esnault-Viehweg theory for cyclic coverings of 
smooth surfaces, developed in~\cite{Esnault-Viehweg82}, to normal surfaces with quotient singularities (see~\cite{ACM19}). 

From an algebraic point of view, a $d$-cyclic covering $\tilde X\to X$ of a projective normal surface $X$ with at most quotient 
singularities ramified along a Weil divisor $D$ is determined by the 
choice (and existence) of a divisor class $H\in \cl(X)$ satisfying $D\sim dH$.
As long as $\cl(X)$ has no torsion, the mere existence of $H$ is enough. Otherwise, the choice of the particular $H\in \cl(X)$ 
is also necessary (see Example~\ref{ex:torsion}). 
In this context, if $D$ is a $\QQ$-normal crossing divisor, the decomposition of 
$H^1(\tilde X,\CC)$ in invariant subspaces with respect to the action of the monodromy of the cover can be retrieved
from the Hodge decomposition of $H^1(\tilde X,\cO_{\tilde X})\oplus H^0(\tilde X,\Omega_{\tilde X}^1)$, where the invariant subspaces of the first term 
are naturally isomorphic to $H^1(X,\cO_X(L^{(l)}))$ for certain divisors $L^{(l)}$, $l=0,...,d-1$ and 
$H^0(\tilde X,\Omega_{\tilde X}^1)\cong\overline{H^1(\tilde X,\cO_{\tilde X})}$.
In case $X=\PP^1\times\PP^1$ it is well known that this Hodge decomposition by the monodromy splits in two factors $\Hh\oplus\Hv$,
that correspond with the restriction of the covering $\pi$ to a vertical and a horizontal fiber of the birule of~$X$, 
see section~\ref{sec:p1p1}.

The family of surfaces $S$ presented here has a similar splitting that is described in detail in this paper. 
In order to do that, the concept of \emph{greatest common vertical covering} for a family of coverings will be introduced.
For an explicit solution of the original problem, a description of the cohomology $H^*(S,\cO_S(D))$ of a Weil divisor $D$ is given. 
Such cohomology is oftentimes concentrated in a single degree. Concrete formulas are given in section~\ref{sec:cohomology}.
The main result of this paper is proved in section~\ref{sec:main} (Theorem~\ref{thm:gencase}) 
and it refers to the $H^1$-eigenspace decomposition by the monodromy.

\begin{theorem0}
Let $S_d$ be the cyclic covering of $S$ associated with $(d,D,H)$, $D\in \Div(S)$, $H\in\cl(S)$, 
$D\sim dH$, where $D$ has $\QQ$-normal crossings. Then 
\[
H^1(S_d,\cO_{S_d})\cong \Hh\oplus\Hv
\]
where 
\begin{itemize}
\item $\Hv$ is the $1$-cohomology of the structure sheaf of the restriction of an intermediate cover of $\pi$ to a rational horizontal \emph{fiber} and
\item $\Hh$ is the $1$-cohomology of the structure sheaf of the greatest common vertical cover of  an intermediate cover. 
\end{itemize}
In particular, $H^1(S_d,\cO_{S_d})$ splits as a direct sum of the cohomology of two cyclic covers of~$\PP^1$ 
and the splitting respects the eigenspaces of the monodromy and the Hodge structure.
\end{theorem0}

The notions of rational horizontal \emph{fiber} and \emph{greatest common vertical cover} will be 
explained in the work; the degrees of intermediate covers will be made explicit in the text.

As an outline of the paper, in section~\ref{sec:esnault-viehweg}, the general theory of Esnault-Viehweg is 
reviewed for the sake of completeness. Section~\ref{sec:settings} is devoted to the description of the family 
of surfaces $S$ and their divisor class group. In section~\ref{sec:cohomology} we give explicit formulas for 
the cohomology of Weil divisors on $S$, proving when this cohomology is concentrated in a single degree. 
In section~\ref{sec:main}, the main results are stated and proved. Some relevant examples are given in 
section~\ref{sec:ex}, including isotrivial fibered surfaces.
The paper ends with cyclic covers appearing in weighted L{\^e}-Yomdin singularities, see section~\ref{sec:yomdin}. 

\section*{Acknowledgement}
We want to thank the anonymous referees of the manuscript for their insightful comments, which have helped improve 
the exposition of this paper.

\numberwithin{equation}{section}

\section{Esnault and Viehweg's method}
\label{sec:esnault-viehweg}

For notation, let $\zeta_m:=\exp\frac{2\pi\sqrt{-1}}{m}$.

\subsection{Cyclic covers of abelian quotient singular surfaces}\mbox{}
\label{sec:Esnault}

In~\cite{Esnault-Viehweg82} the authors set the theory of ramified cyclic covers 
of projective smooth varieties. In this paper we are going to use a generalization
of this theory for projective normal surfaces having cyclic quotient singularities \cite{ACM19}.
Let $X$ be either a projective smooth variety or a projective surface
with cyclic quotient singularities. A cyclic cover $\pi:\tilde{X}\to X$
is algebraically determined by three data $(d,D,H)$ where $D$ is linearly equivalent to $dH$.
The number of sheets of $\pi$ is~$d$ and the ramification locus is a divisor~$D$. We emphasize that
$H \in \cl(X)$ while $D$ is a true Weil divisor in $\Div(X)$ and not just a linear equivalence class
in $\cl(X)$. Let 
\[
D:=\sum_{i=0}^r n_i D_i
\]
be the decomposition of $D$ into irreducible divisors.
Such a cyclic cover is defined topologically by a morphism $\rho:H_1(X^{\textrm{reg}}\setminus D,\ZZ)\to\ZZ/d$.
If $\mu_i$ is a meridian of $D_i$, then $\sigma(\mu_i)\equiv n_i\bmod{d}$. Usually one thinks of $D$
as an effective divisor with $0<n_i<d$ but actually the numbers $n_i$ are only defined modulo $d$
and the divisor does not need to be effective. If the meridians of $D$ generate $H_1(X^{\textrm{reg}}\setminus D,\ZZ)$
no more data is needed, but if not, different covers may have the same ramification divisor. 
To determine the cover one needs another Weil divisor $H$ such that $D\sim dH$; to be more precise, only the 
class of $H$ matters. The cover is the normalization of the zero locus of a multisection of
the fiber bundle $\cO_X(H)$ associated with the isomorphism $\cO_X(H)^{\otimes d}\cong\cO_X(D)$.
The data~$d$ is fixed; the divisor $H$ can be replaced by any other linearly equivalent divisor.
Moreover, given any divisor $A$ we can replace $D$ by $D+dA$ and $H$ by $H+A$. In particular,
the data $(d,D,H)$ can be replaced by $(d,\tilde{D},0)$, where 
$\tilde{D}:=D-dH$. 

Note finally that the number of connected
components of $\tilde{X}$ is the index of $\rho(H_1(X^{\textrm{reg}}\setminus D,\ZZ))$ in $\ZZ/d$.
If $X$ is smooth and simply connected, then this number coincides
with $\gcd(d,n_1,\dots,n_r)$. We will state later what happens in the cyclic quotient case.

It is possible to track algebraically the action of the $1$-cohomology of the monodromy $\sigma:\tilde{X}\to\tilde{X}$
of the covering. In fact, it is possible to get this action on $H^1(X,\cO_X)$ and then derive the action 
on $H^1(X,\CC)$ via the Hodge decomposition.
The theorem below has been proved in \cite{Esnault-Viehweg82} in the smooth case and in \cite[Theorem 2.3]{ACM19}
in the cyclic quotient case (we restrict our attention to~$H^1$).

\begin{theorem}\label{thm:ev}
With the previous notations, if $D$ is a divisor with simple $\QQ$-normal crossings
then
\[
  H^1(\tilde{X},\cO_{\tilde{X}}) = \bigoplus_{\l=0}^{d-1} H^1 (X, \cO_X(L^{(\l)})), \quad
  L^{(\l)} = -\l H + \sum_{i=1}^r \floor{\frac{\l n_i}{d}} D_i,
\]
where the monodromy of the cyclic covering acts on $H^1(X,\cO_X(L^{(\l)}))$ by multiplication by~$\zeta_d^{\l}$.
\end{theorem}

\begin{remark}\label{rem:Dtilde}
The above divisors have nicer expressions if we apply them to $(d,\tilde{D},0)$. If we assume that
\begin{equation}\label{eq:Dtilde}
\tilde{D}=D-dH=\sum_{i=1}^n m_i D_i
\end{equation}
then
\[
  L^{(\l)} = \sum_{i=1}^n \floor{\frac{\l m_i}{d}} D_i.
\]
Note also that the  condition on $\QQ$-normal crossings is applied
only to the reduction of $\tilde{D}$ modulo $d$.

Let $B$ be any Weil divisor in~$X$. Then, since $B\cdot \tilde{D}=0$, we have:
\begin{equation}\label{eq:int_form_Ll}
L^{(\l)}\cdot B=\sum_{i=1}^n \floor{\frac{\l m_i}{d}} D_i\cdot B-\sum_{i=1}^n \frac{\l m_i}{d} D_i\cdot B=
-\sum_{i=1}^n \left\{{\frac{\l m_i}{d}}\right\} D_i\cdot B.
\end{equation}
In particular, if $B$ is effective and it does not have common components with $\tilde{D}$, then $L^{(\l)}\cdot B\leq 0$.
Moreover, $L^{(\l)}\cdot B= 0$ if and only if $\l$ is a multiple of $\frac{d}{\gcd(d,m)}$ where
$m = \gcd\{m_i\mid D_i\cdot B\neq 0\}$.
\end{remark}

If we perform a (weighted) blow-up of $\hat{\pi}:\hat{X}\to{X}$, then we obtain a 
new cyclic cover~$\varpi$ by pulling back
our original cover $\pi$.
If $\pi$ is defined by $(d,\tilde{D},0)$ then $\varpi$ is defined by $(d,\hat{\pi}^*\tilde{D},0)$. Note 
that as $\tilde{D}\sim 0$, it is an integral Cartier divisor and then in particular $\hat{\pi}^*\tilde{D}$
is also an integral Cartier divisor. It is determined by $D$ and the multiplicity 
in~$\hat{\pi}^*\tilde{D}$ of the exceptional divisor of $\hat{\pi}$.

\begin{ex}
Let us suppose that $(X,P)\cong\frac{1}{n}(a,b)$ and we perform a $(p,q)$-weighted blow-up $\hat{\pi}:\hat{X}\to X$; for simplicity,
assume $n,a,b$ are pairwise coprime and $\gcd(p,q)=1$. Let $e:=\gcd(n,pb-qa)$. Assume that $\tilde{D}$
is as in~\eqref{eq:Dtilde}, and let $\nu_i$ be the $(p,q)$-multiplicity of $D_i$ (it vanishes if $P\notin D_i$).
Let $E$ be the exceptional component of $\hat{\pi}$ and let us denote $D_i$ the strict transform of $D_i$ in $\hat{X}$
(the context will indicate to which divisor we are referring to). Then,
\[
\hat{\pi}^*\tilde{D}=\frac{1}{e}\left(\sum_{i=1}^n \nu_i m_i\right)E+\sum_{i=1}^n m_i D_i.
\]
If $\tilde{D}$ is a simple $\QQ$-normal crossing divisor modulo $d$,
then it is also the case for $\hat{\pi}^*\tilde{D}$.
\end{ex}

\begin{defi}
Let $(W,P)$ be a germ of type $\frac{1}{n}(a,b)$, where 
$n,a,b$ are pairwise coprime. This space admits two special expressions,
namely $\frac{1}{n}(1,b')$ or $\frac{1}{n}(a',1)$, where 
$a',b'$ are well defined modulo $n$. A 
weighted blow-up of $W$ is said to be \emph{special}
if it is of weight $(a',1)$ or~$(1,b')$.
\end{defi}

This special blow-ups have nice properties. The curvettes of
the exceptional components are the extremal smooth curve germs of $W$.
For example, if $a',b'$ are reduced modulo $n$, then they are the first step
of the Jung-Hirzebruch resolutions.

\begin{lemma}\label{lema:quot_local}
Let $(W,P)$ be a germ of type $\frac{1}{n}(a',1)\cong\frac{1}{n}(1,b')$, 
$\gcd(a',n)=\gcd(b',n)=1$. Let $\pi_x,\pi_y$ be the special
weighted blow-ups of weights $(a',1)$, $(1,b')$, respectively, with 
exceptional components $E_x,E_y$. We denote by $\mu_x,\mu_y,\mu_x^e,\mu_y^e$
suitable meridians of $\{x=0\}$, $\{y=0\}$, $E_x$, $E_y$, in 
$W\setminus\{xy=0\}$. 

\begin{enumerate}[label=\rm\alph{enumi})]
\item The local fundamental group of $W\setminus\{P\}$
is cyclic of order~$n$; $\mu_x^e$ and $\mu_y^e$ are (separately)
generators of this group.

\item  The local fundamental group of $W\setminus\{y=0\}$
is~$\ZZ$. The sets $\{\mu_x^e\}$ and $\{\mu_y^e,\mu_y\}$ generate
this group (separately). A similar statement holds
for $W\setminus\{x=0\}$.

\item The local fundamental group of $W\setminus\{xy=0\}$
is~$\ZZ^2$.  The sets $\{\mu_x^e,\mu_x\}$ and $\{\mu_y^e,\mu_y\}$ generate
this group (separately).

\end{enumerate}
\end{lemma}

\begin{proof}
Iterating the special blow-ups one obtains the Jung-Hirzebruch resolution of $W$.
Using Mumford's method~\cite{Mumford-topology} a presentation of the distinct fundamental
groups (generated by the meridians of all divisors) is given 
and the result follows.
\end{proof}

Note that we are not asking the coverings $\pi,\varpi$ above to be connected. In the classical case ($X$ smooth and 
simply connected) it is easy to relate the number of connected components and the arithmetic of the coefficients of~$D$. 
If we drop the smoothness, more conditions are needed.

\begin{prop}\label{prop:numb_conn_comp}
Let $X$ be a simply connected projective surface with normal 
cyclic quotient
singularities. Let $\pi:\tilde{X}\to X$ be a cyclic branched cover
associated with $(d,D,H)$, where $D$ has \emph{smooth} components
and $\QQ$-normal crossings modulo $d$. Let $\hat{\sigma}:\hat{X}\to X$
be the composition of \emph{one} special blow-up for each 
singular point of $X$.

Let $m$ be the greatest common divisor of $d$ and the coefficients of the divisor
$\hat\sigma^*(D-dH)$. Then $\tilde{X}$ has $m$ connected components.
\end{prop}

\begin{proof}
Let $\sigma_Y:Y\to X$ be the minimal resolution of the singularities of $X$.
Let 
\[
D_Y:=d\left\{\frac{\sigma_Y^*(D-dH)}{d}\right\},\qquad 
H_Y:=-\left\lfloor\frac{\sigma_Y^*(D-dH)}{d}\right\rfloor.
\]
Let us denote also 
\[
\hat{D}:=d\left\{\frac{\hat{\sigma}^*(D-dH)}{d}\right\},\qquad 
\hat{H}:=-\left\lfloor\frac{\hat{\sigma}^*(D-dH)}{d}\right\rfloor.
\]
Let $E$ be an exceptional component of $\hat{\sigma}$; its multiplicity
in $\hat{D}$ coincides with the multiplicity of its strict
transform in $D_Y$. The same applies for the irreducible components of~$D$.

Since $X$ is simply connected, it is the case for $Y$.
Then $H_1(Y\setminus D_Y,\ZZ)$ is generated by the meridians of the irreducible components
of $D_Y$. Let $C$ be an irreducible component of $D_Y$ and let $\mu_C$
the class of its meridians. Let $\rho:H_1(Y\setminus D_Y,\ZZ)\to\ZZ/d$
be the morphism determining the covering over $Y$ (which has the same number~$m$
of connected components as $\tilde{X}$). Recall that $\rho(\mu_C)$ is the
coefficient of $C$ in $D_Y$ ($\bmod{\,d}$).
Then, $m$ equals the greatest common divisor of $d$ and the coefficients of the divisor
$D_Y$.

This comes from the fact that the whole set of meridians generate $H_1(Y\setminus D_Y,\ZZ)$. But Lemma~\ref{lema:quot_local} implies
that only the strict transforms of the irreducible components of $\hat{D}$ suffice and the result follows.
\end{proof}

Let $\pi$ be a cyclic cover of a surface $X$ with cyclic quotient
singular points, associated with $(d,D,H)$ where $D$ is a simple
$\QQ$-normal crossing divisor.
Let $C\subset X$ be an irreducible curve (with only unibranch points) such that the union  of $C$ and the support of $D$  is a $\QQ$-normal crossing divisor. Then, $\pi_C:=\pi_{|}:\pi^{-1}(C)\to C$ is a (maybe non connected) cyclic cover
of the curve~$C$. In the smooth case it is easy to obtain the divisors defining this cover. In the cyclic quotient case
some work has to be done.

Let $P_1,\dots,P_s\in C\cap\sing X$ and let $\hat\sigma:\hat{X}\to X$ be a composition of weighted blow-ups
at $P_1,\dots, P_s$ such that the strict transform of $C$ (still
denoted by $C$) is contained in the regular part of $\hat{X}$,
which exists because of the $\QQ$-normal crossing condition.
Let $\hat\pi$ be the pull-back of $\pi$ by $\hat\sigma$.
Note that $\pi_C$ and $\hat{\pi}_C$ can be identified. 

The covering $\hat{\pi}$ is associated with a triple $(d,\hat{D},\hat{H})$
obtained as follows. Consider $\tilde{D}=D-d H$;
then $\hat\sigma^*\tilde{D}=\hat{D}-d\hat{H}$ where the support 
of $\hat{D}$ is contained in the support of the $\mathbb{\QQ}$-divisor
$\pi^*(D)$.

\begin{prop}
\label{prop:restriccion}
Let $C\subset X$ be an irreducible curve (with only unibranch points) such that the union of the support of $D$ with~$C$ is $\QQ$-normal crossings. 

Then, the divisor $D_C$ of $C$ defining $\pi_C$ has support at $D\cap C\equiv\hat{D}\cap C$,
and the multiplicity of each point $P\in\hat{D}\cap C$ is the coefficient
of the irreducible component $D_i$ of $\hat{D}$ containing~$P$.
\end{prop}

\begin{proof}
We add some blow-ups of $X$ such that $\hat{\sigma}$ is as in Proposition~\ref{prop:numb_conn_comp}. The extra blow-ups do not affect
to the multiplicities of the components intersecting~$C$ and then
the result follows.
\end{proof}

Note that, since the $\cl(\PP^1)$ is completely determined by the degree,
we deduce that $d$ divides $\deg D_C$ and $H_C$ is a divisor of 
degree~$\frac{\deg D_C}{d}$.

\subsection{Application to covers of \texorpdfstring{$\PP^1$}{P1}}
\label{subsec:coverP1}
\mbox{} 

We are going to study the characteristic polynomial and the $H^1$-eigenspace decomposition 
of a $d$-cyclic covering of $\PP^1$ associated with a divisor
\[
D=\sum_{j=1}^s m_j\langle p_j\rangle,\qquad \sum_{j=1}^s m_j=d h,\qquad h\in\ZZ,
\]
i.e.~$\deg H=h$; no coprimality conditions are assumed. Then we have
\[
L^{(\l)}\sim -\l H+\sum_{j=1}^s \left\lfloor\frac{\l m_j}{d}\right\rfloor\langle p_j\rangle
\Longrightarrow 
\deg L^{(\l)}=-\sum_{j=1}^s \left\{\frac{\l m_j}{d}\right\}.
\]
Note that $L^{(\l)}$ depends only on $\l\bmod{d}$ and 
$L^{(0)}=0$. Let $n:=\gcd(d,m_1,\dots,m_r)$ and $\hat{d}=\frac{d}{n}$,
$\hat{m}_j:=\frac{m_j}{n}$. Then 
\[
\deg L^{(\l)}=-\sum_{j=1}^s \left\{\frac{\l \hat{m}_j}{\hat{d}}\right\}
\]
and actually $L^{(\l)}$ depends only on $\l\bmod{\hat{d}}$. 
We conclude that for $\l\in\{0,1,\dots,d-1\}$:
\begin{equation}
\label{eq:h1-deg}
h^1(\PP^1,\cO_{\PP^1}(L^{(\l)}))=
\begin{cases}
0&\text{if }\l\equiv 0\bmod{\hat{d}},\\
\displaystyle -1+\sum_{j=1}^s\left\{\frac{\l \hat{m}_j}{\hat{d}}\right\}&\text{otherwise.}
\end{cases}
\end{equation}
Geometrically the $d$-cyclic cover of $\PP^1$ associated
with $D$ is the disjoint union of $n$ copies of a $\hat{d}$-cyclic cover
of $\PP^1$ associated with $\frac{1}{n}D\in\Div(\PP^1)$, where the monodromy exchanges
cyclically these copies. Applying %
Lemma~\ref{covering_proj_line}
the characteristic polynomial of the monodromy is
\begin{equation}\label{eq:acampo}
\Delta(t)=\frac{(t^n-1)^2(t^d-1)^{s-2}}{\displaystyle\prod_{j=1}^s(t^{\gcd(d,m_j)}-1)}.
\end{equation}
Note that we have obtained more than that since we have the monodromy
action on the Hodge structure of the covering space.

\subsection{Normal-crossing covers of \texorpdfstring{$\PP^1\times\PP^1$}{P1xP1}}
\label{sec:p1p1}
\mbox{}

We illustrate the coverings ramified along normal-crossing divisors on surfaces, studying $\PP^1\times\PP^1$.
Let $\pi:\tilde X\to\PP^1\times\PP^1$ be a cover associated to $(d,D,H)$ where $D$ is a normal crossing divisor. 
The condition $D\sim dH$ completely determines $H$ using bidegrees (this is the first difference with 
\emph{reducible normal fake quadrics}, see Definition~\ref{def:rnfq},
since their class group may not be torsion-free).

The cohomology of $\tilde X$ can be studied using Theorem~\ref{thm:ev} (or more precisely the original Theorem of Esnault-Viehweg)
for which we need to know the $1$-cohomology of some divisors. The following result is well known.

\begin{prop}
Let $S$ be a section and let $F$ be a fiber of $\PP^1\times\PP^1$. Then:
\begin{enumerate}[label=\rm(\arabic{enumi})]
\item\label{h0-P1P1} 
$\dim H^0(\PP^1\times\PP^1;\cO(a S+b F))=
\begin{cases}
(a+1)(b+1) & \textrm{ if } a,b\geq 0\\
0 & \textrm{ otherwise. }
\end{cases}
$
\item\label{h2-P1P1} 
$\dim H^1(\PP^1\times\PP^1;\cO(a S+b F))=
\begin{cases}
(a+1)(b+1) & \textrm{ if } a,b\leq -2\\
0 & \textrm{ otherwise. }
\end{cases}
$
\item\label{h1-P1P1} 
$\dim H^2(\PP^1\times\PP^1;\cO(a S+b F))=
\begin{cases}
-(a+1)(b+1) & \textrm{ if } (a+2)(b+2)<0\\
0 & \textrm{ otherwise. }
\end{cases}
$
\end{enumerate}
\end{prop}

\begin{figure}[ht]
\centering
﻿\begin{tikzpicture}[scale=.9]
\fill[color=red!10!white] (0,3.5)--(0,0)--(3.5,0)--(3.5,3.5)--cycle;
\fill[color=green!10!white] (-3.5,-2)--(-2,-2)--(-2,-3.5)--(-3.5,-3.5)--cycle;
\fill[color=blue!10!white] (0,-3.5)--(0,-2)--(3.5,-2)--(3.5,-3.5)--cycle;
\fill[color=blue!10!white] (-3.5,3.5)--(-3.5,0)--(-2,0)--(-2,3.5)--cycle;
\foreach \a in {-3,-2,...,3}
{
\foreach \b in {-3,-2,...,3}
{
\fill[black] (\a,\b) circle [radius=.1cm];
}
}
\draw[->] (-4,0)--(4,0) node[right] {$b=0$};
\draw[->]  (0,-4)--(0,4) node[above] {$a=0$};
\draw[line width=2,color=red] (-3.5,-1)--(3.5,-1) node[right,color=black] {$H^0=H^1=H^2=0$};
\draw[line width=2,color=red] (-1,-3.5)--(-1,3.5);
\node at (1.5,1.5) {$H^0$};
\node at (-2.5,-2.5) {$H^2$};
\node at (1.5,-2.5) {$H^1$};
\node at (-2.5,1.5) {$H^1$};
\end{tikzpicture}
\caption{Map of the cohomology of $H^*(\PP^1\times\PP^1;\cO(a S+b F))$, where 
$S$ is a section and $F$ is a fiber.}
\end{figure}
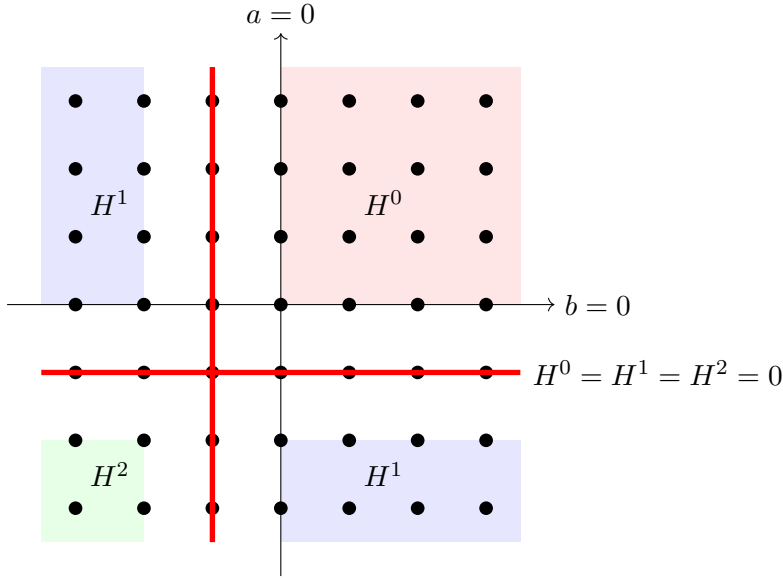

This result is a direct consequence of the following:
\begin{enumerate}[label=\rm($\mathcal{H}$\arabic{enumi})]
\item The space $\dim H^0(\PP^1\times\PP^1;\cO(a S+b F))$
is isomorphic to the space of polynomials of bidegree $(a,b)$ and~\ref{h0-P1P1} follows.

\item Using Serre duality~\ref{h2-P1P1} follows.

\item Using Riemann-Roch and combining the previous results,~\ref{h1-P1P1} follows.
\end{enumerate}

Actually we are only interested in $H^1$ for $a,b\leq 0$ and only for $(a,0),(0,a)$, $a<0$, the contribution is positive.
Let us decompose $D=\mcdhv_v D_v+\mcdhv_h D_h+\mcdhv_m D_m$ where $D_v,D_h,D_m$ are primitive (their multiplicities are coprime), and 
all the components of $D_v$ have bidegree of type $(0,b)$, all the components of $D_h$ have bidegree of type $(a,0)$,
and all the components of $D_h$ have bidegree of type $(a,b)$, $a,b>0$. The following result is well known.

\begin{theorem}
The space $H^1(\tilde X;\cO_{\tilde X})$ is decomposed as a direct sum 
$H^1(\tilde X_v;\cO_{\tilde X_v})\oplus H^1(\tilde X_h;\cO_{\tilde X_h})$ where 
$\pi_v:\tilde X_v\to F$ is the restriction to $F\cong \PP^1$ of the intermediate cover of degree~$\gcd(d,\mcdhv_h,\mcdhv_m)$ and
$\pi_h:\tilde X_h\to S$ is the restriction to $S\cong \PP^1$ of the intermediate cover of degree~$\gcd(d,\mcdhv_v,\mcdhv_m)$.
\end{theorem}

\section{Reducible normal fake quadrics}\label{sec:settings}

\subsection{A ramified covering of the projective line}\label{sec:cover-P1} \mbox{}

For the sequel we need to define an orbifold $\orb$, with the following data. 
The orbifold is supported by $\PP^1=\CC\cup\{\infty\}$ and the orbifold points are $\gamma_1,\dots,\gamma_r\in\CC\subset\PP^1$, $r\geq 0$,
of orders $d_1,\dots,d_r\in\ZZ_{>1}$. Let us consider also $q_i\in%
\{1,\dots,d_i-1\}$, 
$\gcd(d_i,q_i)=1$, $i=1,\dots,r$
such that
\begin{equation}\label{eq:alpha}
\alpha=\sum_{i=1}^r\frac{q_i}{d_i}\in\ZZ.
\end{equation}
This imposes strong conditions on $d_1,\dots,d_r$. For instance,
\begin{equation}\label{eq:lmc-r-1}
d_r\text{ divides } \lcm(d_1,\dots,d_{r-1}).
\end{equation}
In particular if $\vdeg:=\lcm(d_1,\dots,d_r)$, then $\vdeg^2$~divides $d_1\cdot\ldots\cdot d_r$ (see Remark~\ref{rem:prop:class} in page~\pageref{rem:prop:class}). 

There is an orbifold $\vdeg$-cyclic covering $\tau:G\to\orb$ associated to the epimorphism (well-defined from \eqref{eq:alpha})
\[
\begin{tikzcd}[row sep=0pt,/tikz/column 1/.append style={anchor=base east},/tikz/column 2/.append style={anchor=base west}]
\pi_1^{\textrm{orb}}(\orb)=\langle\mu_1,\dots,\mu_r\mid \mu_1\cdot\ldots\cdot\mu_r=1, \mu_1^{d_1}=\dots=\mu_r^{d_r}=1\rangle\ar[r]&\ZZ/\vdeg, \\
\mu_i\ar[r,mapsto]&q_i\dfrac{\vdeg}{d_i} \bmod \vdeg;
\end{tikzcd}
\]
the covering $\tau$ has been introduced in the Introduction.
The position of the orbifold points has an influence
on the analytic type of $G$ but not on its topological type. The following result is a direct
consequence of the definition of a cover associated with an epimorphism onto a cyclic group.

\begin{lemma}
There exists a unique generator $\eta:G\to G$ of the monodromy of $\tau$ such that for any 
$i\in\{1,\dots,r\}$ and $p\in\tau^{-1}(\gamma_i)$, there exists 
a local coordinate $y$ of $G$ centered at $p$ such that 
\[
\eta^{\frac{\vdeg}{d_i}}(y)=\zeta_{d_i}^{q_i}y.
\]
\end{lemma}

Using Riemann-Hurwitz method, one can compute the genus of $G$:
\[
2-2g(G)=\chi(G)=\vdeg(2-r)+\sum_{i=1}^r \frac{\vdeg}{d_i}=\vdeg\chi^{\textrm{orb}}
\]
where
\begin{equation}\label{eq:beta}
\chi^{\textrm{orb}}:=\chi^{\textrm{orb}}(\orb)=2-\sum_{i=1}^r\left(1-\frac{1}{d_i}\right) \in \ZZ \frac{1}{\vdeg}. 
\end{equation}

\begin{lemma}
With the previous notations, if $r>2$, then $g(G)>0$ and $\chi^{\textrm{orb}}\leq 0$.
\end{lemma}

\begin{proof}
It is enough to prove that $\chi^{\textrm{orb}}\leq 0$. Since $r>2$ and
\[
\chi^{\textrm{orb}}\leq \frac{r}{2}+2-r=\frac{4-r}{2},
\]
it is enough to rule out the case $r=3$. 

In that case $\chi^{\textrm{orb}}$ can be positive only if 
$(d_1,d_2,d_3)$ is one of the following $(2,2,n)$, $(2,3,3)$, $(2,3,4)$, or $(2,3,5)$,
but none of them satisfy~\eqref{eq:alpha}.
\end{proof}

\subsection{Definition and description of a reducible normal fake quadric} \mbox{} \label{subsec:S}

Recall the classical definition of a \emph{fake quadric} as a smooth projective surface with the same rational 
cohomology as, but not biholomorphic to, the quadric surface $\PP^1\times\PP^1$. Examples of fake quadrics are 
quotients of a product of two smooth projective curves by a free diagonal action, such surfaces are called 
\emph{reducible} (also called isogenous to a product of curves of an \emph{unmixed} type, 
see~\cite{Cat-fibred,Bauer-Cat-Grunewald-classification}). 
We extend this concept to \emph{reducible normal fake quadrics}, where the freeness condition is dropped.

Let us consider the following $(\ZZ/\vdeg)^2$-Galois  and $\ZZ/\vdeg$-Galois covers:
\[
\begin{tikzcd}[row sep=1,/tikz/column 1/.append style={anchor=base east},/tikz/column 2/.append style={anchor=base west}]
&S:=(G\times\PP^1)/(\ZZ/\vdeg)\ar[rd]&\\[20pt]
G\times\PP^1\ar[rr]\ar[ru,"\tau_2"]&&\PP^1\times\PP^1\\
(p,z)\ar[rr,mapsto]&&(\tau(p),z^\vdeg).
\end{tikzcd}
\]
The action defining~$S$ is given by a diagonal action
\begin{equation}\label{eq:diag}
(1\bmod \vdeg,(p,z))\longmapsto(\eta(p),\zeta_\vdeg^{-1}z).
\end{equation}

\begin{defi}\label{def:rnfq}
The surface $S:=(G\times\PP^1)/(\ZZ/\vdeg)$ with the action~\eqref{eq:diag} is called the
\emph{reducible normal fake quadric} associated with~$\tau$.
\end{defi}

\begin{figure}[ht]
﻿\begin{tikzpicture}
\draw (-2,-1) -- (2,-1) node[right] {$C$} ;
\draw (-2,1) -- (2,1)  node[right] {$E$};
\draw (-1,-1.5) -- (-1,1.5)  node[left] {$F_1$};
\draw (1,-1.5) -- (1,1.5)  node[right] {$F_r$};
\node at (0,0) {$\dots$};
\node at (2,0) {$\mathbb{P}^1\times\mathbb{P}^1$};

\draw[->] (-2.5,3)--(0,1.5);
\draw[->] (-6,1.5)--(-3.5,3) node[above,pos=.5] {$\tau_2$};
\draw[->] (-4.5,0)-- (-1.5,0);

\begin{scope}[xshift=-6cm]
\draw (-2,-1) -- (2,-1) node[right] {$C$};
\draw (-2,1) -- (2,1)  node[right] {$E$};
\draw (-1,-1.5) -- (-1,1.5)  node[left] {$F_1$};
\draw (1,-1.5) -- (1,1.5)  node[right] {$F_r$};
\node at (0,0) {$\dots$};
\node at (-2,0) {$G\times\mathbb{P}^1$};
\end{scope}
\begin{scope}[xshift=-3cm, yshift=5cm]
\draw (-2,-1) -- (2,-1) node[right] {$C$} node[left,pos=0] {$0$};
\draw (-2,1) -- (2,1)  node[below right] {$E$}node[below left,pos=0] {$0$};
\draw (-1,-1.5) -- (-1,1.5)  node[right, pos=.5] {$A_1$};
\draw (1,-1.5) -- (1,1.5)  node[left,pos=.5] {$A_r$};
\node at (0,0) {$\dots$};
\node at (-2,0) {$S$};
\fill (-1,-1) circle [radius=.1cm] ;
\fill (1,-1) circle [radius=.1cm] ;
\fill (-1,1) circle [radius=.1cm] ;
\fill (1,1) circle [radius=.1cm] ;
\node[above left] at (-1,1) {$\frac{1}{d_1}(1,q_1)$};
\node[above right] at (1,1) {$\frac{1}{d_r}(1,q_r)$};
\node[below left] at (-1,-1) {$\frac{1}{d_1}(1,-q_1)$};
\node[below right] at (1,-1) {$\frac{1}{d_r}(1,-q_r)$};
\end{scope}
\end{tikzpicture}
 \caption{Covering construction of $S$.}
\label{fig:S2}
\end{figure}
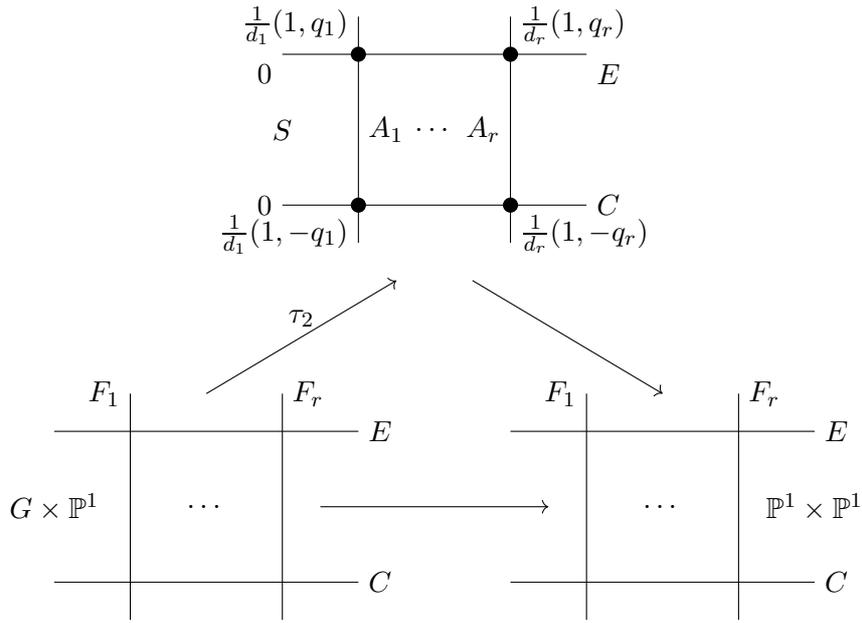

The reason for the name \emph{reducible normal fake quadric} will be cleared from the following description.

\begin{lemma}
\label{lemma:S}
Let $S$ be the reducible normal fake quadric associated with~$\tau$.
It is a normal ruled surface $\pi_S:S\to\PP^1$, $[(p,z)]\mapsto \tau(p)$ with two sets 
$\sing(C):=\{P_1,\dots,P_r\}$ and $\sing(E):=\{Q_1,\dots,Q_r\}$
of singular points (of cyclic quotient type). The following holds:
\begin{enumerate}[label=\rm(S\arabic{enumi})]
\item\label{lemma:S:1} The curves $C:=\tau_2(G\times\{0\})$ and $E:=\tau_2(G\times\{\infty\})$ are \emph{sections}
of $\pi_S$ with self-intersection~$0$.
\item\label{lemma:S:2} There are \emph{fibers} $A_i$ of $\pi_S$ such that $\{P_i\}=A_i\cap C$ and $\{Q_i\}=A_i\cap E$.
\item\label{lemma:S:3} The type of $P_i$ is $\frac{1}{d_i}(1,-q_i)$ and the type of $Q_i$ is $\frac{1}{d_i}(1,q_i)$.
\end{enumerate}
\end{lemma}

\begin{proof}
Figure~\ref{fig:S2} describes $S$ as a middle cover. We can identify $C$ and $E$ with
$\CC\cup\{\infty\}\equiv\PP^1$ where $P_i$ and $Q_i$ become $\gamma_i$. Let $p\in G$ such that
$P_i=[(p,0)]$. A neighborhood of $P_i$ in $S$ is isomorphic to a neighborhood of
the origin in $\CC^2/\mu_{d_i}$, where $\mu_{d_i}=\langle \zeta\rangle$ is the cyclic group of $d_i$-roots
of unity in $\CC^*$ and the action is defined by
\[
\zeta\cdot(y,z)=(\zeta^{q_i}y,\zeta^{-1} z),
\]
and thus the type of $P_i$ as a quotient singular point is calculated. Since $z^{-1}$ is a
local coordinate at $\infty$, the type of $Q_i$ is computed in the same way.
The self-intersection computation is straightforward.
\end{proof}

\begin{remark}\label{rem:dif_dq}
The surface $S$ does not determine the original data given by $(d_1,\dots,d_r)$ and $(q_1,\dots,q_r)$.
For instance, interchanging $0$ and $\infty$ in $\PP^1$ and choosing $\eta^{-1}$ as a 
generator of the monodromy of $\tau$ results in the same surface~$S$, which is associated with 
the data $(d_1,\dots,d_r)$, $q_i':=d_i-q_i$, and $\alpha':=r-\alpha$. 
However, note that in general, replacing $\eta$ by $\eta^\ell$ with $\gcd(\ell,\vdeg)=1$ does not result 
in an isomorphic surface to~$S$.
\end{remark}

\subsection{An alternative construction}
\label{sec:alternative}
\mbox{}

There is an alternative description of reducible normal fake quadrics in terms of generalized Nagata
operations of ruled surfaces. 
Consider a reducible normal fake quadric $S$ as above and $(d_i,q_i)$, $i=1,...,r$ such that 
$\alpha:=\sum_{i=1}^r \frac{q_i}{d_i}\in\ZZ$ as in Lemma~\ref{lemma:S}.

\begin{lemma}
\label{lemma:S:Nagata}
The surface $S$ and the smooth ruled surface $\Sigma_\alpha$ both have a common weighted blown-up
space obtained as follows.
\begin{itemize}
\item From $S$: composition of $(1,q_i)$-weighted blow-ups of $Q_i$.
\item From $\Sigma_\alpha$: composition of $(d_i,q_i)$-weighted blow-ups at points in a section with
self-intersection~$\alpha$.
\end{itemize}
\end{lemma}

\begin{proof}
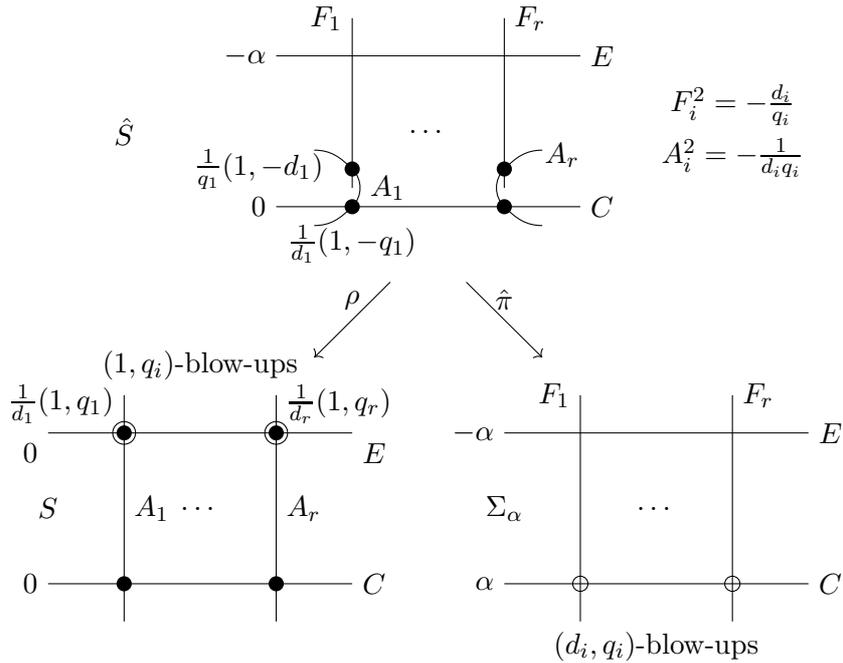
\begin{figure}[ht]
﻿\begin{tikzpicture}
\draw (-2,-1) -- (2,-1) node[right] {$C$} node[left,pos=0] {$\alpha$};
\draw (-2,1) -- (2,1)  node[right] {$E$}node[left,pos=0] {$-\alpha$};
\draw (-1,-1.5) -- (-1,1.5)  node[left] {$F_1$};
\draw (1,-1.5) -- (1,1.5)  node[right] {$F_r$};
\node at (0,0) {$\dots$};
\node at (-2,0) {$\Sigma_\alpha$};
\draw (-1,-1) circle [radius=.1cm] ;
\node[below] at (0,-1.5) {$(d_i,q_i)$-blow-ups};
\draw (1,-1) circle [radius=.1cm] ;

\draw[->] (-2.5,3)--(-1.5,2) node[above,pos=.5] {$\hat{\pi}$};
\draw[->] (-3.5,3)--(-4.5,2) node[above,pos=.5] {$\rho$};

\begin{scope}[xshift=-3cm, yshift=5cm]
\draw (-2,-1) -- (2,-1) node[right] {$C$} node[left,pos=0] {$0$};
\draw (-2,1) -- (2,1)  node[right] {$E$}node[left,pos=0] {$-\alpha$};
\draw (-1,-.75) -- (-1,1.5)  node[left] {$F_1$};
\draw (1,-.75) -- (1,1.5)  node[right] {$F_r$};
\node at (0,0) {$\dots$};
\draw (-1.5,-1.25) to[out=0,in=-135] (-1,-1) to[out=45,in=-45] node[right] {$A_1$}(-1,-.5) to[out=135,in=0] (-1.5,.-.25) ;
\fill (-1,-1) circle [radius=.1cm] node[below=3pt] {$\frac{1}{d_1}(1,-q_1)$} ;
\fill (-1,-.5) circle [radius=.1cm] node[left=7pt] {$\frac{1}{q_1}(1,-d_1)$} ;

\draw (1.5,-1.25) to[out=180,in=-45] (1,-1) to[out=135,in=-135] (1,-.5) to[out=45,in=180] node[right=5] {$A_r$}(1.5,.-.25);
\fill (1,-1) circle [radius=.1cm] ;
\fill (1,-.5) circle [radius=.1cm] ;
\node at (4,-.35) {$A_i^2=-\frac{1}{d_i q_i}$};
\node at (4,0.35) {$F_i^2=-\frac{d_i}{q_i}$};
\node at (-4,0) {$\hat{S}$};
\end{scope}
\begin{scope}[xshift=-6cm]
\draw (-2,-1) -- (2,-1) node[right] {$C$} node[left,pos=0] {$0$};
\draw (-2,1) -- (2,1)  node[below right] {$E$}node[below left,pos=0] {$0$};
\draw (-1,-1.5) -- (-1,1.5)  node[right, pos=.5] {$A_1$};
\draw (1,-1.5) -- (1,1.5)  node[right,pos=.5] {$A_r$};
\node at (0,0) {$\dots$};
\node at (-2,0) {$S$};
\fill (-1,-1) circle [radius=.1cm] ;
\fill (1,-1) circle [radius=.1cm] ;
\fill (-1,1) circle [radius=.1cm] ;
\draw (-1,1) circle [radius=.15cm] ;
\node[above=2pt] at (0,1.5) {$(1,q_i)$-blow-ups};
\fill (1,1) circle [radius=.1cm] ;
\draw (1,1) circle [radius=.15cm] ;
\node[above left] at (-1,1) {$\frac{1}{d_1}(1,q_1)$};
\node[above right] at (1,1) {$\frac{1}{d_r}(1,q_r)$};
\end{scope}
\end{tikzpicture}
 \caption{Blow-up construction of $S$.}
\label{fig:S1}
\end{figure}
The proof is depicted in Figure~\ref{fig:S1}.
Let us start from $S$. We perform the composition
of the $(1,q_i)$-weighted blow-ups at $Q_i$, $i=1,\dots,r$; if $(u_i,v_i)$
are the local variables, then $E$ (resp.~$A_i$) is given by $u_i=0$ (resp.~$v_i=0$). Let $F_1,\dots,F_r$ 
be the exceptional components. From \cite[Theorem~4.3]{AMO-Intersection}
we obtain that $(F_i^2)_{\hat{S}}=-\frac{d_i}{q_i}$. For the strict
transforms we have
\[
(A_i^2)_{\hat{S}}=(A_i^2)_{S}-\frac{1^2}{d_i q_i}=-\frac{1}{d_i q_i},\quad 
(E^2)_{\hat{S}}=(E^2)_{S}-\sum_{i=1}^r\frac{q_i^2}{q_i d_i}=-\sum_{i=1}^r\frac{q_i^2}{q_i d_i}=-\alpha.
\]
Since the centers of the blowing-ups are disjoint to $C$, we still have 
$(C^2)_{\hat{S}}=0$. Moreover, the surface $\hat{S}$ is smooth
along $E$, and $\hat{S}$ has cyclic quotient singular points of type $\frac{1}{q_i}(1,-d_i)$ at $F_i\cap A_i$.

Note that:
\begin{equation}
\begin{split}
\hat{\pi}^*(F_i)=F_i+ d_i A_i,\qquad
\hat{\pi}^*(C)=C+\sum_{i=1}^r {q_i}A_i,\\
\rho^*(A_i)=A_i+\frac{1}{d_i}F_i,\qquad
\rho^*(E)=E+\sum_{i=1}^r \frac{q_i}{d_i}F_i.
\end{split}
\end{equation}
The surface $\hat{S}$ along $A_i$ looks like the exceptional component
of a weighted blow-up of type $(d_i,q_i)$ at a smooth point. This shows
that the stated weighted blowing-ups of $\Sigma_\alpha$
also yield $\hat{S}$. 
\end{proof}

\begin{remdef}
\label{rem:def:S}
As a consequence of Lemmas~\ref{lemma:S} and~\ref{lemma:S:Nagata}, associated with any $(d_i,q_i)$, $i=1,...,r$ 
such that $\gcd(d_i,q_i)=1$ and $\alpha:=\sum_{i=1}^r \frac{q_i}{d_i}\in\ZZ$, there is a reducible normal fake quadric,
say $S$. According to Remark~\ref{rem:dif_dq}, this correspondence is not one to one, but one can still refer to $S$ as 
\emph{the reducible normal fake quadric associated with $(d_i,q_i)$, $i=1,...,r$}.
\end{remdef}

Summarizing, the following describes the numerical relation between the more relevant 
divisors on~$S$:
\begin{equation}
\label{eq:intersections}
A_i^2=E^2=C^2=F^2=0, \quad F\cdot C=F\cdot E=1, \quad C\cdot A_i = E \cdot A_i =\frac{1}{d_i},
\end{equation}
where $F$ is a generic fiber of $\pi_S$.

\begin{remark}\label{rem:fibr2}
The surface $S$ admits another map $\pi_G:S\to\PP^1$ corresponding to the map $[(p,z)]\mapsto z^\vdeg$. 
The fiber corresponding to $0$ is $\vdeg C$, the fiber corresponding to $\infty$ is $\vdeg E$
and the other fibers are isomorphic to $G$ (and they will again be denoted by $G$).
These curves admit a simple characterization.
\end{remark}

\begin{lemma}\label{lem:hv}
Let $D\subset S$ be an irreducible curve.
\begin{enumerate}[label=\rm(\roman{enumi})]
\item\label{lem:v} If $D\cdot F=0$, then $D$ is equal to one of these curves:
$A_1,\dots,A_r$ or a generic fiber~$F$ of $\pi_S$. A linear combination
of such divisors will be called a \emph{vertical divisor}.
\item\label{lem:h} If $D\cdot C=0$, then $D$ is equal to one of these curves:
$C,E$ or a curve~$G$. A linear combination
of such divisors will be called a \emph{horizontal divisor}.
\end{enumerate}
A linear combination of irreducible divisors which are neither vertical
nor horizontal will be called a \emph{slanted divisor}.
\end{lemma}

\begin{proof}
Let us start with \ref{lem:v}. Let $p\in D$ and $w:=\pi_S(p)\in\PP^1$.
Note that $\pi_S^{-1}(w)$ is either $A_1,\dots,A_r$ or a fiber~$F$; then
$D\cdot \pi_S^{-1}(w)=0$ and $D\cap\pi_S^{-1}(w)\neq\emptyset$.
Since $D$ and $\pi_S^{-1}(w)$ are irreducible the only option is $D=\pi_S^{-1}(w)$.
For \ref{lem:h} we follow the same ideas using the map~$\pi_G$.
\end{proof}

\subsection{Weil divisor class group} \mbox{}

Consider $S$ the reducible normal fake quadric associated with $(d_i,q_i)$, $i=1,...,r$ 
such that $\gcd(d_i,q_i)=1$ and $\alpha:=\sum_{i=1}^r \frac{q_i}{d_i}\in\ZZ$ as defined in Remark-Definition~\ref{rem:def:S}.
The descriptions of $S$ given in \S\ref{sec:settings} together with \cite[\S2.3]{campillo} are the main ingredients for the computation of the Weil divisor class group~$\cl(S)$. 
Despite $\Sigma_\alpha$ having a simple class group isomorphic to $\ZZ^2$, note the following 
description of this group in terms of the divisor classes involved in the construction of $\Sigma_\alpha$:
\begin{equation}\label{eq:class_alpha}
\cl(\Sigma_\alpha)=\langle C,E,F,F_1,\dots,F_r\mid E\sim C-\alpha F,\  F\sim F_1\sim\dots\sim F_r\rangle,
\end{equation}
where $F$ is a generic fiber. In order to obtain $\cl(\hat{S})$, see Figure~\ref{fig:S1}, the previous generators need to be 
replaced by their strict transforms, the classes of the exceptional components added, and the linear equivalence relations rewritten 
in terms of the new generators,
see~\cite[Proposition~2.10]{campillo},
\begin{equation}\label{eq:class_alpha_blow}
\cl(\hat{S})=\langle C,E,F,F_1,\dots,F_r,A_1,\dots,A_r\mid E\sim C+\sum_{i=1}^r q_i A_i-\alpha F,\  F\sim F_i+d_i A_i\rangle.
\end{equation}
As in~\cite[Proposition~2.12]{campillo}, a presentation of the class group for a blowing-down can easily be obtained
if the exceptional components are part of the presentation of the class group of the source, and hence presentation~\eqref{eq:class_alpha_blow} 
comes in handy. In this situation, it is enough to ``forget'' those exceptional components, that is,
\begin{equation}\label{eq:clS}
\cl(S)=\langle C,E,F,A_1,\dots,A_r\mid E\sim C+\sum_{i=1}^r q_i A_i-\alpha F,\  F\sim d_i A_i\rangle.
\end{equation}

\begin{prop}\label{prop:class}
The class group $\cl(S)$ has the following structure as an abelian group.

\begin{equation}
\label{eq:class}
\cl(S)\cong\ZZ^2 \oplus \bigoplus_{i=1}^{r-1} %
\ZZ%
/
m_i%
,
\end{equation}
where $m_i:=\frac{\hat{d}_i}{\hat{d}_{i-1}}$, $\hat{d}_0=1$, and 
$\hat{d}_i=\gcd(\{\prod_{j\in I} d_j\}_{I\subset \{1,...,r\}, |I|=i})$.

Moreover, the following holds:
\begin{enumerate}[label=\rm(Cl\arabic{enumi})]
 \item\label{prop:class:1}
The free part is generated by the class of $C$ and the class of a suitable linear combination
of $A_1,\dots,A_r$ (which is a rational multiple of $F$).
 \item\label{prop:class:2}
The torsion part has order $m_1\cdot\ldots\cdot m_{r-1}=\frac{d_1\cdot\ldots\cdot d_r}{\vdeg}$.

\item\label{prop:class:3}
The element $T:=E-C\in \tor\cl (S)$ has maximal order $\vdeg=m_{r-1}$ in $\tor\cl (S)$.
\end{enumerate}
\end{prop}

\begin{remark}\label{rem:divT}
Note that there might be more than one subgroup of order $\vdeg$ in $\tor\cl (S)$, but the one 
generated by $T$ will be specially useful for our purposes. 
Note that $T$ is \emph{horizontal} since it is the difference of two sections, but 
it is also \emph{vertical} as it is linearly equivalent to 
$\sum_{i=1}^r q_i A_i-\alpha F$.
\end{remark}

\begin{proof}
From the presentation matrix one can easily see that $\cl(S)$ is the direct sum of 
the free subgroup $\cl_C(S):=\ZZ\langle C\rangle$ and $\cl_F(S):=\ZZ\langle A_1,\dots,A_r\rangle$.
Note that $\cl_F(S)\otimes_\ZZ\QQ=\QQ\langle F\rangle$ (of dimension~$1$) and that $\tor\cl (S)\subset\cl_F(S)$.
This shows part~\ref{prop:class:1}.

The presentation matrix for $\cl_F(S)$ is given as
\[
\begin{pmatrix}
d_1\!\!\!&\dots&0&-d_r\\
\vdots&\ddots&\vdots&\vdots\\
0&\dots&\!d_{r-1}\!\!\!&-d_r
\end{pmatrix}.
\]
Its Fitting ideals are $(\hat{d}_i)$, $i=1,...,r-1$, and hence its invariant factors are 
$m_i:=\frac{\hat{d}_i}{\hat{d}_{i-1}}$, which ends the structure shown in~\eqref{eq:class}.

Part~\ref{prop:class:2} follows from the formula $\hat{d}_{r-1} := 
\gcd( \frac{d_1\cdot\ldots\cdot d_r}{d_1},\ldots,
\frac{d_1\cdot\ldots\cdot d_r}{d_r} ) = \frac{d_1\cdot\ldots\cdot d_r}{\lcm(d_1,\ldots,d_r)}$,
see \eqref{eq:gcd-lcm} from the appendix, and the definition of $\vdeg=\lcm(d_1,\ldots,d_r)$.

For part~\ref{prop:class:3}, note that from the presentation matrix it follows that the maximal 
order of $\tor\cl (S)$ is $\vdeg$. Hence, it remains to verify that the order of $T$ is $\vdeg$,
\[
\vdeg T\sim \vdeg(E-C)\sim \vdeg\sum_{i=1}^r q_i A_i-\vdeg\alpha F\sim
\left(\sum_{i=1}^r \frac{\vdeg}{d_i} q_i -\vdeg\alpha\right)F=0.
\]
Then $\vdeg T\sim 0$. To check that $\vdeg$ is exactly the order
of $T$, we need Lemma~\ref{lem:unico} below. Assume that another
integer $\vdeg_1$ satisfies $\vdeg_1 T\sim 0$. Then, $d_i$ divides $\vdeg_1 q_i$. Since $q_i$ and $d_i$ are coprime,
$\vdeg_1$ is a multiple of $d_i$, and hence it is a multiple of~$\vdeg$. The fact that $\vdeg$ is precisely $m_{r-1}$
is a consequence of another arithmetic property, see \eqref{eq:gcd-lcm2} from the appendix.
\end{proof}

\begin{remark}
\label{rem:prop:class}
Note that only part~\ref{prop:class:3} depends on the condition~\eqref{eq:alpha}. 
Also, as a consequence of parts~\ref{prop:class:2} and~\ref{prop:class:3}, note that $\vdeg$ divides 
$\frac{d_1\cdot\ldots\cdot d_r}{\vdeg}$ and thus $\vdeg^2$ divides $d_1\cdot\ldots\cdot d_r$.
\end{remark}

\begin{lemma}\label{lem:unico}
Let $D\in\cl_F(S)$. Then, there are unique $\sf,a_i\in\ZZ$, $0\leq a_i< d_i$, $i=1,\dots,r$, such that
\[
D\sim\sum_{i=1}^r a_i A_i+\sf F.
\]
\end{lemma}

\begin{proof}
If $D$ had two representations as in the statement, then the difference would represent 0 as a combination 
$0=\sum_{i=1}^r a_i A_i+\sf F\in\cl_F(S)$, where $-d_i< a_i<d_i$. Then, this expression would be a linear 
combination of $F-d_i A_i$ and hence $d_i$ would divide $a_i$, which can only happen if $a_i=0$. 
This implies~$\sf=0$.
\end{proof}

Note the following additional linear equivalences $G\sim \vdeg C\sim \vdeg E$ given by the projection~$\pi_G$.

\begin{remark}\label{rem:hats}
There are canonical ways to represent a divisor class in~$S$ up to linear equivalence, but for technical reasons, 
we will oftentimes use non-canonical expressions. However, one can apply Lemma~\ref{lem:unico} to find a 
unique representative for a divisor class. Note that any divisor $D$ is linearly equivalent to a non-unique
expression of the form
\begin{equation}
c\, C+eE+\sum_{i=1}^r a_i A_i+\sf F.
\end{equation}
The following term $\sv:=F\cdot D=c+e \in \ZZ$ is intrinsic to $D$. Hence, 
$D\sim \sv C+\hat{D}$, where
\[
\hat{D} = \sum_{i=1}^r (a_i+eq_i)A_i + (\sf-e\alpha) F \in \cl_F(S).
\]
Using Lemma~\ref{lem:unico} on $\hat{D}$ one obtains the \emph{canonical form}
\begin{equation}\label{eq:canonical}
D\sim \sv C+\sum_{i=1}^r \hat{a}_i A_i+\hat{\sf} F, 
\end{equation}
where

\begin{enumerate}[label=(D\arabic{enumi})]\label{enum:D}
\item\label{enum:D1} $\sv=F\cdot D\in \ZZ$,
\item\label{enum:D2} $\hat{a}_i\equiv (a_i+eq_i) \bmod d_i$ are integers in $[0,d_i)$, and
\item\label{enum:D3} $\hat{\sf}=\sf+a-\hat{a}\in\ZZ$, 
for $\hat{a}:=\sum_{i=1}^r\frac{\hat{a}_i}{d_i}\in\ZZ\frac{1}{\vdeg}$ and $a:=\sum_{i=1}^r\frac{a_i}{d_i}\in\ZZ\frac{1}{\vdeg}$.
\end{enumerate}

By Lemma~\ref{lem:unico}, the triple $(\sv,(\hat{a}_i)_{i=1,...r},\hat{\sf})$ 
characterizes the linear equivalence class of~$D$. Also, note that 
\[
\h:=C\cdot D=\hat{\sf}+\hat{a}=\sf+a\in\ZZ\frac{1}{\vdeg},
\]
and, moreover, the pair 
$(C\cdot D,F\cdot D)=(\h,\sv)\in\ZZ\frac{1}{\vdeg}\times\ZZ$ 
determines the linear class of $D$ up to torsion.
\end{remark}

In a natural way, we have the following exact sequence involving the \emph{horizontal} part $\cl_H(S):=\ZZ\langle C,E\rangle$ of $\cl(S)$:
\begin{equation}
\label{eq:horizontal-seq}
\begin{tikzcd}
0\ar[r]&\ZZ\oplus\ZZ/\vdeg\cong\cl_H(S)\ar[r,hook]&\cl(S)\ar[r]&\ZZ\oplus\bigoplus_{i=1}^{r-2}\ZZ/m_i\ar[r]&0.
\end{tikzcd}
\end{equation}

\begin{remark}\label{rem:cond_clh}
By~\eqref{eq:canonical}, the condition for $D$ to be in $\cl_H(S)$ is equivalent to 
$\varphi_D=0$ and $\hat{a}_i\equiv 0 \bmod d_i$.
By~\ref{enum:D2}, the latter is equivalent to the existence of a solution of
\begin{equation}\label{eq:aimod}
x \equiv a_iq_i^{-1} \mod d_i, \quad \forall i=1,\ldots,r.
\end{equation}
\end{remark}

\subsection{Canonical divisor}\mbox{}

In the forthcoming calculations, the role of the class of the canonical divisor on $S$ will be essential.
The following result describes it.
\begin{prop}\label{prop:canonical}
The divisor
\[
K_S:=-(C+E)+(r-2)F-\sum_{i=1}^r A_i
\]
is a canonical divisor of~$S$.
\end{prop}

\begin{proof}
Recall the blowing-up-down construction starting from $\Sigma_\alpha$ described in Figure~\ref{fig:S1}. 
A canonical divisor of $\Sigma_\alpha$ is $-(C+E+F+F')$, where $F$ and $F'$ are two fibers.
It is more convenient to consider the following linearly equivalent divisor
\[
K_{\Sigma_\alpha}:=-(C+E)+(r-2)F-\sum_{i=1}^r F_i.
\]
Recall $K_{\hat{S}}=\hat{\pi}^* K_{\Sigma_\alpha}+K_{\hat{\pi}}$ where $K_{\hat{\pi}}$ is the relative 
canonical divisor of $\hat{\pi}:\hat{S}\to\Sigma_\alpha$. 
Since the divisor $K_{\Sigma_\alpha}$ is \emph{logarithmic} at the centers of the blow-ups, then 
\[
K_{\hat{S}}:=-(C+E)+(r-2)F-\sum_{i=1}^r (A_i+F_i).
\]
The direct image under $\rho$ gives the result.
\end{proof}

\begin{remark}
Note that $K_S\cdot F=-2$ and $K_S\cdot C=-\chi^{\textrm{orb}}$.
\end{remark}
\section{Cohomology of line bundles}
\label{sec:cohomology}
Let $D$ be a Weil divisor of $S$. The main goal of this section
is to compute the cohomology groups $H^i(S,\cO_S(D))$ for $i=0,1,2$.
The key point in these calculations relies on the interpretation of
the global sections of $\cO_S(D)$ as global sections of a line bundle on a 
weighted projective plane, that is, the vector space of quasihomogeneous
polynomials of a fixed degree satisfying certain vanishing conditions.
Then, Serre's duality and Riemann-Roch's formula for normal surfaces is applied
to obtain the second cohomology group and the Euler characteristic, respectively.
Finally, the first cohomology group is obtained as a side product. For this reason, we 
have organized this section in four parts, where the different objects are studied, namely 
\S\ref{subsec:global} global sections, 
\S\ref{subsec:euler} Euler characteristics,  
\S\ref{sec:1st-cohomology} general vanishing results, and \S\ref{subsec:special} special cases (the last two both serve the
understanding of the first cohomology group).

\subsection{Global sections}\label{subsec:global}
\mbox{}

Consider $D$ a Weil divisor in $S$. By~\eqref{eq:clS}, its class in $\cl(S)$ can be written as 
$c\, C + eE + \sum_{i=1}^r a_i A_i + \sf F$, where $c, e, a_i \in \ZZ$, $i=1,\ldots,r$, that is,
\begin{equation}\label{class-of-D}
D \sim c\, C + eE + \sum_{i=1}^r a_i A_i + \sf F \sim \sv\, C + \sum_{i=1}^r \hat{a}_i A_i + \hat{\sf}F,
\end{equation}
where the right-most expression is unique as described in Remark~\ref{rem:hats}.
Note, however, that $c,e,a_i$ are not uniquely determined by $D$, since the group $\cl(S)$ is not torsion 
free and $C,E,A_i$ are not linearly independent. 

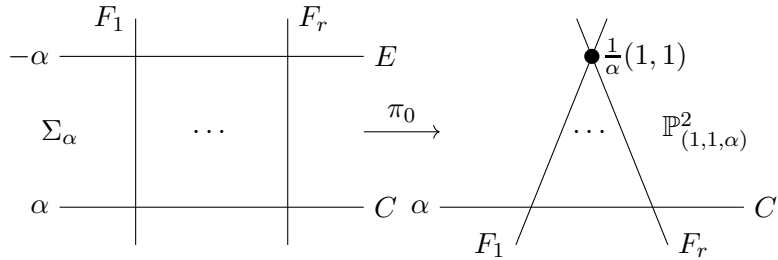
\begin{figure}[ht]
﻿\begin{tikzpicture}
\draw (-2,-1) -- (2,-1) node[right] {$C$} node[left,pos=0] {$\alpha$};
\draw (-2,1) -- (2,1)  node[right] {$E$}node[left,pos=0] {$-\alpha$};
\draw (-1,-1.5) -- (-1,1.5)  node[left] {$F_1$};
\draw (1,-1.5) -- (1,1.5)  node[right] {$F_r$};
\node at (0,0) {$\dots$};
\node at (-2,0) {$\Sigma_\alpha$};

\draw[->] (2,0)--(3,0) node [above,pos=.5] {$\pi_0$};

\begin{scope}[xshift=5cm]
\draw (-2,-1) -- (2,-1) node[right] {$C$} node[left,pos=0] {$\alpha$};
\draw (-1,-1.5) -- ($-.2*(-1,-1.5)+1.2*(0,1)$)  node[left,pos=0] {$F_1$};
\draw (1,-1.5) --  ($-.2*(1,-1.5)+1.2*(0,1)$) node[right,pos=0] {$F_r$};
\node at (0,0) {$\dots$};
\node at (1.5,0) {$\mathbb{P}^2_{(1,1,\alpha)}$};
\fill (0,1) circle [radius=.1] node[right] {$\frac{1}{\alpha}(1,1)$};
\end{scope}
\end{tikzpicture}
\caption{Birational transformation to $\PP^2_{(1,1,\alpha)}.$}
\end{figure}

Recall that $S$ is birationally equivalent to
$\PP^2_{(1,1,\alpha)}$ following the diagram
\[
\begin{tikzcd}
\PP^2_{(1,1,\alpha)}
& [20pt] \Sigma_\alpha \ar[l, "\pi_0"']
& [40pt] \widehat{S} \ar[l, "\hat{\pi}:=\pi_1 \circ \dots \circ \pi_r"'] \ar[r, "\rho"] \ar[ll, "\pi"', bend left = 25]
& [20pt] S,
\end{tikzcd}
\]
where
\[
\begin{aligned}
\rho^{*} (A_i) &= A_i + \frac{1}{d_i} F_i, & \qquad & \rho^{*} (E) = E + \sum_{i=1}^r \frac{q_i}{d_i} F_i, & \qquad & \rho^{*} (C) = C, \\
\pi^{*} (F_i) &= F_i + \frac{1}{\alpha} E + d_i A_i, && \pi^{*} (C) = C + \sum_{i=1}^r q_i A_i,
& \qquad & \pi^{*} (F) = F + \frac{1}{\alpha} E. \\
\end{aligned}
\]
By the Projection Formula for normal surfaces (see~\cite[Theorem 2.1]{Sakai84}) 
\[
H^0(S,\cO_S(D)) \simeq H^0 \left( \widehat{S},\cO_{\widehat{S}}(D') \right),
\]
where
\[
D' := \lfloor \rho^* (D) \rfloor = c\, C + e E + \sum_{i=1}^r a_i A_i + \sf F + 
\sum_{i=1}^r \left\lfloor \frac{eq_i + a_i}{d_i} \right\rfloor F_i.
\]
Using the morphism $\pi$, there is a natural identification of the global sections of $\cO_{\widehat{S}}(D')$ with those 
of~$\cO_{\PP^2_{(1,1,\alpha)}}(\pi_*(D'))$.
More precisely, according to \cite[Proposition 4.2(2)]{ACM19},
\[
H^0 \left( \widehat{S},\cO_{\widehat{S}}(D') \right)\! \simeq\!
\left\{
H \!\in\! \CC[x,y,z]_{(1,1,\alpha),d} \ \bigg| 
\begin{aligned}
& \mult_{E'} (\pi^*(H))\! \geq \mult_{E'} \left( \pi^* \left(%
\pi_* (D') %
\right) \!-\! D' \right) \\
& \ \forall E' \in\exc(\pi) = \{E,A_1,\ldots,A_r\}
\end{aligned}
\right\}\!,
\]
where $\CC[x,y,z]_{(1,1,\alpha),d}$ denotes the $(1,1,\alpha)$-quasihomogeneous polynomials in $x,y,z$ 
of degree $d := \deg_{(1,1,\alpha)} (\pi_{*} (D'))$.

Note that
\[
\pi_* (D') %
= c\, C + \sf F + \sum_{i=1}^r \left\lfloor \frac{eq_i + a_i}{d_i} \right\rfloor F_i,
\]
which has degree $d = \alpha c + \sf + \sum_{i=1}^r \left \lfloor \frac{eq_i+a_i}{d_i} \right\rfloor=\alpha \sv + \hat{\sf}$, 
and thus
\begin{equation}\label{eq:condD}
\pi^* \left(%
\pi_* (D') %
\right) - D'
= \frac{1}{\alpha}\left( \sum_{i=1}^r \left\lfloor \frac{eq_i+a_i}{d_i} \right\rfloor - e\alpha + \sf \right) E
+ \sum_{i=1}^r \left( cq_i + \left\lfloor \frac{eq_i+a_i}{d_i} \right\rfloor d_i - a_i \right) A_i.
\end{equation}

Using that $(c+e)=D\cdot F=\sv$ (see~\ref{enum:D1}) and 
$eq_i+a_i - \left\lfloor \frac{eq_i+a_i}{d_i} \right\rfloor d_i \equiv (eq_i+a_i)\equiv \hat{a}_i \bmod d_i$
(see~\ref{enum:D2}),
the coefficient of $A_i$ in \eqref{eq:condD} can be rewritten as~$\sv q_i - \hat{a}_i$
and the one of $E$ as $\frac{\hat{\sf}}{\alpha}$.

Assume, without loss of generality, that $\pi_0: \Sigma_\alpha \to \PP^2_{(1,1,\alpha)}$ is the
$(1,1)$-blow-up at the point $[0:0:1] \in \PP^2_{(1,1,\alpha)}$. Then, 
\[
\mult_E (\pi^* (H)) = \frac{1}{\alpha} \ord ( H(x,y,1) ).
\]
Also assume that $F_i$ is the line in $\PP^2_{(1,1,\alpha)}$ given by $x-\gamma_i y = 0$, $i = 1,\ldots,r$,
($\gamma_i\neq\gamma_j$, if $i\neq j$) and $C = \{z=0\}$ so that $F_i \cap C = \{ [\gamma_i:1:0]\}$. 
Then $\pi_i$ is the $(d_i,q_i)$-blow-up at a smooth point $(\gamma_i,0) \in \CC^2$ with local coordinates $(x,z)$. Hence
\[
\mult_{A_i} (\pi^* (H)) = \ord (H(x^{d_i}+\gamma_i,1,z^{q_i})).
\]
Summarizing $H^0(S,\cO_S(D))$ can be identified via $\pi$ and $\rho$ with the vector space of $(1,1,\alpha)$-quasihomogeneous
polynomials $H(x,y,z)$ in $x,y,z$ satisfying
\begin{equation}\label{description-h0OD}
\left\{\begin{aligned}
& \deg (H(x,y,z)) = \sv \alpha + \hat{\sf} = \sv \alpha + \sf + a - \hat{a}  = d, \\
& \ord (H(x,y,1)) \geq \hat{\sf} = \sf + a - \hat{a}, \\[3pt]
& \ord (H(x^{d_i}+\gamma_i,1,z^{q_i})) \geq \sv q_i - \hat{a}_i, \quad \forall i = 1,\ldots,r,
\end{aligned}\right.
\end{equation}
where $a=\sum_{i=1}^r \frac{a_i}{d_i}$ and $\hat{a}=\sum_{i=1}^r \frac{\hat{a}_i}{d_i}$ as described in Remark~\ref{rem:hats}.

In order to describe the contribution of the different cohomology spaces $H^i(S,\cO_S(D))$ it is very 
convenient to construct a lattice that will encode relevant properties of divisors classes.

\begin{defi}
\label{sec:lattice} 
We shall define the \emph{divisor lattice} $L:=\ZZ\frac{1}{\vdeg} \times \ZZ\subset \QQ^2$ 
and the map  $\cl(S)\to L$ given by $D\mapsto \ell_D:=(\h,\sv)=(D\cdot C,D\cdot F)\in L$. 
By the discussion in Remark~\ref{rem:hats}, this map is onto. 
Moreover, its kernel is given by the torsion part of~$\cl(S)$. In particular, given a 
lattice point $\ell\in L$, there are exactly $\frac{d_1\cdot\ldots\cdot d_r}{\vdeg}$ 
divisor classes in $\cl(S)$ whose images coincide with~$\ell\in L$.
\end{defi}

For instance, the following proposition states that $H^0(S,\cO_S(D)) \neq 0$ is only possible if 
$\ell_D$ sits on the first quadrant $L_{\geq 0}:=L\cap \QQ^2_{\geq 0}$ (lattice axes included) of~$L$.

\begin{prop}\label{two-directions}
Using the previous notation, 
\begin{equation}
\label{H0}
H^0(S,\cO_S(D)) \neq 0 \ \Rightarrow \ \ell_D\in L_{\geq 0},
\end{equation}
see the left-hand side of Figure{\rm~\ref{fig:H012}}.
\end{prop}

\begin{proof}
We will use the intersection theory for weighted projective planes developed in~\cite[Proposition 5.2]{AMO-Intersection}.
Choose $H(x,y,z) \in H^0(S,\cO_S(D))$ different from zero.

Since $F$ is generic, $H$ and $F$ do not have common components and the intersection $H \cap F$ consists of a finite
number of points. Moreover,
\[
  \frac{d}{\alpha} = \frac{\deg (H) \cdot \deg (F)}{\alpha} = H \cdot F
  = \sum_{P \, \in \, \PP^2_{(1,1,\alpha)}} (H \cdot F)_P \geq (H \cdot F)_{[0:0:1]}
  = \frac{\ord(H(x,y,1))}{\alpha}.
\]
Therefore $d \geq \ord(H(x,y,1))$. Recall that $d = \sv\alpha + \hat{\sf}$
and, due to \eqref{description-h0OD}, $\ord(H(x,y,1)) \geq \hat{\sf} = \sf + a-\hat{a}$.
Hence $\sv\alpha \geq 0$ and
$D \cdot F = \sv= c+e \geq 0$.

Let us check now that $D \cdot C \geq 0$.
Assume $H(x,y,z) = z^m H'(x,y,z)$, $m\geq 0$, where $H'$ and $C = \{ z = 0 \}$ do not have
common components. According to~\eqref{description-h0OD},
\[
\begin{aligned}
& \deg (H'(x,y,z)) = \alpha (c-m) + \sum_{i=1}^r \left \lfloor \frac{eq_i+a_i}{d_i} \right\rfloor = \hat{\sf} + (\sv-m)\alpha, \\
& \ord ( H'(x,y,1) ) \geq \hat{\sf}, \\[7pt]
& \ord (H'(x^{d_i}+\gamma_i,1,z^{q_i})) \geq (\sv-m) q_i - \hat{a}_i, \quad \forall i = 1,\ldots,r.
\end{aligned}
\]
We apply Bézout's identity to $H'$ and $C$ and obtain
\begin{equation}\label{eq:bezout-GC}
  \hat{\sf} + (\sv-m)\alpha
  = \frac{\deg (H') \cdot \deg (C)}{\alpha}
  = H' \cdot C = \sum_{P \, \in \, \PP^2_{(1,1,\alpha)}} (H' \cdot C)_P
  \geq \sum_{i=1}^r (H' \cdot C)_{F_i \, \cap \, C}.
\end{equation}
It can be checked that $\mult_{A_i} (\pi^* (H')) = \ord (H'(x^{d_i}+\gamma_i,1,z^{q_i})) \leq d_i (H' \cdot C)_{F_i \, \cap \, C}$.
Indeed, since this is a local problem, one can assume $\gamma_i=0$, that is, $F_i = \{x=0\}$ and $C=\{z=0\}$.
If $(H' \cdot C)_{F_i \, \cap \, C} = n$, then $x^n$ is a term of $H'(x,1,z)$ and thus $\ord_{(d_i,q_i)} (H'(x,1,z)) \leq n d_i$.
The multiplicity of $\pi^* (H')$ along $A_i$ equals the $(d_i,q_i)$-order of $H'(x,1,z)$ because $\pi_i$ is nothing but the $(d_i,q_i)$-blow-up
at the point $F_i \cap C$.
Then, %
using~\eqref{eq:alpha}, one has
\begin{equation}\label{eq:bound-GC}
\begin{aligned}
& \sum_{i=1}^r (H' \cdot C)_{F_i \, \cap \, C}
\geq \sum_{i=1}^r \frac{\mult_{A_i} (\pi^* (H'))}{d_i}
\geq \sum_{i=1}^r \frac{(\sv-m)q_i - \hat{a}_i}{d_i} = (\sv-m)\alpha-\hat{a}.
\end{aligned}
\end{equation}
Combining~\eqref{eq:bezout-GC} and~\eqref{eq:bound-GC} gives $\hat{a} + \hat{f} = a + f \geq 0$ as desired.
\end{proof}

\begin{theorem}\label{thm-h0}
Let $D\sim c\, C+eE+\sum_{i=1}^r a_i A_i+\sf F$ be a Weil divisor on the normal surface $S$.
The dimension of $H^0(S,\cO_S(D))$ as a $\CC$-vector space is
\[
h^0(S,\cO_S(D)) = \sum_{j=0}^{\sv} \max\{b_j(D),0\},
\]
where
\[
b_j(D) := 1+\h-\sum_{i=1}^r \left\{ \frac{a_i+(j-c)q_i}{d_i} \right\}\in \ZZ.
\]
In particular, $h^0(S,\cO_S(D))$ does not depend on the position of the singular points of $S$,
but only on the singular types $\{(d_i;1,q_i)\}_{i=1}^r$ and the class of $D$ in $\cl(S)$.
\end{theorem}

\begin{proof}
We come back to the description given in~\eqref{description-h0OD} by considering $H(x,y,z)$ a 
generic $(1,1,\alpha)$-weighted homogeneous polynomial of degree $d$. Hence, let us write
\begin{equation}\label{eq:decomposition-H}
H(x,y,z) = \sum_{j \geq 0} h_{d-\alpha j}(x,y) z^j,
\end{equation}
where $h_{d-\alpha j}(x,y)$ is a homogeneous polynomial of degree $d-\alpha j$.

Let $j_{\max}=\lfloor \frac{d}{\alpha}\rfloor$ denote the maximum value of $j$ such that 
$d-\alpha j \geq 0$. Then $H(x,y,1) = \sum_{j \geq 0} h_{d-\alpha j}(x,y)$ and its order is
\[
\ord ( H(x,y,1) )
= d - \alpha j_{\max} = (c_D - j_{\max})\alpha + \hat{\sf},
\]
which is greater than or equal to $\hat{\sf}$ if and only if $j_{\max} \leq \sv$. 
Hence the sum in~\eqref{eq:decomposition-H} runs from $j=0$ to $j=\sv$.

The condition $\ord (H(x^{d_i}+\gamma_i,1,z^{q_i})) \geq \sv q_i - \hat{a}_i$,
$\forall i = 1,\ldots,r$, implies that $h_{d - \alpha j}(x,y)$ is of the form
\[
h_{d - \alpha j}(x,y) = \prod_{i=1}^r (x-\gamma_i y)^{m_{ij}} g_j(x,y),
\]
where $g_j(x,y)$ is a homogeneous polynomial of degree $d-\alpha j - \sum_{i=1}^r m_{ij}$, where 
\[
m_{ij} = \left\lceil \frac{(\sv-j) q_i - \hat{a}_i}{d_i} \right\rceil \geq 0.
\]
Since the degrees of freedom of $g_j(x,y)$
is its degree plus one if the degree is nonnegative, or zero otherwise, the required dimension is
\[
\sum_{j=0}^{\sv} \max \left\{ 1 + d - \alpha j - \sum_{i=1}^r m_{ij}, 0 \right\}.
\]
Note that
\[
\begin{aligned}
1+d-\alpha j-\sum_{i=1}^r m_{ij} & = 
1+\hat{\sf}+(\sv-j)\alpha - \sum_{i=1}^r\left\lceil \frac{(\sv-j) q_i - \hat{a}_i}{d_i} \right\rceil\\
& = 1+(\hat{\sf}+\hat{a})+ \sum_{i=1}^r \left( \frac{(\sv-j) q_i - \hat{a}_i}{d_i} - 
\left\lceil \frac{(\sv-j) q_i -\hat{a}_i}{d_i} \right\rceil \right)\\
& = 1+\h-\sum_{i=1}^r \left\{ \frac{\hat{a}_i+(j-\sv) q_i}{d_i}\right\} \\
&= 1+\h - \sum_{i=1}^r \left\{ \frac{a_i+(j-c) q_i}{d_i}\right\}.
\end{aligned}
\]
The last equality follows from the fact that 
\[
\frac{a_i+(j-c) q_i}{d_i}-\frac{\hat{a}_i+(j-\sv) q_i}{d_i}=
\frac{a_i+e q_i-\hat{a}_i}{d_i}\in \ZZ.
\]
Once a formula for computing $h^0(S,\cO_S(D))$ has been found, the last part of the statement easily follows.
\end{proof}

As a consequence one can determine the region of $L$ where $H^2(S,\cO_S(D)) \neq 0$ is concentrated.

\begin{cor}\label{cor:h2-OSD}
Using the previous notation, 
\begin{equation}
\label{H2} 
H^2(S,\cO_S(D)) \neq 0 \ \Rightarrow \ -\ell_D\in (\chi^{\textrm{orb}},2)+L_{\geq 0},
\end{equation}
see the middle part of Figure{\rm~\ref{fig:H012}}.
In this case,
\[
h^2(S,\cO_S(D)) = \sum_{j=0}^{-(2+\sv)} \max \{ b_j(K_S-D), 0 \}
\]
and
\[
b_j (K_S-D) = 1-(\h + \chi^{\textrm{orb}})-
\sum_{i=1}^r \left\{ \frac{-1-a_i + (c+j+1)q_i}{d_i} \right\}.
\]
\end{cor}

\begin{proof}
By Serre's duality, $h^2(S,\cO_S(D)) = h^0(S,\cO_S(K_S-D))$.
Recall that the canonical divisor of $S$ is $K_S = - C - E - \sum_{i=1}^r A_i + (r-2)F$.
Then,
\[
K_S - D = (-1-c)C + (-1-e)E + \sum_{i=1}^r (-1-a_i)A_i + (r-2-\sf)F.
\]
From Proposition~\ref{two-directions},
$h^0(S,\cO_S(K_S-D)) \neq 0$ implies that $(K_S - D) \cdot F \geq 0$ and
$(K_S - D) \cdot C \geq 0$, or equivalently, $D \cdot F \leq K_S \cdot F = -2$
and $D \cdot C \leq K_S \cdot C = - \sum_{i=1}^r \frac{1}{d_i} + r-2 = -\chi^{\textrm{orb}}$,
as claimed.

The second part of the statement follows from Theorem~\ref{thm-h0} 
applied to $K_S - D$.
\end{proof}

\subsection{Euler characteristic}\label{subsec:euler}
\mbox{}

The main purpose of this section is to compute the Euler characteristic of the sheaf $\cO_S(D)$.
As above assume that
\[
D \sim c\, C + eE + \sum_{i=1}^r a_i A_i+\sf F,
\]
where $c,e,\sf,a_i \in \ZZ$, $i=1,\ldots,r$ and recall that $K_S = - C - E - \sum_{i=1}^r A_i + (r-2)F$.
In order to calculate $\chi(S,\cO_S(D))$, we will use the Riemann-Roch formula on singular normal surfaces
from~\cite[\S 1.2]{Blache95}, see also~\cite{jiJM-correction}, that is,
\begin{equation}
\label{eq:RRF}
\chi(S,\cO_S(D)) = \chi(S,\cO_S) + \frac{D \cdot (D - K_S)}{2} + R_S(D).
\end{equation}

Recall that the correction term $R_S(D)$ is a sum of local invariants associated with each singular point $P$ 
of $S$ and the local class of $D$ at $P$. In our case
\begin{equation}
\label{eq:RS}
R_S(D)=\sum_{i=1}^r R_{S_i^{+}}(D)+\sum_{i=1}^r R_{S_i^{-}}(D),
\end{equation}
where $S_i^{\pm}$ denotes the local singularity type~$\frac{1}{d_i}(1,\pm q_i)$.

First note that $S$ is a rational surface (it is birationally equivalent to a weighted projective plane)
and hence~$\chi(S,\cO_S)=1$.
Also, using the fact that $C$ and $E$ are numerically equivalent to each other, $D$ and $K_S$ can numerically 
be described as
\[
D \equiv \sv\, C + \h F, \qquad K_S \equiv -2C -\chi^{\textrm{orb}} F.
\]
Then, 
\[
\begin{aligned}
& D^2 = 2\sv\, \h, \\
& D \cdot K_S = -\sv\, \chi^{\textrm{orb}} - 2\h, \\
& D^2 - D \cdot K_S = 2(\sv+1) \h + \sv\, \chi^{\textrm{orb}},
\end{aligned}
\]
and therefore
\begin{equation}\label{eq:chi-OSD}
\begin{aligned}
\chi(S,\cO_S(D))
&= 1 + (\sv+1) \h + \frac{1}{2}\sv\, \chi^{\textrm{orb}} + R_S(D). \\
\end{aligned}
\end{equation}

In order to continue with this calculation, we need to understand the local contribution of each $R_{S_i^{\pm}}(D)$ to the correction term $R_S(D)$. 

Consider $d,q \in \ZZ$ any three integers and denote by $S^\pm$ the cyclic singularity 
$\frac{1}{d}(1,\pm q)$. Since $R_{S^\pm}(D)$ only depends on the local class of the divisor $D$, 
we can consider it as a map $R_{S^\pm}:\Weil(S^\pm)/ \Cart(S^\pm)\cong\ZZ/d\ZZ \to \QQ$.
The following result gives a closed formula for the combined contribution $R_{S^+}(n) + R_{S^-}(n-mq)$
with the convention that 
$\sum_{j=k_1}^{k_2}f(j)=0$ if $k_1>k_2$.

\begin{lemma}\label{lemma:RR}
Under the conditions above, let $n,m \in \ZZ$ be such that $m \geq -1$. Then,
\[
R_{S^+}(n) + R_{S^-}(n-mq) = - \sum_{j=0}^{m} \left\{ \frac{n+(j-m)q}{d} \right\} + m \frac{d-1}{2d}.
\]
\end{lemma}

\begin{proof}
For $m \geq 0$, we proceed by induction on $m$. The first case, namely 
$R_{S^+}(n) + R_{S^-}(n)= - \left\{ \frac{n}{d} \right\}$, was already proven 
in~\cite[Proposition 2.16(2)]{campillo}. To be precise, in \emph{loc.~cit.}~the result was stated 
in terms of the so-called $\Delta$-invariant. In this context, the relationship between $\Delta$ 
and $R$ is simply given by $R_{S^+}(n) = -\Delta_{S^+}(-n)$ (see~\cite[Introduction]{jiJM-correction}).
Similarly $m=1$ is a direct consequence of~\cite[Proposition 2.16(1)]{campillo} and the case $m=0$.

Assume the result is true for $m \geq 1$ and we will prove it for $m+1$. The cases $m=1$, $m=0$,
and the induction hypothesis tell us respectively that 
\[
\begin{aligned}
R_{S^+}(n) &= - R_{S^-}(n-q) - \left\{ \frac{n-q}{d} \right\} - \left\{ \frac{n}{d} \right\} + \frac{d-1}{2d}, \\
0 &= R_{S^+}(n-q) + R_{S^-}(n-q) + \left\{ \frac{n-q}{d} \right\}, \\
R_{S^-}(n-q-mq) &= - R_{S^+}(n-q) - \sum_{j=0}^m \left\{ \frac{n-q+(j-m)q}{d} \right\} + m \frac{d-1}{2d}.
\end{aligned}
\]
Adding up these three equations gives the result for $m+1$.

It remains to prove the case $m=-1$, that is, $R_{S^+}(n) + R_{S^-}(n+q) = - \frac{d-1}{2d}$,
which is again a reformulation of~\cite[Proposition 2.16(1)]{campillo}.
\end{proof}

\begin{theorem}\label{thm:chi-OSD}
\[
\chi(S,\cO_S(D)) =
\begin{cases}
\displaystyle\sum_{j=0}^{\sv} b_j(D)  & \text{if} \quad \sv \geq 0, \\[0.75cm]
0 & \text{if} \quad \sv=-1, \\[0.30cm]
\displaystyle\sum_{j=0}^{-(\sv+2)} b_j(K_S-D) & \text{if} \quad \sv \leq -2.
\end{cases}
\]
\end{theorem}

\begin{proof}
The result will follow after combining~\eqref{eq:RRF},~\eqref{eq:RS},~\eqref{eq:chi-OSD},
and Lemma~\ref{lemma:RR}.
\[
\begin{aligned}
R_S(D)
&= \sum_{i=1}^r \left( R_{S_i^+}(D) + R_{S_i^-}(D) \right) \\
&= \sum_{i=1}^r \left( R_{S_i^+}(eE+a_iA_i) + R_{S_i^-}(c\, C+a_iA_i) \right) \\
&= \sum_{i=1}^r \left( R_{S_i^+}(a_i+eq_i) + R_{S_i^-}(a_i-cq_i) \right). \\
\end{aligned}
\]
By Lemma~\ref{lemma:RR}, applied to $d=d_i$, $q=q_i$, $n=a_i+eq_i$, $m=\sv \geq -1$, 
each summand can be rewritten in terms of fractional parts so that
\[
R_S(D) = \sum_{i=1}^r \left( - \sum_{j=0}^{\sv} \left\{ \frac{a_i+(j-c)q_i}{d_i} \right\} + 
\sv\, \frac{d_i-1}{2d_i} \right).
\]
Hence
\[
\begin{aligned}
\chi(S,\cO_S(D))
&= 1+(\sv+1)\h+\frac{1}{2}\sv\, \chi^{\textrm{orb}} - \sum_{i=1}^r \sum_{j=0}^{\sv} 
\left\{ \frac{a_i+(j-c)q_i}{d_i} \right\} + \frac{1}{2}\sv\, \sum_{i=1}^r \frac{d_i-1}{d_i}\\
&= 1+\frac{1}{2}\sv\, \left( \chi^{\textrm{orb}} + \sum_{i=1}^r \frac{d_i-1}{d_i}\right)
+ \sum_{j=0}^{\sv} \left( \h -\sum_{i=1}^r\left\{ \frac{a_i+(j-c)q_i}{d_i} \right\}\right) \\
&= \sum_{j=0}^{\sv} b_j(D),
\end{aligned}
\]
as claimed.

In particular, when $\sv=-1$, one has $\chi(S,\cO_S(D))=0$.
The last part of the statement follows from Serre's duality
and the fact that $D \cdot F = \sv \leq -2$ implies that $(K_S-D) \cdot F = -(\sv+2) \geq 0$.
\end{proof}

\subsection{General vanishing results}\label{sec:1st-cohomology}
\mbox{}

According to Theorem~\ref{thm-h0}, Corollary~\ref{cor:h2-OSD}, and Theorem~\ref{thm:chi-OSD}, 
the Euler characteristic of the sheaf $\cO_S(D)$ coincides with the dimension of $H^0$ 
(resp.~$H^2$) as long as $b_j(D) \geq 0$ (resp. $b_j(K_S-D) \geq 0$), $\forall j$.
In this section we will investigate when each one of these conditions hold. This will affect the vanishing 
of the first cohomology group. As above we assume that
\begin{equation}\label{eq:ntc}
D \sim c\, C + eE + \sum_{i=1}^r a_i A_i + \sf F \sim \sv\, C + \sum_{i=1}^r \hat{a}_i A_i + \hat{\sf} F,
\end{equation}
where $(\sv,\{\hat{a}_i\}_{i=1,...,r}, \hat{\sf})$ are uniquely determined by its canonical form 
(see Remark~\ref{rem:hats}), and $\ell_D=(C\cdot D,F\cdot D)=(\h,\sv)$ (see~\ref{sec:lattice}).

Figure~\ref{fig:H012} depicts a projection of $D\in\cl(X)$ onto the divisor lattice $(\h,\sv)\in L$ 
(see Definition~\ref{sec:lattice}) and it describes regions where cohomological triviality is assured for 
all divisor classes corresponding to each value $(\varphi,c)\in L$. 
Combining these, one obtains in Figure~\ref{fig:H012a} four cones where cohomology is concentrated in one degree. 

The following result is obtained combining the previous results on global sections and Euler characteristics.
It describes the general regions of $L$ where the cohomology is concentrated in a single degree. The 
statement of the following theorem is 
summarized in the right-hand side of Figure~\ref{fig:H012}.

\begin{theorem}\label{prop:h1}
Let $D$ be as in~\eqref{eq:ntc}. Then the following holds: 
\begin{enumerate}[label=\rm(\arabic{enumi})]
\item\label{prop:h1:1} 
If $\ell_D\in (-\chi^{\textrm{orb}},-2)+L_{>0}$, then $h^*(S,\cO_S(D))$ is concentrated in degree~$0$.

\item\label{prop:h1:2} 
If $-\ell_D\in L_{>0}$, then $h^*(S,\cO_S(D))$ is concentrated in degree~$2$.

\item\label{prop:h1:4} 
If $\ell_D=(\h,\sv)\in L$ with either
\begin{enumerate}[label=\rm(\roman{enumii})]
\item $\h< 0$ and $\sv\geq 0$, or
\item $\h>-\chi^{\textrm{orb}}$ and $\sv\leq -2$,
\end{enumerate}
then $h^*(S,\cO_S(D))$ is concentrated in degree~$1$.
\item\label{prop:h1:3} 
If $\sv = -1$, then $h^i(S,\cO_S(D))=0$ for $i=0,1,2$.
\end{enumerate}
\end{theorem}

\begin{proof}
To prove~\ref{prop:h1:1}, let us assume $\h > -\chi^{\textrm{orb}}$ and $\sv > -2$.
Note that the maximum value of $\{ \frac{n}{d_i} \}$ is reached when $n=d_i-1$. Then,
\begin{equation}\label{bjD-bound}
b_j(D) = 
1+\h-\sum_{i=1}^r \left\{ \frac{a_i+(j-c)q_i}{d_i} \right\}
\geq 1+ \h -\sum_{i=1}^r \frac{d_i-1}{d_i}
=\h + \chi^{\textrm{orb}}- 1.
\end{equation}
Under the current hypotheses this number is an integer strictly greater than $-1$ and hence $\max\{ b_j(D), 0 \} = b_j(D)$.
By Theorems~\ref{thm-h0} and~\ref{thm:chi-OSD}, $h^0 (S,\cO_S(D)) = \chi(S,\cO_S(D))$.
Corollary~\ref{cor:h2-OSD} provides the vanishing of $H^2(S,\cO_S(D))$ and thus the vanishing of~$H^1(S,\cO_S(D))$.

Part~\ref{prop:h1:2} is a consequence of Serre's duality. 

Part~\ref{prop:h1:4} follows from~\eqref{H0} in Proposition~\ref{two-directions} and 
\eqref{H2} in Corollary~\ref{cor:h2-OSD}.

Finally, for part~\ref{prop:h1:3} note that,
if $\sv=-1$, then $h^0(S,\cO_S(D)) = h^2(S,\cO_S(D)) = 0$ from 
Proposition~\ref{two-directions}
and Corollary~\ref{cor:h2-OSD}, respectively. Then $h^1(S,\cO_S(D)) = - \chi(S,\cO_S(D)) = 0$
from Theorem~\ref{thm:chi-OSD} and the result follows.
\end{proof}

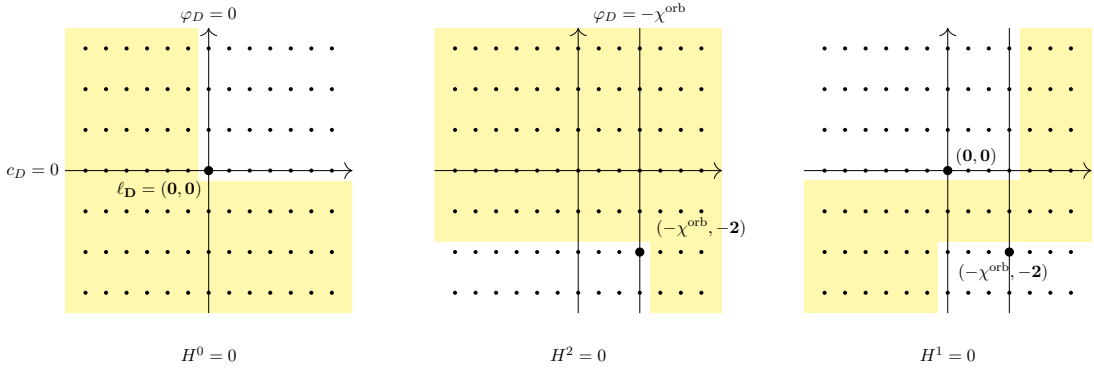
\begin{figure}[ht]
\begin{center}
\scalebox{.6}
{
﻿\begin{tikzpicture}[scale=.9]
\fill[color=yellow!40!white] (-3.5,3.5)--(-3.5,-3.5)--(3.5,-3.5)--(3.5,-.25)--(-.25,-.25)--(-.25,3.5)--cycle;
\foreach \a in {-3,-2.5,...,3}
{
\foreach \b in {-3,-2,...,3}
{
\fill[black] (\a,\b) circle [radius=.05cm];
}
}
\draw[-{[scale=2]>}] (-3.5,0)--(3.5,0) node[left,pos=0] {$\sv=0$};
\draw[-{[scale=2]>}]  (0,-3.5)--(0,3.5) node[above] {$\h=0$};
\node[below] at (-1.2,-.1) {$\mathbf{\ell_D=(0,0)}$};
\draw[fill=black] (0,0) circle[radius=.1];
\node at (0,-4.5) {$H^0=0$};

\begin{scope}[xshift=9cm]
\fill[color=yellow!40!white] (-3.5,3.5)--(-3.5,-1.75)--(1.75,-1.75)--(1.75,-3.5)--(3.5,-3.5)--(3.5,3.5)--cycle;
\foreach \a in {-3,-2.5,...,3}
{
\foreach \b in {-3,-2,...,3}
{
\fill[black] (\a,\b) circle [radius=.05cm];
}
}
\draw[-{[scale=2]>}] (-3.5,0)--(3.5,0);
\draw[-{[scale=2]>}]  (0,-3.5)--(0,3.5) ;
\node[below] at (3,-1) {$\mathbf{(-\chi^{\textrm{orb}},-2)}$};
\draw[fill=black] (1.5,-2) circle[radius=.1];
\node at (0,-4.5) {$H^2=0$};
\draw (1.5,-3.5) -- (1.5,3.5) node[above] {$\h=-\chi^{\textrm{orb}}$};
\end{scope}

\begin{scope}[xshift=18cm]
\fill[color=yellow!40!white] (1.75,3.5)--(1.75,-1.75)--(3.5,-1.75)--(3.5,3.5)--cycle;
\fill[color=yellow!40!white] (-3.5,-3.5)--(-.25,-3.5)--(-.25,-.25)--(-3.5,-.25)--cycle;
\fill[color=yellow!40!white] (-.25,-1.75) rectangle (1.75,-.25);
\foreach \a in {-3,-2.5,...,3}
{
\foreach \b in {-3,-2,...,3}
{
\fill[black] (\a,\b) circle [radius=.05cm];
}
}
\draw[-{[scale=2]>}] (-3.5,0)--(3.5,0);
\draw[-{[scale=2]>}]  (0,-3.5)--(0,3.5);

\node[below] at (.7,.7) {$\mathbf{(0,0)}$};
\node[below] at (1.35,-2.1) {$\mathbf{(-\chi^{\textrm{orb}},-2)}$};
\draw[fill=black] (1.5,-2) circle[radius=.1];
\draw[fill=black] (0,0) circle[radius=.1];
\node at (0,-4.5) {$H^1=0$};
\draw (1.5,-3.5) -- (1.5,3.5);
\end{scope}
\end{tikzpicture}
}
\caption{Vanishing of $H^j$.}
\label{fig:H012}
\end{center}
\end{figure}

\subsection{Special cases}\label{subsec:special}
\mbox{}

For our goal in \S\ref{sec:main} we are interested in finding explicit formulas for $h^1(S,\cO_S(D))$ when $-\ell_D\in L_{\geq 0}$.
By Theorem~\ref{prop:h1}, one has $h^1(S,\cO_S(D))=0$ whenever $-\ell_D\in L_{> 0}$. 
The rest of this section will be devoted to the special case $-\ell_D\in L_{\geq 0}$ and either $\h=0$ or $\sv=0$.

The simplest case, which we want to exclude, is $D\sim 0$, 
for which $h^0(S,\cO_S(D))=1$ and $h^1(S,\cO_S(D))=h^2(S,\cO_S(D))=0$.
Note that in this case $\h=\sv=0$.

Let us first consider the case $\sv = 0$, $\h \leq 0$.
Using the canonical form of $D$ and using $\sv=0$, one has 
\begin{equation}\label{eq:D0<}
D\sim \sum_{i=1}^r \hat{a}_i A_i + \hat{\sf} F.
\end{equation}
One has the following explicit formula for $h^1(S,\cO_S(D))$.

\begin{prop}\label{prop:h10<}
Let $D$ be 
such that
$-\ell_D=(-\h,0)\in L_{\geq 0}$ and $D\not\sim 0$.
Then, 
\[
h^0(S,\cO_S(D))=h^2(S,\cO_S(D))=0
\]
and 
\begin{equation}\label{eq:axis1}
h^1(S,\cO_S(D))
= -1 - \sum_{i=1}^r \left\lfloor \frac{\hat{a}_i}{d_i} \right\rfloor - \hat{\sf}.
\end{equation}
In particular, if $\ell_D=(0,0)$, then
\begin{equation}\label{eq:originNot0}
h^1(S,\cO_S(D))
=-1 + \sum_{i=1}^r \left\{ \frac{\hat{a}_i}{d_i} \right\}.
\end{equation}
\end{prop}

\begin{proof}
By Corollary~\ref{cor:h2-OSD}, $h^2(S,\cO_S(D))=0$. 
On the other hand, $\sv=0$ together with Theorem~\ref{thm-h0} implies 
$h^0(S,\cO_S(D)) = \max \{ b_0(D), 0 \}$, where
\[
\ZZ \ni b_0(D)
= 1+\h-\sum_{i=1}^r \left\{ \frac{\hat{a}_i}{d_i} \right\}
= 1+\hat{\sf}+\sum_{i=1}^r \left\lfloor \frac{\hat{a}_i}{d_i} \right\rfloor.
\]
Since $\h\leq 0$, $D\not\sim 0$, and $b_0(D)\in\ZZ$, one deduces that $b_0(D)\leq 0$ 
and thus $h^0(S,\cO_S(D))=0$. 

Recall that $\chi(S,\cO_S(D)) = b_0(D)$ and the equality for $h^1(S,\cO_S(D))$ holds.
\end{proof}

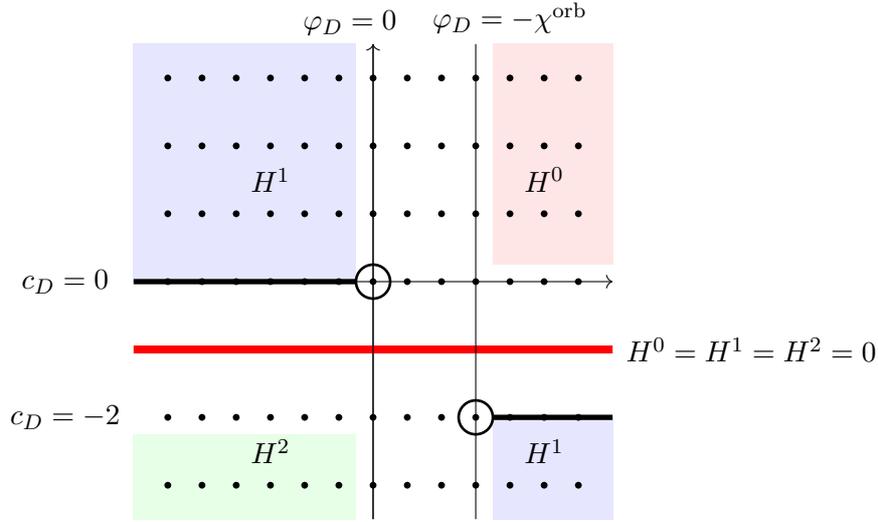
\begin{figure}[ht]
\begin{center}
﻿\begin{tikzpicture}[scale=.9]
\fill[color=red!10!white] (1.75,3.5)--(1.75,.25)--(3.5,.25)--(3.5,3.5)--cycle;
\fill[color=green!10!white] (-3.5,-2.25)--(-.25,-2.25)--(-.25,-3.5)--(-3.5,-3.5)--cycle;
\fill[color=blue!10!white] (1.75,-3.5)--(1.75,-2)--(3.5,-2)--(3.5,-3.5)--cycle;
\fill[color=blue!10!white] (-3.5,3.5)--(-3.5,0)--(-.25,0)--(-.25,3.5)--cycle;
\foreach \a in {-3,-2.5,...,3}
{
\foreach \b in {-3,-2,...,3}
{
\fill[black] (\a,\b) circle [radius=.05cm];
}
}
\draw[->] (-3.5,0)--(3.5,0) ;
\draw[->]  (0,-3.5)--(0,3.5);
\draw[line width=1] circle [radius=.25];
\draw[line width=1] (1.5,-2) circle [radius=.25];
\draw[line width=2] (1.75,-2)--(3.5,-2);
\draw[line width=2] (-3.5,0)--(-.25,0);
\draw[line width=3,color=red] (-3.5,-1)--(3.5,-1) node[right,color=black] {$H^0=H^1=H^2=0$};
\node at (2.5,1.5) {$H^0$};
\node at (-1.5,-2.5) {$H^2$};
\node at (2.5,-2.5) {$H^1$};
\node at (-1.5,1.5) {$H^1$};
\node at (-4.5,-2) {$\sv=-2$};
\node at (-4.5,0) {$\sv=0$};
\draw (1.5,-3.5) -- (1.5,3.5) node[above=10,right=-20] {$\h=-\chi^{\textrm{orb}}$};
\draw (0,-3.5) -- (0,3.5) node[above=8,right=-30] {$\h=0$};
\end{tikzpicture}
\caption{Cohomology concentrated in a single degree.}
\label{fig:H012a}
\end{center}
\end{figure}

Consider now the case $D\not\sim 0$ and $-\ell_D=(0,-\sv)\in L_{\geq 0}$.
Note that $h^0(S,\cO_S(D)) = 0$ and, by Theorem~\ref{thm:chi-OSD} and Corollary~\ref{cor:h2-OSD},
\[
h^1(S,\cO_S(D))
= \sum_{j=0}^{-(2+\sv)} \Big( -b_j(K_S-D) + \max\{ b_j(K_S-D), 0 \} \Big),
\]
where in this case,
\begin{equation}\label{eq:bound-bjKS-D}
\begin{aligned}
b_j (K_S-D) &= 1- \chi^{\textrm{orb}}- \sum_{i=1}^r \left\{ \frac{-1-a_i + (c+j+1)q_i}{d_i} \right\}\\
&\geq 1- \chi^{\textrm{orb}}- \sum_{i=1}^r \left( \frac{d_i-1}{d_i} \right)=1- \chi^{\textrm{orb}} - r + \sum_{i=1}^r \frac{1}{d_i}=-1.
\end{aligned}
\end{equation}
Hence
\begin{equation}\label{eq:h1-axis2}
h^1(S,\cO_S(D))
= \# \left\{ j \in \{0,1,\ldots,-(2+\sv)\} \mid b_j(K_S-D) = -1 \right\}.
\end{equation}

\begin{prop}\label{lemma:h1_horizontal}
Let $D$ be a divisor in $S$ such that $\ell_D=(0,\sv)$ with $\sv < 0 $.

\begin{enumerate}[label=\rm(\roman{enumi})]
\item\label{lemma:h1_horizontal_i} If $D\notin\cl_H(S):=\ZZ\langle C,E\rangle$, then $H^1(S,\cO_S(D)) = 0$.
\item\label{lemma:h1_horizontal_ii} If $D\in\cl_H(S)$, that is $D \sim c C+e E+gG$ 
for some $c,e,g\in \ZZ$ ($\sv=c+e+\vdeg g<0$), then
\[
h^1(S,\cO_S(D)) = - 1 - \left\lfloor \, \frac{c}{\vdeg} \, \right\rfloor - 
\left\lfloor \, \frac{e}{\vdeg} \, \right\rfloor - g.
\]
\end{enumerate}
\end{prop}

\begin{proof}[Proof of Proposition{\rm~\ref{lemma:h1_horizontal}}]
From Proposition~\ref{prop:h1}, if $\sv = -1$, then $h^i(S,\cO_S(D)) = 0$, $i=0,1,2$,
and the formula holds.

Assume $\sv \leq -2$, $\h = 0$, and $H^1(S,\cO_S(D)) \neq 0$.
According to~\eqref{eq:h1-axis2}, one has to study when $b_j(K_S-D)=-1$.
This happens precisely when the inequality~\eqref{eq:bound-bjKS-D} becomes
an equality
\begin{equation}
\label{eq:condition:h1}
\sum_{i=1}^r \left\{ \frac{-1-\hat{a}_i + (\sv+j+1)q_i}{d_i} \right\}
= \sum_{i=1}^r \frac{d_i-1}{d_i}
\end{equation}
which happens if and only if
\[
-1-\hat{a}_i + (\sv+j+1)q_i\equiv (d_i-1) \bmod d_i, 
\]
i.e.,
there exists a solution for the system~\eqref{eq:aimod}. Then, \ref{lemma:h1_horizontal_i} has been proven.

Let us prove~\ref{lemma:h1_horizontal_ii}. Since $G\sim \vdeg C$, it is enough to show the result for $g=0$.
Then $D\sim c C + e E$ and the condition given in~\eqref{eq:condition:h1} becomes
\[
\sum_{i=1}^r \left\{ \frac{-1 + (c+j+1)q_i}{d_i} \right\}
= \sum_{i=1}^r \frac{d_i-1}{d_i},
\]
which implies $j\equiv -(c+1) \bmod d_i$. In particular, there is a $0\leq j_0<\vdeg$, such that 
$j_0\equiv -(c+1) \bmod \vdeg$. Moreover,
\[
j_0 = -(c+1)-\vdeg\left\lfloor \frac{-c-1}{\vdeg} \right\rfloor.
\]
The system~\eqref{eq:aimod} has $\ell:=h^1(S,\cO_S(D))\geq 0$ solutions in $\{0,1,\ldots,-(\sv+2)\}$,
namely
\[
j_0+(\ell-1)\vdeg\leq -(\sv+2)< j_0+\ell \vdeg.
\]
This way $\ell-1$ can be described as an integer satisfying
\[
\frac{-(\sv+2)-j_0}{\vdeg} - 1 < \ell - 1 \leq \frac{-(\sv+2)-j_0}{\vdeg}.
\]
In other words $\ell = 1 + \left\lfloor \frac{-(\sv+2)-j_0}{\vdeg} \right\rfloor$ and then
\begin{equation}\label{eq:h1-axis2-special}
\begin{aligned}
h^1(S,\cO_S(D)) &= 1+ \left\lfloor \frac{-(\sv+2)-j_0}{\vdeg} \right\rfloor
= 1 + \left\lfloor \frac{-e-1}{\vdeg} + \left\lfloor \frac{-c-1}{\vdeg} \right\rfloor \right\rfloor \\
&= 1 + \left\lfloor \frac{-e-1}{\vdeg} \right\rfloor + \left\lfloor \frac{-c-1}{\vdeg} \right\rfloor
= - 1 - \left\lfloor \, \frac{e}{\vdeg} \, \right\rfloor - \left\lfloor \, \frac{c}{\vdeg} \, \right\rfloor.
\end{aligned}
\end{equation}
The last equality follows from~\eqref{eq:floor}.
\end{proof}

To summarize, Theorems~\ref{thm-h0} and \ref{thm:chi-OSD} together with Serre's duality show that 
the Betti numbers $h^i(S,\cO_S(D))$ are determined by the divisor class of~$D$.

In addition, there are four translated cones in the divisor lattice $L$, as shown in Figure~\ref{fig:H012a}, where the 
cohomology is concentrated in a single degree. Moreover, the dimensions $h^i(S,\cO_S(D))$ are in fact determined 
by the image $\ell_D$ in one of these translated cones.

The remaining threshold area is special in two directions. First, Betti numbers are not necessarily
concentrated in a single degree anymore, and second, they need not be determined by~$\ell_D$.

To end this section, an example of the special behavior on the threshold area is provided.

\begin{ex}
\label{ex:torsion}
Let us consider the surface $S$ associated with $(d_i)=(3,3,3,3)$ and $(q_i)=(1,2,1,2)$. In this case
$\alpha=2$, $\chi^{\textrm{orb}}=-\frac{2}{3}$, and $\vdeg=3$. According to Proposition~\ref{prop:class},
$\cl(S) \simeq \ZZ^2 \times (\ZZ/3\ZZ)^3$. The free part is generated by the classes of $C$ and $A_4$
and the torsion part by the classes of $A_1-A_4$, $A_2-A_4$, $A_3-A_4$.
Let
\[
D_1= A_1,\quad D_2=2A_1-A_2,\quad D_3=2 A_1-A_2+A_3-A_4.
\]
Note that $\ell_{D_i}=\ell_{3C+D_i}=(\frac{1}{3},0)$ is in the threshold area since 
$0<\varphi_{D_i}<-\chi^{\textrm{orb}}=\frac{2}{3}$. 
The following table can be obtained using Theorems~\ref{thm-h0} and \ref{thm:chi-OSD} together with Serre's duality.
It shows that Betti numbers of $D$ are not determined by $\ell_D$.
\begin{table}[ht]
\begin{center}
\begin{center}
\begin{tabular}{|c|c|c|c|c|c|c|}\hline
 & $D_1$ & $D_2$ & $D_3$ & $3C+D_1$ & $3C+D_2$ & $3C+D_3$ \\\hline
$h^0$ & 1 & 0 & 0 & 2 & 0 & 1\\\hline
$h^1$ & 0 & 0 & 1 & 1 & 0 & 2\\\hline
$h^2$ & 0 & 0 & 0 & 0 & 0 & 0\\\hline
\end{tabular}
\vspace{5mm}
\caption{Betti numbers for $D_i$ and $3C+D_i$.}
\end{center}
\end{center}
\end{table}
\end{ex}

\section{\texorpdfstring{$H^1$}{H1}-eigenspace decomposition and characteristic polynomial of the monodromy of a ramified cyclic cover}
\label{sec:main}

Let us consider a ramified $d$-cyclic covering of a curve or a normal surface
and its monodromy $\s$. This monodromy acts on the first cohomology group
of the cover and its characteristic polynomial 
is called the \emph{characteristic polynomial} of the monodromy. 
For historical reasons we will also use the name \emph{Alexander polynomial} of the covering
for this characteristic polynomial.

Let  $\pi:\dcover\to\Si$ be the $d$-cyclic cover
ramified along a divisor $D\sim dH$.
Before we describe its characteristic polynomial, 
we will discuss two particular cases which will be essential for the general construction.
These special covers are called \emph{vertical} and \emph{horizontal}.

\subsection{Vertical coverings}\label{sec:vertical-coverings}
\mbox{}

Consider $D$ a \emph{vertical divisor}, that is, 
\begin{equation}
\label{eq:Dvertical}
D=\sum_{i=1}^r a_i A_i+\sum_{j=1}^s f_j F_j,
\end{equation}
where $F_j$ are generic fibers. Note that $D$ is set as a Weil divisor, not only as a divisor class.
Let $H$ be a divisor, such that $D\sim dH$ for some $d\in \ZZ_{>0}$.
The purpose of this section is to describe the eigenspace decomposition of $H^1(S_d,\CC)$ 
(also called the $H^1$-eigenspace decomposition of $S_d$) for $S_d$ the cyclic cover $\pi:S_d\to S$ of 
$S$ associated with $(d,D,H)$. In particular, the characteristic polynomial of the monodromy of the 
$d$-cyclic covering of $S$ coincides with that of $\pi_E:E_d:=S_d|_{\pi^{-1}(E)}\to E=\PP^1$, which 
is a $d$-cyclic covering of a rational curve as described in Proposition~\ref{prop:restriccion} and 
its preceding paragraph.

The following result will be proven.

\begin{prop}
\label{prop:vertical-covers-case}
Consider $\pi:S_d\to S$ the cyclic cover of $S$ associated with $(d,D,H)$ as described above. 
Then, the decomposition in invariant subspaces of the monodromy of the cover can be obtained by restricting 
the covering to a horizontal section of $\pi_S:S\to\PP^1$ such as $E$ (or $C$), that is, the decomposition of the cover 
$\pi_E:E_d\to E$ ($\pi_C:C_d\to C$) associated with $(d,D_E,H_E)$.
\end{prop}

\begin{proof}
Let us break the proof in several steps.

\begin{step}
\label{step1-vertical} Reduction to $(d,D,0)$.
\end{step}

By the discussion in section~\ref{sec:Esnault} one can replace $(d,D,H)$ by $(d,D-dH,0)$. Also, using the canonical form for $H$ given
in~\eqref{eq:canonical}, and since $dH\cdot F=F\cdot D=0$, one has $H\sim \sum \hat a_i A_i + \hat f F$. In particular, 
$D'=D-dH$ is a vertical divisor and hence it is enough to study the case $(d,D',0)$, where
\begin{equation}
\label{eq:Dprime-vertical}
D'=\sum_{i=1}^r a'_i A_i+\sum_{j=1}^s f'_j F_j.
\end{equation}

Note that the condition $D'\sim 0$ is equivalent to
\begin{subequations}
\label{eq:cond-vertical1}
\begin{align}
&\h=\sum_{i=1}^r\frac{a'_i}{d_i}+\sum_{j=1}^s f'_j=0 \qquad\text{ and } \label{eq:cond-vertical1a}\\
&a'_i\equiv 0 \ \bmod{d_i}, \qquad\qquad i=1,\dots,r.
\label{eq:cond-vertical1b}
\end{align}
\end{subequations}

\begin{step}
 Calculation of $h^1(S,\cO_S(L^{(\l)}))$.
\end{step}

To apply Esnault and Viehweg's method we consider the following divisors for $\l\in\ZZ$:
\[
L^{(\l)}=\sum_{i=1}^r \left\lfloor\frac{\l a'_i}{d}\right\rfloor A_i+
\sum_{j=1}^s \left\lfloor\frac{\l f'_j}{d}\right\rfloor F_j.
\]
Note that $L^{(\l)}\sim 0$ if and only if 
\begin{subequations}
\label{eq:cond-vertical2}
\begin{align}
&\varphi_{L^{(\l)}}=\sum_{i=1}^r \frac{1}{d_i}\left\lfloor\frac{\l a'_i}{d}\right\rfloor+
\sum_{j=1}^s \left\lfloor\frac{\l f'_j}{d}\right\rfloor=0 \qquad
\text{ and } \label{eq:cond-vertical2a} \\
&\left\lfloor\frac{\l a'_i}{d}\right\rfloor\equiv 0 \ \bmod{d_i}, \qquad\qquad\qquad i=1,\dots,r.
\label{eq:cond-vertical2b}
\end{align}
\end{subequations}
Subtracting~\eqref{eq:cond-vertical2a} from~$\frac{\l}{d}$\eqref{eq:cond-vertical1a} 
and using both $\left\{ x\right\}=x-\lfloor x \rfloor$ and 
$\left\{ x\right\}\geq 0$ one deduces
\begin{equation}
\array{lclcl}
\label{eq:cond-vertical3}
\displaystyle\sum_{i=1}^r \frac{1}{d_i}\left\{\frac{\l a'_i}{d}\right\}+
\sum_{j=1}^s \left\{\frac{\l f'_j}{d}\right\}=0 & \Leftrightarrow & 
\displaystyle\left\{\frac{\l a'_i}{d}\right\}=\left\{\frac{\l f'_j}{d}\right\}=0 & 
\Leftrightarrow & \displaystyle\frac{\l a'_i}{d},\frac{\l f'_j}{d}\in\ZZ
\\
\endarray
\end{equation}
for all $i=1,...,r$ and $j=1,...,s$. This is equivalent to $\frac{\l}{d_n}\in\ZZ$, where 
$d_n=\frac{d}{n}$, for $n:=\gcd(a'_1,...,a'_r,f'_1,...,f'_s,d)$.
Summarizing, if $\l$ is a multiple of $d_n$, then $h^1(S,\cO_S(L^{(\l)}))=0$. 
Otherwise, $\l\not\equiv 0\bmod{d_n}$, using~\eqref{eq:axis1}
\begin{equation}
\label{eq:h1vertical} 
h^1(S,\cO_S(L^{(\l)}))=-\sum_{i=1}^r \left\lfloor\frac{\left\lfloor\frac{\l a'_i}{d}\right\rfloor}{d_i}\right\rfloor-
\sum_{j=1}^s \left\lfloor\frac{\l f'_j}{d}\right\rfloor -1=
-\sum_{i=1}^r \left\lfloor\frac{\l a'_i}{dd_i}\right\rfloor-
\sum_{j=1}^s \left\lfloor\frac{\l f'_j}{d}\right\rfloor -1.
\end{equation}
where the last equality follows from~\eqref{eq:floorfloor}.

\begin{step}
Restriction to a 1-dimensional cover of $\PP^1$.
\end{step}

Let us restrict this cover to $E$ (for $C$ it works the same way). 
We will be using Proposition~\ref{prop:restriccion} and the construction before it.
Since $E$ contains $r$ singular points, say $\{\gamma_1,...,\gamma_r\}$ of the surface $S$, 
one needs to perform a $(1,q_i)$-weighted-blowup at each point $\gamma_i$ with exceptional component $E_i$. 
The total transform of $D$ is:
\[
\sum_{i=1}^r a'_i A_i+\sum_{i=1}^r \frac{a'_i}{d_i} E_i+\sum_{j=1}^s f'_j F_j;
\]
and hence $E_C=\sum_{i=1}^r\frac{a'_i}{d_i}\langle \gamma_i\rangle+ \sum_{j=1}^sf_j\langle P_j\rangle$,
where $P_j=E\cap F_j$ and $\frac{a'_i}{d_i}\in\ZZ$ by~\eqref{eq:cond-vertical1b}.

The decomposition in invariant subspaces of the monodromy of this cover is calculated in \S\ref{subsec:coverP1} and 
it matches that obtained in~\eqref{eq:h1vertical}.
\end{proof}

As a consequence of the proof and \eqref{eq:acampo}, one has the following specific formulas for the dimensions of the invariant subspaces and 
the Alexander polynomials of the monodromy.

\begin{cor}
\label{cor:Alexander-vertical}
Under the hypothesis of Proposition{\rm~\ref{prop:vertical-covers-case}}, after appropriate reduction to the case $(d,D',0)$ as 
in~\eqref{eq:Dvertical}, the dimension of the invariant subspaces is given by
\begin{equation}
\label{eq:cor-h1vertical} 
h^1(S,\cO_S(L^{(\l)}))=
-\sum_{i=1}^r \left\lfloor\frac{\l a'_i}{dd_i}\right\rfloor-
\sum_{j=1}^s \left\lfloor\frac{\l f'_j}{d}\right\rfloor -1.
\end{equation}
Moreover, the characteristic polynomial of the monodromy is
\begin{equation}
\label{eq:cor-Alexvertical} 
\Delta_1(t)=\frac{(t^n-1)^2(t^d-1)^{r+s-2}}{\prod_{i=1}^r (t^{\gcd(d,\frac{a'_i}{d_i})}-1)\prod_{j=1}^s (t^{\gcd(d,f'_j)}-1)}.
\end{equation}
\end{cor}

\subsection{Horizontal coverings}\label{sec:horizontal-coverings}
\mbox{}

The second type of special cyclic covers of $S$ are those associated with $(d,D,H)$ where $D$ is a \emph{horizontal divisor}, that is,
\begin{equation}\label{eq:horizontal}
D=cC+eE+\sum_{j=1}^s g_j G_j\sim dH,
\end{equation}
where $G_1,\dots,G_s$ are distinct fibers of $\pi_G$ all of them linearly equivalent to~$G$, see Remark~\ref{rem:fibr2}. 

As was discussed in section~\ref{sec:vertical-coverings}, it is enough to consider the case $(d,D-dH,0)$. However,
$D-dH$ is not necessarily a horizontal divisor, as described in~\eqref{eq:horizontal-seq}. This subtlety makes this case a bit more
involved than the vertical case.

The best reduction one can expect is given by the following result.

\begin{lemma}
\label{lem:Hhorizontal}
There exist $\gamma,\eta\in\ZZ$ and a divisor
$T'\in\tor\cl(S)$ such that 
\[
H\sim \gamma C+ \eta E+T';
\]
the order of $T'$ is a divisor of $\gcd(d,\vdeg)$
and the only common multiple of $T'$ and $T$ is $0$
(see Remark{\rm~\ref{rem:divT}} for the definition of~$T$).

Moreover, there exist integers $c',e'$ such that 
\[
e=d\eta+\vdeg e',\quad 
c=d\gamma+\vdeg c',\quad
c'+e'+\sum_{j=1}^s g_j=0.
\]
\end{lemma}

\begin{proof}
Since $\gcd(q_i,d_i)=1$, we can assume that 
\[
H\sim \gamma' C+ \eta' E+ \sum_{i=1}^r q_i\alpha'_i A_i+\phi' F;
\]
recall that $C,E,F, A_i$ form a generator system of $\cl(S)$. Moreover, since
$d_i A_i\sim F$ and $q_i,d_i$ are coprime, such an expression exists.
The condition $D\sim d H$ is equivalent to
\begin{equation*}%
\sum_{i=1}^r\frac{q_i\alpha'_i}{d_i}+\phi'=0,\quad
c+e+\vdeg\sum_{j=1}^s g_j=d(\gamma'+\eta'),\quad 
e\equiv d(\eta'+\alpha'_i)\bmod{d_i}.
\end{equation*}
In particular, $\gcd(d,d_i)$
divides~$e$, i.e.,
\[
\lcm(\gcd(d,d_1),\dots,\gcd(d,d_r))=\gcd(d,\lcm(d_1,\dots,d_r))=
\gcd(d,\vdeg)
\]
also divides~$e$. Note that $\gcd(d,\vdeg)$ divides~$c$.

Let $\eta_0$ be a solution of $e\equiv d\eta'\bmod{\vdeg}$; the solutions
of this equation are $\eta_h:=\eta_0+h \vdeg_1$ 
where $\vdeg_1:=\frac{\vdeg}{\gcd(d,\vdeg)}$ and $h\in\ZZ$. 
Let
\[
\gamma_h:=\frac{c+e+\vdeg\sum_{j=1}^s g_j}{d}-\eta_h\in\ZZ.
\]
Let $H_h:=\gamma_h C+\eta_h E=H_0+h \vdeg_1 T$.
Note that
\[
\begin{aligned}
D-dH_h &\sim 
(c-d\gamma_h)C+(e-d\eta_h)E+\sum_{j=1}^s g_j G_j \\
&\sim \left(c-d\gamma_h+\vdeg\sum_{j=1}^s g_j\right)C+(e-d\eta_h)E
\sim (e-d\eta_h) T\sim 0.
\end{aligned}
\]
Then $T_h:=H_h-H=T_0+h\vdeg_1 T$ defines a torsion class such that $d T_h\sim 0$.
Since the maximal order of torsion classes is~$\vdeg$ we deduce
that  $\gcd(\vdeg,d)T_h=0$.

Using the structure of $\cl(S)$ given in Proposition~\ref{prop:class},
let us fix a direct-sum decomposition where
the component of $\ZZ/m_{r-1}=\ZZ/\vdeg$ is generated by $T$. The
coordinate $\beta_h\bmod{\vdeg}$ of $T_h$ in this component must 
satisfy 
\[
\gcd(\vdeg,d)\beta_h\equiv0\bmod{\vdeg}\Longleftrightarrow
\beta_h\equiv0\bmod{\vdeg_1}\Longleftrightarrow
\beta_h=\hat{\beta}_h\vdeg_1,
\]
and note that $\hat{\beta}_h\equiv\hat{\beta}_0+h\bmod{\vdeg_1}$.
Hence for a suitable~$h$ the coordinate $\hat{\beta}_h$ of $H_h$ in $T$
vanishes  in $\ZZ/\vdeg$. 

Let us denote $\gamma=\gamma_h$, $\eta=\eta_h$, and $T'=T_h$.
They are the values in the statement (in fact, those values may not be unique but we do not claim that).

The choice of $H_0$ is \emph{well defined} up to congruence 
with $T_1=\vdeg_1 T$, with order exactly $\gcd(\vdeg,d)$. 
The definition of $\gamma,\eta$ as solutions of a congruence
equation ends the proof.
\end{proof}

\begin{red}
\label{red:horizontal}
As an immediate consequence of Lemma~\ref{lem:Hhorizontal}, considering 
$D'=D-d(\gamma C+ \eta E)$, one can reduce the general case of horizontal coverings to those associated with $(d,D',T')$, where
\begin{equation}\label{eq:reduction}
D'=\vdeg c'C+\vdeg e'E+\sum_{j=1}^s g_j G_j\sim 0,\quad 
T'\sim\sum_{i=1}^r \alpha_iA_i.
\end{equation}
The class $T'$ is torsion of order~$d_\t$ and its only common multiple with $T$ is~$0$.
For convenience, we denote $\tau:=\frac{d}{d_\tau}\in\ZZ$.
\end{red}

These integers $c',e'$ are particularly important, since they provide an interesting feature of the restriction of the covering to
the preimage $\fd$ of a generic fiber~$F$. Namely, this cover~$\pi_F$ ramifies at $\vdeg s+2$ points with ramification indices
$\vdeg c$ (at $P_C=C|_F$), $\vdeg e$ (at $P_E=C|_F$), and $g_j$, $j=1,...,s$ (at each of the $\vdeg$ points $G_j|_F=P_{j1}+...+P_{j\vdeg}$).
\begin{equation}\label{pull-back-F-1}
\begin{tikzcd}
\fd\ar[d,"\pi_F"]\ar[r]&X\ar[d,"\pi_X"]\\
F\ar[r]&\PP^1\\[-25pt]
z\ar[r,mapsto]&z^\vdeg
\end{tikzcd}
\end{equation}
Actually this cover is the pull-back of a cover $\pi_X$ with ramification
indices $c$ (at~$z=0$), $e$ (at $z=\infty$), and ramifications $g_j$ at $s$ points of $\CC^*$. 

\begin{defi}
The $d$-covering $\pi_X:X\to\PP^1$ of \eqref{pull-back-F-1} is called the \emph{primitive vertical cover} of~$\pi$. 
\end{defi}

The \emph{vertical} coverings are the restrictions of the covering~$\pi$ to the preimage of the fibers. 
The above covering is called \emph{primitive} because $\pi_F$
can be retrieved from $\pi_X$ due to \eqref{eq:reduction} and how each $G_j$
intersects the fibers at $\vdeg$ points.

This behavior is repeated for each special fiber $A_i$, replacing $\vdeg$ by $\frac{\vdeg}{d_i}$ 
and taking into account that the factorization may not work for the restrictions 
$\pi_{A_i}:\ai{d}\to A_i$ but only for intermediate covers  $\pi_{A_i,d'}:\ai{d'}\to A_i$,
where $d'$ is a divisor of~$d$. For each~$i$, we set $e_i$ as the maximal divisor of~$d$
such that the following diagram holds:
\begin{equation}\label{pull-back-Ai-1}
\begin{tikzcd}
\ai{e_i}\ar[d,"\pi_{A_i,e_i}"]\ar[r]&X_{e_i}\ar[d,"\pi_{X,e_i}"]\\
A_i\ar[r]&\PP^1\\[-25pt]
z\ar[r,mapsto]&z^{\frac{\vdeg}{d_i}}.
\end{tikzcd}
\end{equation}

\begin{defi}\label{def:gcd-cover}
Let $\hat{e}:=\gcd(e_1,\dots,e_r)$. 
The $\hat{e}$-covering $\pi_{X,\hat{e}}:X_{\hat{e}}\to\PP^1$ is called the \emph{greatest common vertical cover} of~$\pi$. 
\end{defi}

In order to apply Esnault and Viehweg's method one has to consider the following divisors for $\l=0,...,d-1$.
The divisor for Esnault and Viehweg's method is
\begin{align*}
L^{(\l)}=&
\left\lfloor\frac{\l \vdeg c'}{d}\right\rfloor C+\left\lfloor\frac{\l \vdeg e'}{d}\right\rfloor E+
\sum_{j=1}^s \left\lfloor\frac{\l g_j}{d}\right\rfloor G_j -\l T'\\
\sim &
\left(\left\lfloor\frac{\l \vdeg c'}{d}\right\rfloor+
\left\lfloor\frac{\l \vdeg e'}{d}\right\rfloor+
\vdeg\sum_{j=1}^s \left\lfloor\frac{\l g_j}{d}\right\rfloor\right) C+
\left\lfloor\frac{\l \vdeg e'}{d}\right\rfloor T-\l T'.
\end{align*}
From Lemma~\ref{lem:Hhorizontal}, one has
\[
\l \vdeg c' + \l \vdeg e' 
+\vdeg\sum_{j=1}^s \l g_j=
\l\vdeg\left(c' +  e' 
+\sum_{j=1}^s g_j\right)
= 0
\]
and hence 
\[
\left\{\frac{\l \vdeg c}{d}\right\} +
\left\{\frac{\l \vdeg e}{d}\right\} 
+\vdeg\sum_{j=1}^s \left\{\frac{\l g_j}{d}\right\}\in\ZZ.
\]
Let 
\begin{equation*}
\tilde{L}^{(\l)}:=
-\left(\left\{\frac{\l \vdeg c'}{d}\right\} +
\left\{\frac{\l \vdeg e'}{d}\right\} 
+\vdeg\sum_{j=1}^s \left\{\frac{\l g_j}{d}\right\}\right) C
+ \left\lfloor \frac{\l \vdeg e'}{d}\right\rfloor T - \l T'.
\end{equation*}

\begin{equation}
\label{eq:L-horizontal}
\begin{aligned}
L^{(\l)}-\tilde{L}^{(\l)}\sim&
\left(\left\lfloor\frac{\l \vdeg c'}{d}\right\rfloor+
\left\{\frac{\l \vdeg c'}{d}\right\} +
\left\lfloor\frac{\l \vdeg e'}{d}\right\rfloor+
\left\{\frac{\l \vdeg e'}{d}\right\} +
\vdeg\sum_{j=1}^s \left(
\left\lfloor\frac{\l g_j}{d}\right\rfloor
+\left\{\frac{\l g_j}{d}\right\}
\right)\right) C\\
\sim&\
\frac{\l\vdeg}{d}\left( c'+ e' +
\sum_{j=1}^s g_j
\right) C=0.
\end{aligned}
\end{equation}

\begin{remark}
The common divisor 
\[
n:=\gcd(d,\vdeg c,\vdeg e,g_1,\dots,g_s)
\]
of the coefficients of $D'$ and the degree of the cover will be useful 
as well as $d_n:=\frac{d}{n}$. 
Note that $n$ is also the greatest common divisor of 
$d$ and the coefficients of $D$ before the reduction.
\end{remark}

As a first approach, we will consider the simpler case $(d,D',0)$ for a horizontal divisor $D'\sim 0$, that is, $T'\sim 0$. 
In this case, one obtains the following result.

\begin{prop}
\label{prop:purehorizontal}
Let $\pi:S_d\to S$ be a horizontal cyclic cover of $S$ associated with $(d,D',0)$. 
Then its $H^1$-eigenspace decomposition can be described as a direct sum $\Hh\oplus \Hm$, 
where $\Hm$ comes from a vertical cover of type $(n,-\vdeg eT,0)$
and $\Hh$ comes from the \emph{greatest common vertical cover} which is of degree~$d$.
Moreover, for any $\l=1,\dots,d-1$,
\begin{equation}\label{eq:h1_pure-hor}
h^1(S,\cO_S(L^{(\l)}))=
\begin{cases}
\displaystyle -1+\left\{\frac{\l c}{d}\right\}+
\left\{\frac{\l e}{d}\right\}+
\sum_{j=1}^s \left\{\frac{\l g_j}{d}\right\}
& \textrm{if } d_n\nmid\l, \\
\displaystyle -1+\sum_{i=1}^r\left\{-\frac{\l_2 \vdeg eq_i}{nd_i}\right\} & \textrm{if } \l=\l_2d_n.
\end{cases}
\end{equation}
\end{prop}

\begin{proof}
By~\eqref{eq:L-horizontal}, one obtains
\begin{equation}
\label{eq:cL-horizontal}
c_{L^{(\l)}}=L^{(\l)}\cdot F=-\left( \left\{ \frac{\l \vdeg c}{d}\right\}+\left\{ \frac{\l \vdeg e}{d}\right\}+
\vdeg\sum_{j=1}^s \left\{ \frac{\l g_j}{d}\right\}\right)\leq 0.
\end{equation}
Note that $c_{L^{(\l)}}=0$ exactly when $d_n|\l$, say $\l=\l_2d_n$ for some $\l_2\in\ZZ$, 
in which case $L^{(\l)}\sim\frac{\vdeg e}{n}\l_2T$. 
This falls in case~\eqref{eq:originNot0} of Proposition~\ref{prop:h10<} 
for $\hat{a}_i=\frac{\vdeg e q_i}{n}\l_2$ and one can check that 
the second part of formula~\eqref{eq:h1_pure-hor} follows.
For these terms one has
\[
\sum_{\l_2=0}^{n-1}h^1(S,\cO_S(L^{(\l_2d_n)}))=h^1(S_n,\cO_{S_n}),
\]
where $S_n$ is a vertical cover of $S$ associated with $(n,0,\frac{\vdeg e}{n}T)$ or equivalently, with $(n,-\vdeg eT,0)$. 
This is the \emph{vertical} cover producing~$\Hm$.

For the remaining terms $c_{L^{(\l)}}<0$ and hence, by the vanishing result in Proposition~\ref{lemma:h1_horizontal}
and~\eqref{eq:floorfloor}, one has
\begin{equation}
\label{eq:h1L-2}
h^1(S,\cO_S(L^{(\l)}))=-1-\left\lfloor \frac{\l c}{d}\right\rfloor-\left\lfloor \frac{\l e}{d}\right\rfloor
-\sum_{j=1}^s \left\lfloor \frac{\l g_j}{d}\right\rfloor
=-1+\left\{ \frac{\l c}{d}\right\}+\left\{ \frac{\l e}{d}\right\}+\sum_{j=1}^s \left\{ \frac{\l g_j}{d}\right\},
\end{equation}
where the last equality follows from $\lfloor x\rfloor=x-\{x\}$ and $c+e+\sum g_j=0$. This proves the first part of~\eqref{eq:h1_pure-hor}.

Now, let us describe the restrictions of the original $d$-covering to the curves $A_i$. Recall that 
the $d$-cover $\pi_F:\fd\to F\cong\PP^1$ is the pull-back of a cover $\pi_X:X\to\PP^1$ 
by the cyclic cover $z\mapsto z^\vdeg$.

In order to describe the restriction to $A_i$ one needs to perform a blow-up at the singular points of $S$ on $A_i$ as explained 
in Proposition~\ref{prop:restriccion} and the preceding paragraph. In particular, it is enough to perform a 
$(q'_i,1)$-weighted (resp.~$(d_i-q'_i,1)$) blow-up at $A_i\cap E$ (resp.~$A_i\cap C$), where 
$q_i q'_i\equiv 1 \bmod{d_i}$.

In addition, note that $D'\sim 0$ and horizontal
implies that the multiplicity of the exceptional component $E_i$ (resp.~$C_i$) of the blowing-up 
of $A_i\cap E$ (resp.~$A_i\cap C$) is 
$\frac{\vdeg e}{d_i}$ (resp.~$\frac{\vdeg c}{d_i}$), i.e., the following diagram holds:
\begin{equation}\label{pull-back-Ai-1a}
\begin{tikzcd}
\ai{d}\ar[d,"\pi_{A_i}"]\ar[r]&X\ar[d,"\pi_X"]\\
A_i\ar[r]&\PP^1\\[-25pt]
z\ar[r,mapsto]&z^{\frac{\vdeg}{d_i}}
\end{tikzcd}
\end{equation}
As a consequence, the covering $\pi_X$ is the \emph{greatest common vertical covering} of $\pi$ (the divisor~$e$ of Definition~\ref{def:gcd-cover}
is exactly~$d$).

Denote by $L_\vdeg^{(\l)}$ the Esnault-Viehweg divisors in $\PP^1$
associated with $\pi_X$.
Note that 
\[
L_\vdeg^{(\l)}=\left\lfloor \frac{\l c}{d}\right\rfloor P_C+\left\lfloor \frac{\l e}{d}\right\rfloor P_E
+\sum_{j=1}^s \left\lfloor \frac{\l g_j}{d}\right\rfloor P_{G_j},
\]
where 
\[
\deg L_\vdeg^{(\l_2d_n)}=
-\left\{\frac{\l_2 c}{n}\right\}-
\left\{\frac{\l_2 e}{n}\right\}-
\sum_{j=1}^s \left\{\frac{\l_2 g_j}{n}\right\}=
-\left\{\frac{\l_2 c}{n}\right\}-\left\{\frac{\l_2 e}{n}\right\}=
\begin{cases}
0 & \textrm{ if } \frac{\l_2 c}{n}\in \ZZ,\\
-1 & \textrm{ otherwise. }
\end{cases}
\]
The last equality follows since $\frac{\l_2 c}{n}+\frac{\l_2 e}{n}=-\sum_j \frac{\l_2 g_j}{n}\in \ZZ$.
Hence $h^1(\PP^1,\cO_{\PP^1}(L_\vdeg^{(\l_2d_n)}))=0$ either way (see~\eqref{eq:h1-deg}). This shows that 
\[
h^1(X,\cO_X)=\sum_{\l=0}^{d-1} h^1(\PP^1,\cO_{\PP^1}(L_\vdeg^{(\l)}))=\sum_{\l=0, \, d_n\nmid\l}^{d-1} h^1(\PP^1,\cO_{\PP^1}(L^{(\l)})),
\]
which ends the proof.
\end{proof}

According to Reduction~\ref{red:horizontal}, the general horizontal case can be given by $(d,D',T')$, 
where $T'\in\tor\cl(S)$ as in Lemma~\ref{lem:Hhorizontal}. 
The $H^1$-eigenspace decomposition in this case splits into a purely horizontal and a vertical parts 
as follows.

\begin{prop}
\label{prop:horizontal}
Let $\pi:S_d\to S$ be the horizontal cyclic cover of $S$ associated with $(d,D',T')$. 
Then its $H^1$-eigenspace decomposition splits as $\Hh\oplus \Hv$, where $\Hv$ comes from 
the vertical cover associated with $(n,dT'-\vdeg eT,0)$ and $\Hh$ comes from the greatest common 
vertical cover of $\pi$ which is of degree~$\t$.
In particular,
it decomposes as a direct sum of the cohomology of two cyclic covers of~$\PP^1$ and the splitting respects the eigenspaces of the monodromy and the Hodge structure. 
\end{prop}

\begin{proof}
The proof runs along the lines of that of Proposition~\ref{prop:purehorizontal}. We will highlight the differences.
Formula~\eqref{eq:cL-horizontal} holds and in this case $c_{L^{(\l)}}=0$ implies
\[
L^{(\l)}\sim\l_2
\underbrace{\left(
\frac{\vdeg e}{n} T -d_n T'\right)}_{=:T''},\qquad 
L^{(\l+d_n)}\sim L^{(\l)}+T''.
\]
\begin{enumerate}[label=(L\arabic{enumi})]
 \item\label{l2} 
 The case $c_{L^{(\l)}}=0$ is equivalent to $\l=\l_2d_n$, i.e.~$L^{(\l)}\sim\l_2 T''$. 
 The value of $h^1(S,\cO_S(L^{(\l)}))$ has been computed in Proposition~\ref{prop:h10<}. 
 More precisely, this case corresponds with the vertical cover associated with 
 $(n,0,T'')$ (or equivalently $(n,dT'-\vdeg eT,0)$) considered in section~\ref{sec:vertical-coverings}.

 \item\label{l1} 
 If $c_{L^{(\l)}}<0$, then by the vanishing result in Proposition~\ref{lemma:h1_horizontal}, 
 $h^1(S,\cO_S(L^{(\l)}))\neq 0$ only if $\l T'\sim 0$. Hence, we assume $\l=\l_1d_\t$. 
 By Proposition~\ref{lemma:h1_horizontal} and \eqref{eq:floorfloor} one has
\begin{equation}\label{eq:h1_hor}
h^1(S,\cO_S(L^{(\l)}))=-1+
\left\{\frac{\l_1 c}{\t}\right\}+
\left\{\frac{\l_1 e}{\t}\right\}+
\sum_{j=1}^s \left\{\frac{\l_1 g_j}{\t}\right\}.
\end{equation}
Recall that the condition of Reduction~\ref{red:horizontal} implies $0\sim D'\sim dT'\sim \t (d_\t T')$ is a horizontal divisor.
\end{enumerate}
In order to describe the restriction to $A_i$ one needs to perform a blow-up at the singular points of $S$ on $A_i$ as explained 
in Proposition~\ref{prop:restriccion} and the preceding paragraph. In particular, it is enough to perform a 
$(q'_i,1)$-weighted (resp.~$(d_i-q'_i,1)$) blow-up at $A_i\cap E$ (resp.~$A_i\cap C$), where $q_i q'_i\equiv 1 \bmod{d_i}$
and hence one can find $h_i\in\ZZ$ such that $q_i q'_i=1+h_i d_i$.

In addition, note that $D'\sim dT'\sim 0$ implies the existence of integers $\delta_i$ such that $d\alpha_i=d_i\delta_i$
($\alpha_i$ is the coefficient of $T'$ in $A_i$).
Hence, the multiplicity of the exceptional component $E_i$ of the blowing-up of $A_i\cap E$ is
\[
\frac{\vdeg e+q'_i q_i d\alpha_i}{d_i}=e\frac{\vdeg}{d_i}+d h_i\alpha_i+\frac{d\alpha_i}{d_i}=
\frac{\vdeg}{d_i}e+d h_i\alpha_i+\delta_i\equiv \frac{\vdeg}{d_i}e+\delta_i\bmod{d}, 
\]
since the multiplicities are only relevant modulo $d$.
Analogously, one can compute the multiplicity of the exceptional component $C_i$ of the blowing-up of $A_i\cap C$ as~$\frac{\vdeg}{d_i}c
-\delta_i$. 

Summarizing, the restriction of the original $d$-covering to the curve $A_i$ is a $d$-cover of $A_i$ ramified at $\frac{\vdeg}{d_i}s+2$ points
with multiplicities $\frac{\vdeg}{d_i}e+\delta_i$ (at $P_{E_i}=A_i\cap{E_i}$), 
$\frac{\vdeg}{d_i}c-\delta_i$ (at $P_{C_i}=A_i\cap{C_i}$), and 
$g_j$, $j=1,...,s$ (at each of the $\frac{\vdeg}{d_i}$ points of $G_j\cap{A_i}$).

Let us consider the intermediate $\t$-cover.
Recall that $d_\t T'\sim 0$. There exists $\beta_i$ such that $d_\t\alpha_i=d_i\beta_i$:
\[
d_i\delta_i=d\alpha_i=\t(d_\t\alpha_i)=\t(d_i\beta_i) \ \Rightarrow \ \delta_i\equiv 0\bmod{\t}.
\]
These congruences show that the degree of the greatest common vertical cover is $\t$.
Its characteristic polynomial~$\Delta_{2,h}(t)$ can be computed with the help of the divisors $L_\t^{(\l_1)}$
using again 
Lemma~\ref{covering_proj_line}. The same argument
about $\deg L_\t^{(\l_1)}$ shows that the terms $\deg L_\t^{(\l_1)}$ for which $\l=\l_1\t$ is a multiple of $d_n$
do not contribute to this horizontal part.
This completes the proof.
\end{proof}

\begin{remark}
The degree of the greatest common vertical cover can be obtained in two ways. From the above proposition, it can be obtained algebraically in terms
of the torsion order of the divisor~$T'$. And from the definition, it can be obtained topologically as the highest divisor of~$d$ for which
the covers over $A_i$ are pull-back of the primitive vertical cover of~$\pi$.
\end{remark}

\begin{cor}
\label{cor-Alex-horizontal}
The characteristic polynomial $\Delta_{2}(t)$ of the monodromy of a horizontal cover of $S$ associated with $(d,D',T')$
as above, factorizes as $\Delta_{2}(t)=\Delta_{2,h}(t)\Delta_{2,m}(t)$, where $\Delta_{2,h}(t)$ and $\Delta_{2,m}(t)$ 
are Alexander polynomials of coverings of~$\PP^1$.
\end{cor}

\begin{proof}
By Proposition~\ref{prop:horizontal}, the $H^1$-eigenspace decomposition of the cover splits as a 
direct sum $\Hh\oplus \Hm$. By Proposition~\ref{prop:vertical-covers-case}, the characteristic 
polynomial $\Delta_{2,m}(t)$ of the monodromy associated with $\Hm$ corresponds with that of the 
restriction of the vertical cover $(n,0,T'')$ to $E$ (or $C$), whereas $\Delta_{2,h}(t)$ corresponds 
with that of the horizontal $\t$-cover described in the proof.
\end{proof}

\subsection{General case}\label{sec:general-coverings}
\mbox{}

To end this section, we are in the position to describe the general case, that is, 
$\pi: S_d \to S$ is a covering associated with $(d,D,H)$ such that $D \sim dH$ and
\[
D = \sum_{j \in J} m_j D_j \in \Div(S),
\qquad H = \gamma C + \eta E + \sum_{i=1}^r \alpha_i A_i 
\in \cl(S),
\]
where $m_j \in \ZZ_{>0}$, $D_j$ is an irreducible (effective) divisor, $j \in J$,
$\gamma, \eta, \alpha_i \in \ZZ$, $i=1,\ldots,r$ and the reduced support $D_{\text{red}}$ of $D$ is a
$\QQ$-simple normal crossing divisor.
Following Lemma~\ref{lem:hv} we decompose $D$ as 
\begin{equation}\label{eq:div_dec}
D=\mcdhv_h D_h+\mcdhv_v D_v+\mcdhv_s D_s
\end{equation}
where $D_h$ is horizontal, $D_v$ is vertical and $D_s$ is slanted (see Lemma~\ref{lem:hv}); 
all of them are primitive (i.e., the $\gcd$ of their multiplicities equals~$1$). We introduce the following notation:
\begin{equation*}
d^h:=\gcd(d,\mcdhv_v,\mcdhv_s),\qquad 
d^v:=\gcd(d,\mcdhv_h,\mcdhv_s).
\end{equation*}
The following result describes the $H^1$-eigenspace decomposition of a general cyclic cover of $S$. 
The notation used is defined in sections~\ref{sec:vertical-coverings} 
(see Proposition~\ref{prop:vertical-covers-case}) and~\ref{sec:horizontal-coverings}
(see Proposition~\ref{prop:horizontal}).

\begin{theorem}
\label{thm:gencase}
Consider $\pi:S_d\to S$ the cyclic cover of $S$ associated with $(d,D,H)$ as above. 
Then $H^1(S;\cO_S)=\Hh\oplus \Hv$,
$\Hv:=H^1(E_{d^v};\cO_{E_{d^v}})$ and $\Hh:=H^1(X_{e^h};\cO_{X_{e^h}})$, where $\pi_E:E_{d^v}\to E\cong\PP^1$ is the restriction of the intermediate 
$d^v$-cover to $E$ and $\pi_X:X_{e^h}\to\PP^1$ is the greatest common vertical cover of the intermediate $d^h$-cover ($e^h$ is a divisor 
of $d^h$ as in Proposition{\rm~\ref{prop:horizontal}}). The eigenspace decomposition of 
$H^1(S;\cO_S)=\Hh\oplus \Hv$ is the direct sum of the natural eigenspace decompositions of 
$\Hh$ and~$\Hv$. In particular, each factor of the splitting is associated with a $\PP^1$-cover.
\end{theorem}

\begin{proof}
The divisor $L^{(\l)}$ associated with Esnault and Viehweg's theory is
\[
L^{(\l)} = -\l H + \sum_{j \in J} \left\lfloor \frac{\l m_j}{d} \right\rfloor D_j.
\]
Let $B \in \Div(S)$ be any divisor. Since $D \sim dH$, the equality
\[
\l H \cdot B = \frac{\l}{d} \sum_{j \in J} m_j D_j \cdot B
\]
holds.
Taking into account \eqref{eq:int_form_Ll} in Remark~\ref{rem:Dtilde} one has 
\[
c_{L^{(\l)}}=L^{(\l)} \cdot F \leq 0\text{ and }
\varphi_{L^{(\l)}}=L^{(\l)} \cdot C \leq 0.
\]
According to Theorem~\ref{prop:h1},
if they are both negative, then $h^1(S,\cO_S(L^{(\l)}))=0$; the same happens if $\l$
is a multiple of~$d$. The remaining cases for $\l\in\ZZ$ are considered below.

\begin{enumerate}[label=\rm(\alph{enumi}),
itemindent=7mm,
leftmargin=0cm]
\item\label{casoa} $c_{L^{(\l)}}=0$.%

Let us denote by
$J_1 = \{ j \in J \mid c_{D_j}\neq 0 \}$, i.e., $D_j$ is a term of either $D_h$ or $D_s$. 
Then $c_{L^{(\l)}} = 0$ if and only if $d$ divides $\l m_j$, $\forall j \in J_1$, 
see Remark~\ref{rem:Dtilde}. The latter is equivalent to ask 
$
\frac{d}{d^v}$ to divide $\l$.
This way $\l$ can be written as $\l = \l_1 \frac{d}{d^v}$ for $\l_1 = 0,1,\ldots,d^v-1$. 
In fact, these values measure the action of the monodromy of the intermediate $d^v$-cover. 

Consider the cover of $S$ associated with $(d^v,\mcdhv_v D_v,H_v)$ where
\[
H_v \sim \frac{d}{d^v} H - \frac{\mcdhv_h}{d^v} D_h - \frac{\mcdhv_s}{d^v} D_s.
\]
This corresponds to a vertical cover.
Let us denote by $L_v^{(\l_1)}$ the Esnault-Viehweg divisors associated with this vertical cover.
A simple check shows that $L_v^{(\l_1)}=L^{\left(\l_1\frac{d}{d^v}\right)}$. Hence, this invariant part of the cohomology is described by 
a vertical cover of $d^v$
sheets in the sense of~\S\ref{sec:vertical-coverings}, which can be decomposed as 
the one of a cover of $\PP^1$ 
and has an associated characteristic polynomial~$\Delta_1(t)$.

\item\label{casob} $\varphi_{L^{(\l)}}=0$. 

Analogously to the previous case, one can define
$J_2 = \{ j \in J \mid \varphi_{D_j}\neq 0 \}$, i.e., $D_j$ is a term of either $D_v$ or $D_s$. 
Then $\varphi_{L^{(\l)}}=0$ if and only if
$\l$ can be written as $\l = \l_2 \frac{d}{d^h}$ for $\l_2 = 0,1,\ldots,d^h-1$.
These values determine
the action of the monodromy of the intermediate $d^h$-cover. 
Consider now the cover of $S$ associated with $(d^h, \mcdhv_h D_h,H_h)$, where
\[
H_v\sim \frac{d}{d^h} H-\frac{\mcdhv_v}{d^h} D_v-\frac{\mcdhv_s}{d^h} D_s.
\]
This is a horizontal cover. Analogously as in the previous case,
one can easily check that $L_h^{(\l_2)}=L^{\left(\l_2\frac{d}{d^h}\right)}$. Hence, this invariant part of the cohomology is described by
a horizontal cover of $d^h$ sheets in the sense of~\S\ref{sec:horizontal-coverings}, which can be decomposed as a 
direct sum of a horizontal cover (greatest common vertical cover) and a vertical cover, intermediate cover 
of order $\m_0:=\gcd(d,\mcdhv_h,\mcdhv_v,\mcdhv_s)$.

Let us denote by $\Delta_2(t)$ the characteristic polynomial
of the cover of $S$ associated with $(d^h, \mcdhv_h D_h,H_h)$,
by $\Delta_{2,h}(t),\Delta_{2,m}(t)$ the characteristic polynomials
of the horizontal and vertical covers 
of the preceding paragraph, see Corollary~\ref{cor-Alex-horizontal}.
Note that
$\Delta_2(t)=\Delta_{2,m}(t)\Delta_{2,h}(t)$.
\end{enumerate}

As word of caution, note that the previous two cases \ref{casoa} and \ref{casob} are not disjoint. 
That is why we need to consider a third case that accounts for repetitions.

\begin{enumerate}[label=\rm(\alph{enumi}),
itemindent=7mm,
leftmargin=0cm]
\setcounter{enumi}{2}
\item\label{casoc} $\ell_{L^{(\l)}}=(0,0)$, that is, $c_{L^{(\l)}}=\varphi_{L^{(\l)}}=0$.

Combining~\ref{casoa} and \ref{casob}, this case occurs whenever $\l=\l_0 \frac{d}{\m_0}$.
Consider the vertical cover of $S$ associated with $(\m_0,0,H_m)$, where
\begin{equation}
\label{eq:casoc}
D_m=0,\qquad 
H_m\sim \frac{d}{\m_0} H-\frac{1}{\m_0}D.
\end{equation}
Note that $H_m$ is a torsion class.
This case matches~\ref{l2} in Proposition~\ref{prop:horizontal}, which accounts for the vertical part $\Hv$ in the 
decomposition of the horizontal cover associated with part~\ref{casob}. In order to see this, it is enough to check 
that the divisors $L_m^{(\l_0)}:=L^{\left(\l_0\frac{d}{\mu_0}\right)}$ and $L_h^{(\l_0\frac{d^h}{\m_0})}$
are related
as follows 
\[
L_m^{(\l_0)}:=L^{(\l_0\frac{d}{\m_0})}=L^{(\l_0\frac{d^h}{\m_0}\frac{d}{d^h})}=L_h^{(\l_0\frac{d^h}{\m_0})}.
\]
In particular, $\Delta_0(t)=\Delta_{2,m}(t)$.\qedhere
\end{enumerate}
\end{proof}

As a result of the proof one obtains the following.

\begin{cor}
\label{cor-Alex-general}
The characteristic polynomial $\Delta(t)$ of the monodromy of a cover of $S$ associated with $(d,D,H)$
as above factorizes as $\Delta(t)=\frac{\Delta_1(t)\Delta_2(t)}{\Delta_0(t)}=\Delta_{1}(t)\Delta_{2,h}(t)$, where 
where $\Delta_1(t)$ and $\Delta_{2,h}(t)$ are the Alexander polynomials of covers of~$\PP^1$.
\end{cor}

\begin{proof}
As a consequence of the proof of Theorem~\ref{thm:gencase}, the final Alexander polynomial is 
$\frac{\Delta_1(t)\Delta_2(t)}{\Delta_0(t)}=\frac{\Delta_1(t)\Delta_{2,m}(t)\Delta_{2,h}(t)}{\Delta_0(t)}$.
At the end of the proof it is shown that $\Delta_0(t)=\Delta_{2,m}(t)$, which completes the proof.
\end{proof}

\section{Examples and applications}
\label{sec:ex}

\subsection{Reducible normal fake quadrics of type \texorpdfstring{$(d_i,q_i)=(3,1),r=3$}{(di,qi)=(3,1),r=3}}
\mbox{}\label{subsec:r3}

Following Remark-Definition~\ref{rem:def:S} consider $S$ the surface associated with $(d_i,q_i)=(3,1)$, $i=1,2,3$. 
In this case $\alpha:=\frac{1}{3}+\frac{1}{3}+\frac{1}{3}=1$ and $\chi^{\textrm{orb}}=0$ (see~\eqref{eq:beta}).
Note that $\vdeg:=\lcm(d_1,d_2,d_3)=3$ and $\cl(S)\cong\ZZ^2\oplus\left(\ZZ/3\right)^2$ by Proposition~\ref{prop:class}.

\begin{ex}\label{ej1}
This example will highlight the relevance of the choice of the torsion class $H$ in the cohomology of the 
covering of $S$ associated with $(d,D,H)$ as well as the importance of the \emph{greatest common  vertical cover} 
introduced in Definition~\ref{def:gcd-cover}. We present different horizontal coverings associated with 
different torsion divisors whose cohomology invariants are different as well as their greatest common  vertical covers.

Let $\sigma_a:\hat{S}_a\to S$ be the cover associated with $(3,D,H_a)$, where $D:=G$, $H_a:= C-aT$, $a=0,1,2$, 
(recall $T:=E-C$). Note that $\sigma_a$ is a horizontal cover. One has 
\[
L^{(\l)} =-\l H_a+\left\lfloor \frac{\l}{3}\right\rfloor G=-\l H_a=\l aE-\l(a+1)C.
\]
Applying Reduction~\ref{red:horizontal}, it is enough to consider the covering of $S$ associated with $(3,D'_a,\!0)$, where
\[
D'_a=D-3H_a=-3(a+1)C+3aE+G\sim 0.
\]
By Proposition~\ref{prop:purehorizontal} one has
\[
\begin{aligned}
h^1(S_a,\cO_{S_a}(L^{(\l)}))=&
-1+\left\{ -\frac{\l (a+1)}{3} \right\} + \left\{ \frac{\l a}{3} \right\} + \left\{\frac{\l}{3} \right\} \\
=&\left\lfloor -\frac{\l (a+1)-1}{3} \right\rfloor - \left\lfloor -\frac{\l a}{3} \right\rfloor 
=
\begin{cases}
1 & \textrm{ if } \l=2, a=1\\
0 & \textrm{ otherwise.}
\end{cases}
\end{aligned}
\]
As a consequence, 
\[
h^1(\hat S_a,\cO_{\hat S_a})=
\begin{cases}
1 & \textrm{ if } a=1\\ 0 & \textrm{ otherwise.}
\end{cases}
\]
Let us consider the composition $\pi$ of the $(1,2)$ weighted blow-ups of the points 
$A_i\cap C$
and the $(1,1)$ weighted blow-ups of $A_i\cap E$. We will denote by $C_i$ (resp. $E_i$) the 
exceptional component resulting after the blow-up of $A_i\cap C$ (resp. $E_i$).
Then,
\[
\pi^*(D'_a)=G-3(a+1)C+3aE+\sum_{i=1}^3(a E_i-(a+1)C_i).
\]
The restriction of the covering to a generic fiber $F$ ramifies
at the $3$ intersection points with $G$ and thus $\pi^{-1}(F)$ is a curve of genus~$1$.
The restriction of the covering to $A_i$ ramifies at the point $A_i\cap G$ with index~$1$.
The other ramification points are $A_i\cap C_i$ (with index $2-a$) and $A_i\cap E_i$ (with index $a$).

In particular, for $a=0,2$, the cover of $A_i$ is rational while for $a=1$ it is a curve of genus~$1$.
Note that $\t=3$, hence the three restrictions coincide with their \emph{greatest common  vertical covering}, 
and the monodromy of the covering of $S$ coincides with the monodromy of the covering of each~$A_i$.

On the other hand, if one replaces $H_a$ by $H'_{a}:=H_a+A_1-A_2$, then $\t'=1$ and the 
\emph{new greatest common  vertical covering} is the identity, i.e., $H^1(\hat{S}'_a,\cO_{\hat S'_a})=0$, 
where $\hat{S}'_a$ denotes the respective covering. Note that the restrictions of both coverings over $F$ do not change.
\end{ex}

\begin{ex}
This example shows the effect of the \emph{mixed} part, see \eqref{eq:casoc}, in a horizontal covering. 
It also shows an example of a non-connected covering.

Consider $\sigma:\hat{S}\to \Si$,  
where $\Si$ is as at the beginning of \S\ref{subsec:r3},
with $d=6$, $D=3(C+G)$, $H=2 C$. After applying Reduction~\ref{red:horizontal}
one has a horizontal covering of $S$ associated to $(6,D',0)$, where $D'=3(G-3 C)\sim 0$ is a horizontal divisor and $n=3$. 
The induced covering over a generic fiber has three connected components, each one being a genus-$1$-curve obtained 
as a double covering of~$\PP^1$ ramified at four points (the three intersections with $G$ and the intersection with~$C$). 
In order to study the covering over $A_i$ we need to use~$\pi$ as in Example~\ref{ej1}, to obtain
\[
\pi^*(D')=3\left(G-3C -\sum_{i=1}^3E_{P_i}\right),
\]
i.e., the coverings of $A_i$ consist of three copies of $\PP^1$. In fact, $\hat{S}$ has three components
with vanishing $1$-homology.

Replacing the $H$ divisor class by $H=2 E$, the situation over $F$ does not change. Applying Reduction~\ref{red:horizontal} one has
$D'=3(G+C-4 E)$, $n=3$, $d_n=2$, and 
\[
L^{(2)}=G+C-4 E\sim 4(C-E)=-4 T\sim 2 T,\quad L^{(4)}\sim T.
\] 
Hence the \emph{mixed} component of the covering corresponds with a $3$-cover ramified over $D=0$ and $H=-T$, i.e.,
of type $(3,0,2T)\equiv(3,-6T,0)$. 

Since $h^1(S,\cO_S(T))=0$, and $h^1(S,\cO_S(2T))=1$ (see Proposition~\ref{prop:h10<} formula~\eqref{eq:axis1}), 
which correspond to the cohomology of a genus-$1$-curve as $3$-cover of~$\PP^1$. In fact, 
\[
\pi^*(D')=3G+C-4E +\sum_{i=1}^3 E_{P_i}-4\sum_{i=1}^3 E_{Q_i},
\]
and then the covers over $A_i$ are $6$-covers of $\PP^1$ ramified along three points with ramification indices $3,1,2$, 
i.e., a curve of genus~$1$. The final characteristic polynomial is $(t^2-t+1)(t^2+t+1)$.

Finally, replacing the $H$ divisor by $H=2E+A_1-A_2$, one obtains 
$L^{(2)}=G+C-4 E-2 A_1+2 A_2\sim -T-2 A_1+2A_2\sim A_2- A_3$, $L^{(4)}\sim A_3- A_2$.

The mixed component, that is, the invariant part coming from the greatest common  vertical covering has trivial monodromy. 
In this case $\t=3$, and the horizontal part is a $2$-covering of $\PP^1$ ramified in principle at three points with 
ramification indices $(3,1,2)\equiv(1,1,0)$, i.e., the covering is a $\PP^1$ and thus its monodromy is trivial.
\end{ex}

\subsection{A general example of cyclic coverings of reducible normal fake quadrics}\label{subsec:full_ex}
\mbox{}

This will describe a general example of a cyclic covering of a reducible normal fake quadric. 
The constructive proof of Theorem~\ref{thm:gencase} will be applied to this case in order to 
explicitly exhibit the vertical and horizontal $\PP^1$-coverings whose $H^1$-eigenspace 
invariant subspaces reconstruct the first cohomology of the cyclic covering.

Let us consider the reducible normal fake quadric $S$ associated with the numerical data 
$(d_1,d_2,d_3)=(6,9,18)$ and $(q_1,q_2,q_3)=(5,1,1)$ as described in Remark-Definition~\ref{rem:def:S}. 
Note that $\alpha=1$, $\vdeg=18$ and $\chi^{\textrm{orb}}=-\frac{2}{3}$.

We consider the divisors
\[
\begin{aligned}
D:=&90(C+E)+15 G_1+165 G_2+18 (2 A_1+ A_2+ 2 A_3),\\
H:=&E+G_2+ 2 A_1- 3 A_2+ A_3.
\end{aligned}
\]
Since 
\[
D-180 H=90(C-E)+15(G_1-G_2)+18(-18 A_1+31 A_2-8 A_3)\sim(-54+62-8)F=0,
\]
the cyclic covering $\pi_{180}:X\to S$ associated with $(180,D,H)$ is well defined.

Following the proof of Theorem~\ref{thm:gencase} one has to determine the \emph{vertical} component of the covering,
which is given by $(\mu_1,D_v,H_v)$, where
\[
\mu_1=\gcd(180,90,90,15,165)=15,\quad
D_v=18(2 A_1+ A_2+2 A_3),
\]
and
\begin{gather*}
\quad 
H_v=12 H-6(C+E)-G_1-11 G_2\sim 6(E-C)+24 A_1-36 A_2+12 A_3\sim\\
6(5 A_1+ A_2+A_3-F)+24 A_1-36 A_2+12 A_3=54 A_1 -30 A_2+18 A_3-6F\sim
F-3 A_2.
\end{gather*}
Applying Reduction~\ref{red:horizontal}, %
this vertical covering is  associated with $(\mu_1,D'_v,0)$, where
\[
D'_v=18(2 A_1+ A_2 +2 A_3)-15 F+45 A_2=36 A_1+63 A_2+36 A_3 -15 F.
\]
Using Proposition~\ref{prop:vertical-covers-case} we see that the monodromy of the \emph{vertical cover} $\pi_{15}$ coincides with the monodromy of the covering 
restricted to the preimage of $E$, associated with the divisor
\[
E_C=6\langle\gamma_1\rangle+ 7\langle\gamma_2\rangle+2\langle\gamma_3\rangle-
15\langle\infty\rangle;
\]
the characteristic polynomial of the action of the monodromy
on the first cohomology group of the structure sheaf of the cover
is
\[
\prod_{j\in\{2,4,6,7,12,14\}}\left(t-\exp\frac{2\pi j\sqrt{-1}}{15}\right);
\]
in particular, the characteristic polynomial of the covering is the product
of the cyclotomic polynomials for $5$ and~$15$.
Analogously, for the horizontal part, the covering is associated with $(\mu_2,D_h,H_h)$, where
\[
\mu_2=\gcd(180,36,18,36)=18,\quad
D_h=90(C+E)+15 G_1+165 G_2,
\]
and
\[
H_h=10 H-(2 A_1+A_2+2A_3)\sim 10 (E+G_2) +18 A_1-31 A_2+8 A_3\sim 10 (E+G_2)+4(2 A_3-A_2),
\]
which is the expression as in Lemma~\ref{lem:Hhorizontal}, where $T'=4(2A_3-A_2)$ is a torsion divisor of 
order~$9$ whose least common multiple with $T=5A_1+A_2+A_3-F$ is $0$. According to Reduction~\ref{red:horizontal}
note that $d_\t=9$, $\t=2$, and 
\[
n=\gcd(18,90,15,165)=3.
\]
In other words, one is interested in studying the horizontal covering associated with $(18,D'_h,T')$, where
\begin{equation}
\label{eq:Dhex}
D'_h=D_h-18H_h=90 (C-E)+15(G_1-G_2).
\end{equation}
The $H^1$-eigenspaces of this covering, according to Proposition~\ref{prop:horizontal}, can be 
recovered from its \emph{mixed} part and its greatest common vertical covering. The mixed part has already 
been considered in the vertical component, that is, in $(\mu_1,D_v,H_v)$ above. As for the greatest 
common vertical $2$-fold covering of the restrictions to the special fibers~$A_i$. From~\eqref{eq:Dhex}, the double 
covering of $\PP^1$ is ramified in principle at four distinct points with indices 
$\left(\frac{90}{18},-\frac{90}{18},15,-15\right)$ which are congruent with $(1,1,1,1)$ $\mod{2}$. 
Then the characteristic polynomial of this part is $\Delta_{2,h}(t)=t+1$ (the cover is an elliptic curve).

\subsection{Application to the cohomology of surfaces which are quotient of product of curves}
\mbox{}\label{subsec:yomdin}

Let us consider a horizontal covering $\pi:X\to S$ of a reducible normal fake quadric $S$ associated with $(d,D,0)$,
where 
\[
D= \vdeg(c C+e E)+\sum_{j=1}^s g_j G_j,\quad\gcd(\vdeg c,\vdeg e,g_1,\dots,g_s)=1.
\]
The condition on the greatest common divisor implies that the covering $\pi_F:\fd\to F$, restriction of $\pi$ on the 
generic fiber $F$, is connected, where $\fd:=\pi^{-1}(F)$ is a smooth projective curve. Consider the following diagram:
\[
\begin{tikzcd}
Y\ar[r]\ar[d,"\tilde{\pi}"]&X\ar[d,"\pi"]\\
G\times\PP^1\ar[r,"\tau_2"]&S
\end{tikzcd}
\]
The covering $\tilde{\pi}$ is associated to a horizontal divisor of $G\times\PP^1$ where the sequence of ramification 
indices is $c$ (at each point in $\tau_2^{-1}(F\cap C)$), $e$ (at each point in $\tau_2^{-1}(F\cap E)$), and $g_i$ 
(at each point in $\tau_2^{-1}(F\cap G_i$)), $i=1,...,s$. 
Actually $Y=G\times \fd$ and $\tilde{\pi}=1_G\times\pi_F$. As a consequence, $X$ is the quotient of  $G\times \fd$ 
and its cohomology can be computed from the formulas in this work. The action on $G\times \fd$ is not free,
but $X$ admits structures of isotrivial fibrations over orbifolds.

\section{L\^e-Yomdin singularities}\label{sec:yomdin}

A useful application of cyclic covers of reducible normal fake quadrics is given for the semistable reduction 
of (weighted) L{\^e}-Yomdin surface singularities introduced in \cite{ABLM-milnor-number}. 
Let $(V,0):= \{F = 0\}\subset (\CC^3,0)$ be a hypersurface singularity where $F := f_m + f_{m+k} + ...$ is the 
weighted homogeneous decomposition of the analytic germ 
$F\in\CC\{x,y,z\}$ for some weights $\theta=(\theta_x,\theta_y,\theta_z)$. We say $(V,0)$ is a 
\emph{$(\theta, k)$-weighted L\^e-Yomdin singularity} if 
\[
V(\Jac(f_m)) \cap V(f_{m+k}) = \emptyset,
\]
where $V(\Jac(f_m)), V(f_{m+k})\subset\PP^2_\theta$ are algebraic varieties in the weighted projective plane 
$\PP^2_\theta=\Proj \CC[x,y,z]_\theta$, $\CC[x,y,z]_\theta$ is the $\theta$-graded polynomial algebra over $\CC$, 
and $\theta=(\deg x,\deg y,\deg z)=(\theta_x,\theta_y,\theta_z)$. %

The term L\^e-Yomdin singularities comes from the original papers by Yomdin~\cite{Yomdin74} and L\^e~\cite{Le80} on hypersurface singularities 
(see~\cite{ALM06} for more details). 
The Milnor fiber of a $(\theta, k)$-weighted L\^e-Yomdin singularity can be studied by means of Steenbrink's 
spectral sequence~\cite{Steenbrink77} associated with the semistable
reduction \cite{Mumford-topology,Steenbrink77,jorge-Semistable} of a $\mathbb{Q}$-resolution of singularities of 
$(V,0)$, see \cite{jorge-Israel} for the non-weighted case. 

This resolution is obtained with two types of blow-ups. The first one is the $\theta$-weighted blow-up of $(\CC^3,0)$.
This first exceptional divisor and its cyclic cover in the semistable reduction deserve special treatment and they are
completely determined by the projectivized tangent cone $V(f_m)$ as a curve in the resulting exceptional divisor, which 
is the weighted homogeneous plane $\PP^2_\theta$, see \cite[Proposition 5.13]{ACM21}. 

The other blow-ups are those required for the minimal $\mathbb{Q}$-embedded resolution of the pair $(\PP^2_\theta,V(f_m))$;
each weighted blow-up of this embedded resolution yields a weighted blow-up of the surface singularity as explained
in~\cite[\S6.2]{jorge-Semistable}. Each one of these blow-ups contributes with an exceptional divisor which produces 
(as cyclic covers) certain divisors in the semistable reduction. Most of the invariants appearing in the Steenbrink's 
spectral sequence are birational invariants. In particular, instead of studying the cyclic covers at the end of the 
resolution, one can study them at this first stage, where they are cyclic covers of quotients of weighted projective planes 
with a very precise ramification locus.

For the particular case of superisolated singularities, the strategy in~\cite{ea:mams} includes performing an additional
birational transformation such that the final result is a cyclic cover of $\PP^1\times\PP^1$ ramified along some fibers 
and sections. The parallel strategy for weighted L\^e-Yomdin singularities replaces $\PP^1\times\PP^1$ by a reducible 
normal fake quadric, the fibers by regular and exceptional fibers, and the sections by curves $C,E,G$, see~\cite{jorge:Yomdin}.
In the upcoming sections computations will be carried out for the particular case of the second type of weighted blowing-ups, 
namely where the blown-up point~$P$ is smooth in the partial resolution and where the weights $(\p,\q,\r)$ of $\omega$ 
are pairwise coprime.
Following~\cite[\S6.2]{jorge-Semistable}, an embedded $\QQ$-resolution of $V(f_m)$ will be described first.

\subsection{Semistable normalization in dimension 2}
\label{sec:SNdim2}
\mbox{}

In this section, one step of the $\QQ$-resolution of a point $V(f_m)$ will be described, namely, a $(\p,\q)$-blow-up 
of a smooth point $P_0$ in a $2$-dimensional chart where the local equation of the total transform of the singularity is 
\[
x^{k\mq_x} y^{k\mq_y} h(x,y)=0;
\] 
the $(\p,\q)$-multiplicity of $h(x,y)$ is denoted by $\pmult$. The condition on the exponents of 
$x,y$ is given by the smoothness in dimension~$3$. The $\nu$-form of $h(x,y)$ is 
\begin{equation}\label{eq:nu2}
\left(x^{a_x} y^{a_y}\prod_{i=1}^{r}(y^\p-\gamma_i x^\q)^{\xp_i}\right)^s,
\end{equation}
where  $\gamma_1,\dots,\gamma_{r}\in\CC^*$ are pairwise distinct, $\vdeg:=\frac{k}{s}\in\ZZ$, and 
$\gcd(\vdeg,a_x,a_y,e_1,\dots,e_{r})=1$. We denote $\xp:=\sum_{i=1}^{r} \xp_i$ and 
$\nu_0:= \p a_x+\q a_y+ \p \q \xp=\frac{\nu}{s}$. 
Moreover, this is related to the third weight of $\omega$ by $\r=\frac{\nu_0}{\vdeg}=\frac{\nu}{k}\in\ZZ$.

\begin{figure}[ht]
\begin{tikzpicture}
\coordinate (P) at (-2,0);
\coordinate (Q) at (2,0);

\draw ($1.5*(P)-.5*(Q)$) node[left] {$x_1=0$} --  ($1.5*(Q)-.5*(P)$) node[right] {$y_2=0$}  ;
\fill (P) circle [radius=.1cm];
\fill (Q) circle [radius=.1cm];
\draw (-1.25,1) to[out=-60,in=90] (-1,0) node[below] {$s \xp_1$}  to[out=90,in=-120] (-.75,1);
\draw[xshift=2cm] (-1.25,1) to[out=-60,in=90] (-1,0) node[below] {$s \xp_r$}  to[out=90,in=-120] (-.75,1);
\node at (0,.5) {$\dots$};
\node[below left] at (P) {$\frac{1}{\p}(1,-\q)$};
\node[below right] at (Q) {$\frac{1}{\q}(-\p,1)$};

\draw[rotate around={45:(-2,0)}, xshift=-1cm, ] (-1.25,1) node[above] {$s a_y$} to[out=-60,in=90] (-1,0)   to[out=90,in=-120] (-.75,1);
\draw[rotate around={-45:(2,0)}, xshift=3cm, ] (-1.25,1) node[right] {$s a_x$} to[out=-60,in=90] (-1,0)   to[out=90,in=-120] (-.75,1);

\draw ($(P)-.75*(0,1)$) node[below] {$y_1=0$} --  ($(P)+1*(0,1)$) node[above] {$k\mq_y$};
\draw ($(Q)-.75*(0,1)$) node[below] {$x_2=0$} --  ($(Q)+1*(0,1)$) node[above] {$k\mq_x$};
\node at ($(P)+1*(-4,1.5)$) {$[(x_1,y_1)]\mapsto(x_1^\p,x_1^\q y)$};
\node at ($(Q)+1*(2.5,1.5)$) {$[(x_2,y_2)]\mapsto(x_2 y_2^\p,y_2^\q)$};
\end{tikzpicture}
\caption{Weighted $(\p,\q)$ blow-up at $P_0$.}
\end{figure}
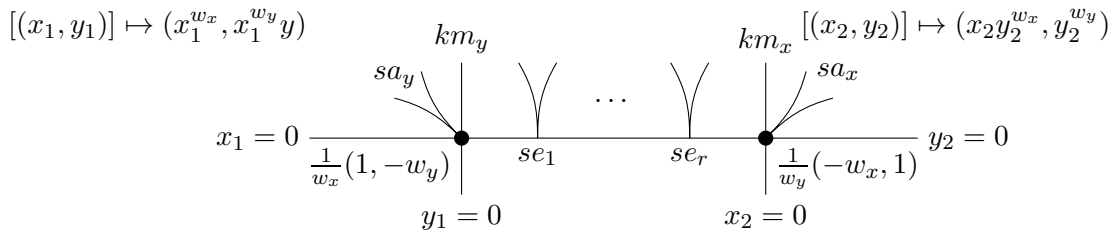

The multiplicity of the exceptional divisor is
\[
k(\p\mq_x+\q\mq_y)+\nu=k\overbrace{(\p\mq_x+\q\mq_y+\r)}^{\newmult}.
\]
Let $u,v \in \ZZ$ such that $u \p +v \q=1$. 
Let us consider the weighted blow-up of type $(N\p-v,1)$ at the quotient singularity of order $\p$, where $N\gg0$. 
The multiplicity of this auxiliar exceptional divisor ($\bmod{\,k\newmult}$) is 
\begin{gather}\nonumber
\frac{(N\p-v)k\newmult+k\mq_y+s a_y}{\p}\equiv\frac{k\mq_y+s a_y-kv(\p\mq_x+\q\mq_y+\r)}{\p}\equiv\\
\label{eq:minfty}
k\underbrace{(u\mq_y-v\mq_x)}_{-\mdet}+s\frac{a_y-v(\p a_x+\q a_y+ \p \q \xp)}{\p}\equiv
s(\underbrace{-\vdeg\mdet+u a_y-v a_x-v\q\xp}_{\xp_\infty}).
\end{gather}
We proceed in a similar way with a weighted blow-up of type $(1,N\q-u)$ at the quotient singularity of order $\q$. 
The multiplicity of this auxiliar exceptional divisor ($\bmod{\,k\newmult}$) is 
\begin{gather}\label{eq:m0}
s(\underbrace{\vdeg\mdet+v a_x-u a_y-u\p\xp}_{\xp_0}).
\end{gather}
Hence the $\gcd$ of the multiplicity at each new divisor and its neighboring divisors 
(which can be replaced by the multiplicities computed above) is
\begin{equation}\label{eq:mcd}
s\gcd(\vdeg\newmult,\xp_0,\xp_\infty,\xp_1,\dots,\xp_r)=s\underbrace{\gcd(\newmult,\xp_0,\xp_\infty,\xp_1,\dots,\xp_r)}_{\delta_\omega}.
\end{equation}
Note that $\gcd(\mcd,\vdeg)=1$ and there exist $\bzt_1,\bzt_2\in\ZZ$ such that $\bzt_1\mcd+\bzt_2\vdeg=1$.

Let us summarize the notation and some formulas introduced in this section.

\begin{enumerate}[label=($\mathcal{C}$\arabic{enumi})]
\item\label{c1} $\kappa=\frac{k}{s}$ with the condition $\gcd(\vdeg,a_x,a_y,\xp_1,\dots,\xp_r)=1$.
\item $\xp=\xp_1+\dots+\xp_r$.
\item $\nu_0=\p a_x+\q a_y+ \p \q \xp=\frac{\nu}{s}$.
\item\label{c4} $\r=\frac{\nu_0}{\vdeg}=\frac{\nu}{k}$.
\item\label{c5} $\newmult=\p\mq_x+\q\mq_y+\r$.
\item\label{c6} $u \p +v \q=1$.
\item\label{c7} $\xp_0\ =v a_x-u a_y-u\p\xp+\vdeg c$.
\item\label{c8} $\xp_\infty=u a_y-v a_x-v\q\xp-\vdeg c$.
\item\label{c9} $\delta_\omega=\gcd(\newmult,\xp_0,\xp_\infty,\xp_1,\dots,\xp_r)$.
\item\label{c10} $\bzt_1\mcd+\bzt_2\vdeg=1$.
\end{enumerate}

\subsection{Semistable normalization in dimension 3}
\label{sec:SNdim3}
\mbox{}

In this section we use the blow-up performed in~\S\ref{sec:SNdim2} at the embedded surface singularity in dimension 3.
In particular, consider a chart where $P$ is the origin and the local equation of the total transform of the 
surface singularity~$V$ is of the form:
\[
(x^{\mq_x} y^{\mq_y})^{m+k} z^m (z^k-h(x,y))=0;
\]
The $\nu$-form of $z^k-h(x,y)$ is
\[
\prod_{j=1}^{s}
\left(z^\vdeg -\zeta_s^j x^{a_x} y^{a_y}\prod_{i=1}^{r}(y^\p-\gamma_i x^\q)^{\xp_i}\right)=
z^k -x^{s a_x} y^{s a_y}\prod_{i=1}^{r}(y^\p-\gamma_i x^\q)^{s \xp_i}.
\]
The exceptional component of the blowing-up is isomorphic to $\PP^2_\omega$. 
The multiplicity of this divisor equals
\begin{equation}
\label{eq:degd}
d:=(\p \mq_x+\q \mq_y+\r) \overbrace{(\mq+ k)}^{\mk}=\newmult\mk,
\end{equation}
where $\newmult$ has been introduced in \ref{c5}.
Since the total transform $V'$ of $V$ is linearly equivalent to the zero divisor,
this is also the case for its intersection  with $E\cong\PP^2_\omega$.
This intersection can be decomposed into two terms:
the divisor obtained by the intersection of $V'$ with $E$ and with the components of $V'$ 
which are different from $E$. The latter can be described as
\begin{equation*}
D:=\mk(\mq_x X+\mq_y Y)+ \mq Z+\sum_{j=1}^s G_j,
\end{equation*}
where
\begin{equation}
\label{eq:eqsLY1}
X:x=0,\quad Y: y=0,\quad Z:z=0, \quad G_j:z^\vdeg -\zeta_j x^{a_x} y^{a_y}\prod_{i=1}^{r}(y^\p-\gamma_i x^\q)^{\xp_i}=0,
\end{equation}
and the divisors $G_j$ are irreducible. 
The former, that is, $E^2$ is linearly equivalent to $H$ where $H$ is any divisor in $\PP^2_\omega$ 
of degree~$1$, e.g. $H\sim u X + v Y$ as $\deg_\omega H=1$, see~\ref{c6}. Hence the intersection is $D-dH\sim 0$, the support of 
a meromorphic function. From the point of view of a cyclic cover of degree~$D$ the geometric ramification occurs 
$\bmod{d}$ and this is why we have provided the decomposition $D-dH$.
From these data one can choose a particular presentation of $\cl(\PP^2_\omega)$ as
\begin{equation}
\label{eq:clP2w}
\langle
X,Y,Z,F_1,\dots,F_{r},F\mid
\r X\!\sim\!\p Z, \r Y\!\sim\!\q Z, \q X\!\sim\!\p Y\!\sim\!F\!\sim\!F_i, 1\leq i\leq r
\rangle
\end{equation}
where $G_1\sim\dots\sim G_s\sim \vdeg  Z$, $F_i$ is the \emph{line}
joining $[0:0:1]_\omega$ and $[1:\gamma_i^{\frac{1}{\p}}:0]_\omega$,
and $F$ is a generic \emph{line} through~$[0:0:1]_\omega$.

Hence, it makes sense to consider the cyclic covering ramified along $(d,D,H)$. Equivalently, using
Reduction~\ref{red:horizontal}, one can consider the covering associated with $(d,D',0)$ where
\[
D'=\mk(\mq_x-u\newmult) X+ \mk(\mq_y-v\newmult) Y + \mq Z+\sum_{j=1}^s G_j.
\]
In order to express this cover as that of a $\QQ$-normal crossing divisor, one needs to perform some 
birational transformations. This is where the reducible normal fake quadrics come into play.
Let us start with the $(\p,\q)$-blow-up of~$[0:0:1]_\omega$. The new surface will
be denoted by $\Sigma$ and the exceptional component of this blow-up by~$E$. 
For simplicity, we keep the notation for the strict transforms in $\Sigma$. 
Taking into account that the total transform of $X$ (resp. $Y$) in $\Sigma$ is $X+\frac{\p}{\r} E$
(resp. $Y+\frac{\q}{\r} E$),
it is not hard to check that 
\[
\cl(\Sigma)=\langle
X,Y,Z,F_1,\dots,F_{r},F,E\mid
\r(u X\!+\!v Y)\!=\! Z-E, 
\q X=\p Y\!=\!F_i\!=\!F, 1\leq i\leq r
\rangle.
\]
The multiplicity of $E$ as a component of the total transform of $D'$ is
\[
\mk\frac{(\mq_x-u\newmult) \p + (\mq_y-v\newmult) \q}{\r}=-\mk,
\]
see \ref{c4}
and \ref{c5} in page~\pageref{c4}. 
By the previous calculations, note that the new ramification divisor $D_\Sigma$ can be described as
\[
D_\Sigma=\mk(\mq_x-u\newmult) X+ \mk(\mq_y-v\newmult) Y-\mk E + m Z+\sum_{j=1}^s G_j.
\]
Next, one can perform generalized Nagata transformations at the fibers over the following points of 
$E\equiv \PP^1$: $0\equiv[0:1],\infty\equiv[1:0],\gamma_i\equiv[\gamma_i:1]$, $i=1,\dots,r$. The goal of
these transformations is to separate the divisors $G_j$ from $Z$ keeping them away from $E$, which can be done
using the following weights:
\begin{equation}
\label{eq:indicesE}
\frac{1}{\gcd(\vdeg ,a_x)}\left(\vdeg ,a_x\right) \textrm{at } 0,
\frac{1}{\gcd(\vdeg ,a_y)}\left(\vdeg ,a_y\right) \textrm{at } \infty,\
\frac{1}{\gcd(\vdeg ,\xp_i)}\left(\vdeg ,\xp_i\right) \textrm{at } \gamma_i.
\end{equation}
Moreover we obtain a reducible normal fake quadric $S$ as in section~\ref{sec:alternative} as one can
check by computing the self-intersection of the strict transform of $Z$ in~$S$:
\begin{gather}\label{eq:self}
(Z\cdot Z)_S=(Z\cdot Z)_\Sigma-\sum_{i=1}^r\frac{\xp_i}{\vdeg}-\frac{a_x}{\vdeg\q}-\frac{a_y}{\vdeg\p}=
\frac{\r}{\p\q}-\frac{a_x\p+a_y\q+\p\q\xp}{\vdeg\p\q}=0.
\end{gather}
To obtain a full description of $S$, let us describe its cyclic quotient singular points.
If the new exceptional divisors are denoted by $A_0,A_\infty,A_1,\dots,A_{r}$, then there are two singular points 
on each $A_i$, $i\in I=\{0,1,...,r,\infty\}$ whose combinatorial data $(d_i,q_i)$, as described in section~\ref{subsec:S},
can be calculated using~\cite[Thm.~4.3]{AMO-Intersection} in terms of the invariants introduced in \eqref{eq:nu2}, \eqref{eq:minfty}, and \eqref{eq:m0} as,
see also \ref{c1}, \ref{c7} and \ref{c8}:
\begin{equation}
\label{eq:di1}
d_i=
\frac{\vdeg }{\gcd(\vdeg ,\xp_i)},\qquad q_i=\frac{\xp_i}{\gcd(\vdeg ,\xp_i)},\ i\in I.
\end{equation}
Note that
\begin{equation}
\label{eq:lcmdi}
\lcm_{i\in I}d_i=\frac{\vdeg }{\gcd(\vdeg ,\xp_i,i\in I)}=\vdeg.
\end{equation}
The first equality in~\eqref{eq:lcmdi} is a consequence of~\eqref{eq:di1} and \eqref{eq:gcd-lcm}
from the appendix. 
For the second equality one needs to show 
\begin{equation}
\label{eq:kei}
\gcd(\vdeg ,\xp_i,i\in I)=1.
\end{equation}
Since $\ZZ=\ZZ\langle\vdeg,a_x,a_y,\xp_1,\dots,\xp_r\rangle$, 
it is enough to show that $a_x,a_y\in \ZZ\langle\vdeg,\xp_i,i\in I\rangle$, see~\ref{c1}.
This is a consequence of 
\[
\begin{pmatrix}
v&-u\\
\p&\q
\end{pmatrix}
\begin{pmatrix}
a_x\\
a_y
\end{pmatrix}
\equiv\!
\begin{pmatrix}
0\\
0
\end{pmatrix}\!
\bmod \ZZ\langle\vdeg,\xp_i,i\in I\rangle
\text{ and }
\left|
\begin{matrix}
v&-u\\
\p&\q
\end{matrix}
\right|
=1
.
\]
As a word of caution, note that we are not imposing that the values $q_i$ be in $\{1,\dots,d_i-1\}$
(see also Remark~\ref{rem:dif_dq}), only $\gcd(d_i,q_i)=1$. There could be other choices of $q_i$ 
corresponding with different generalized Nagata transformations, their remainder classes $\bmod{\ d_i}$ 
however remain unchanged. 
According to our construction and~\eqref{eq:self}, one can check that $\alpha=\sum_{i\in I}\frac{q_i}{d_i}=0$.

Finally, let us compute the total transform $D_S$ of $D_\Sigma$ in $S$. 
First, the multiplicity of $A_i$, $i=1,\dots,{r}$ as a component of $D_S$ is given by:
\begin{equation}
\label{eq:coeffsAi1}
m \frac{\xp_i}{\gcd(\vdeg ,\xp_i)}+s\frac{\xp_i \vdeg }{\gcd(\vdeg ,\xp_i)}=\mk q_i.
\end{equation}
On the other hand, the multiplicity of $A_0$ as a component of $D_S$ is:
\begin{gather}
\nonumber
\mk\frac{a_x+\vdeg (\mq_x -u\newmult)}{\q{\gcd(\vdeg ,\xp_0)}}=
\mk\frac{a_x+\vdeg \mq_x(1 - u \p)-u\vdeg(\mq_y\q+\r)}{\q{\gcd(\vdeg ,\xp_0)}}=\\
\mk\frac{a_x+\vdeg \mq_x v\q-u\vdeg \mq_y\q-u(\p a_x+\q a_y +\p\q\xp)}{\q{\gcd(\vdeg ,\xp_0)}}=
\mk\frac{a_x v+\vdeg\mdet-u(a_y +\p\xp)}{\gcd(\vdeg ,\xp_0)}=\mk q_0.%
\label{eq:coeffsA01}
\end{gather}
Analogously, the multiplicity of $A_\infty$ is $\mk%
q_\infty%
$.
Then, the ramification divisor $D_S$ is
\begin{equation*}
D_S=\mk\sum_{i\in I} q_i A_i+ m Z-\mk E+\sum_{j=1}^s G_j.
\end{equation*}
Finally note that
\begin{equation}\label{eq:clase_S}
\cl(S)=
\langle
A_0,A_\infty,A_1,\dots,A_{r},F,Z,E \left| \right. d_i A_i\sim F, i\in I, E-Z\sim\sum_{i\in I} q_i A_i
\rangle.
\end{equation}
The last step is to apply the contents of section~\ref{sec:general-coverings} to the covering
of $S$ associated with $(d,D_S,0)$. Following the constructive proof of Theorem~\ref{thm:gencase} and its 
notation, the subset of indices~$J_1$ corresponds with the irreducible divisors $Z,E,G_1,\dots,G_s$ and 
then $\mu_1=1$ and also $\mu_0=1$. Then, there is no contribution from the vertical part~\ref{casoa} and
from the mixed part~\ref{casoc}. Let us check the horizontal part~\ref{casob}. 
The irreducible components of $D_S$ involved in $J_2$ are $A_i$, $i\in I$, and hence, using \eqref{eq:gcd-gcd},
one obtains
\begin{gather*}
\gcd(\mk\newmult, \mk q_i, i\in I)\!=
\mk\gcd\left(\newmult, \frac{\xp_i}{\gcd(\vdeg ,\xp_i)}, i\in I\right)=
\mk\gcd\left(\newmult, \frac{\gcd(\xp_i, i\in I)}{\gcd(\vdeg ,\xp_i, i\in I)}\right)=
\mk\mcd.
\end{gather*}
The last equality is a consequence of~\eqref{eq:m0} and~\eqref{eq:kei}, see also \ref{c9}. Recall that $d=\mk\newmult$ by~\eqref{eq:degd}
and $\mcd | \newmult$ by~\eqref{eq:mcd}. Thus the monodromy of the original $d$-cover coincides with the 
monodromy of a horizontal $(\mk\mcd)$-cover associated with $(\mk\mcd,D_h,H_h)$, where, see~\ref{c10},
\begin{gather*}
D_h=m Z-\mk E+\sum_{j=1}^s G_j,\quad
H_h=-\frac{1}{\mcd} \sum_{i\in I} q_i A_i\sim -\bzt_1\sum_{i\in I} q_i A_i\sim \bzt_1(Z-E)
\end{gather*}
(one can easily check that the relation $D_h-\mk\mcd H_h\sim 0$ holds).

Finally, after applying the process described in section~\ref{sec:horizontal-coverings}, 
the horizontal cover satisfies $n=1$ and~$\t=\mk\mcd$. 

As a consequence of Theorem~\ref{thm:gencase} and the discussion above, one has the following.

\begin{prop}\label{prop:yomdin}
Under the conditions and notation described above, the invariant subspaces of the monodromy action on the 
semistable reduction $S$ of~\eqref{eq:eqsLY1} coincide with the \emph{greatest common $(\mk\mcd)$-covering} of 
the restrictions to the special fibers $A_i$, $i\in I$.
\end{prop}

Let us reprove the above result without the use of the torsion arguments. Let us consider the restriction 
$\pi_F:F_{\mk\mcd}\to F$ of the $(\mk\mcd)$-subcover to a general fiber~$F$. The ramification happens 
at $Z\cap F$ (with multiplicity~$m$), $E\cap F$ (with multiplicity~$-\mk$), and $G_j\cap F$ ($\kappa$ points with multiplicity~$1$). The following congruences $\bmod{(\mk\mcd)}$ hold:
\begin{equation}\label{mhat}
m\equiv m-\mk\mcd\bzt_1\equiv m- \mk(1-\bzt_2\vdeg)\equiv\vdeg\overbrace{(\bzt_2\mk-s)}^{\hat{m}}.
\end{equation}
Then, there is a pull-back diagram
\begin{equation}\label{pull-back-F-2}
\begin{tikzcd}
\fd\ar[d,"\pi_F"]\ar[r]&X\ar[d,"\pi_X"]\\
F\ar[r]&\PP^1\\[-25pt]
z\ar[r,mapsto]&z^\vdeg
\end{tikzcd}
\end{equation}
where the covering $\pi_X$ has ramification indices $\hat{m}$ (at~$0$), $1$ (at $s$ points in~$\CC^*$), 
and $-(s+\hat{m})$ (at~$\infty$).

A similar diagram exists if we replace $F$ by $A_i$ and $\vdeg$ by $\frac{\vdeg}{d_i}$, for $i\in I$.
In order to see this, one needs to know the multiplicity of the ramification divisor at $Z\cap A_i$, 
which is obtained by performing a $(q_i',1)$-blowing-up as in the proof of Proposition~\ref{prop:purehorizontal}, 
i.e $q_i'q_i=h_i d_i-1$. The computed multiplicity is 
\[
\frac{\mk q_i q_i'+m}{d_i}=\mk h_i-s\frac{\vdeg}{d-i}.
\]
Also, one need to check the following congruence $\bmod{\mk\mcd}$
\[
\frac{\vdeg}{d_i}{(\bzt_2\mk-s)}\equiv\mk h_i-s\frac{\vdeg}{d_i}\Longleftrightarrow
\frac{\vdeg}{d_i}\bzt_2\mk\equiv\mk h_i,
\]
which is equivalent to $\frac{\vdeg}{d_i}\bzt_2\equiv h_i\bmod{\mcd}$. 
Since $\gcd(d_i,q_i)=1$ and $\mcd$ is a divisor of~$q_i$ the above congruence is equivalent to 
$1-\bzt_1\mcd=\vdeg\bzt_2\equiv d_i h_i\equiv 1+q_i q_i'\bmod{\mcd}$ which is obviously true. 
Then the following diagram also holds:
\begin{equation}\label{pull-back-Ai-2}
\begin{tikzcd}
\ai{\mk\mcd}\ar[d,"\pi_{A_i}"]\ar[r]&X\ar[d,"\pi_X"]\\
A_i\ar[r]&\PP^1\\[-25pt]
z\ar[r,mapsto]&z^{\frac{\vdeg}{d_i}}.
\end{tikzcd}
\end{equation}
This provides an alternative proof of the result since it means that $\mk\mcd$ is the maximal degree where the 
diagrams~\eqref{pull-back-F-2} hold. This describes the curve cover that contains all the necessary information.

\begin{cor}
The characteristic polynomial of the monodromy action on the 
semistable reduction $S$ of~\eqref{eq:eqsLY1} is
\[
\frac{(t-1)^{2-s}(t^{\mk\mcd}-1)^s}{(t^{\mk}-1)(t^{\gcd(m,s\mcd)}-1)}
\]
\end{cor}

\begin{proof}
From Proposition~\ref{prop:yomdin} we know that it coincides with the characteristic polynomial of 
$\pi_X$ from \eqref{pull-back-F-2}. We use a zeta-function argument. The covering admits 
a stratification with a dense strata of Euler characteristic $-s$ where each point has $\mk\mcd$
preimages (the unramified part). There are $s$ points with only one preimage. For the other $2$
remaining points the number of preimages are $\gcd(\mk\mcd,\hat{m})$ and 
$\gcd(\mk\mcd,-\hat{m}-s)=\gcd(\mk\mcd,\bzt_2\mk)=\mk$. Since 
$\hat{m}=\bzt_2\mk-s=m\bzt_2-\bzt_1\mcd s$ we have
\[
\begin{pmatrix}
\hat{m}\\
\mk\mcd
\end{pmatrix}=
\begin{pmatrix}
\bzt_2&-\bzt_1\\
\mcd&\vdeg
\end{pmatrix}
\begin{pmatrix}
m\\
s\mcd
\end{pmatrix},
\]
and as the square matrix is unimodular then 
$\gcd(\mk\mcd,\hat{m})=\gcd(s\mcd,m)$. Hence
\[
\frac{(t-1)^2}{\Delta_1(t)}
=
\frac{\Delta_2(t)\Delta_0(t)}{\Delta_1(t)}
=\zeta(t)=
(t-1)^s(t^\mk-1)(t^{\gcd(s\mcd,m)}-1) (t^{\mk\mcd}-1)^{-s}.
\qedhere
\]
\end{proof}

\appendix

\section{Basic arithmetic}

Some basic well-known properties of the floor and ceiling functions, the greatest common divisor, 
and the least common multiple are collected here for completeness. The purpose is to recall some of the 
most frequently used properties in this paper so the reader can be referred to them, simplifying the 
overall exposition.

\subsection{Floor and ceiling} \mbox{}

Given a real number $a\in\RR$, let us denote by $\lfloor a \rfloor$ its integral part (or floor)
and by $\{ a \}$ its decimal part so that one can write
\[
a = \lfloor a \rfloor + \{ a \},
\]
where $\lfloor a \rfloor$ is the unique integer satisfying $a-1<\lfloor a \rfloor \leq a$
and $0 \leq \{ a \} < 1$. The ceiling (or roundup) is denoted by $\lceil a \rceil$,
it satisfies $\lceil a \rceil = - \lfloor -a \rfloor$ and it is the unique integer
such that $a \leq \lceil a \rceil < a+1$.

Two more useful properties will be used throughout this paper. Let $a \in \RR$ and $m, n \in \ZZ$,
then
\begin{equation}
\label{eq:floorfloor}
\left\lfloor \frac{\lfloor a \rfloor}{n} \right\rfloor
= \left\lfloor \, \frac{a}{n} \, \right\rfloor
\quad \text{and} \quad
\left\lceil \frac{m+1}{n} \right\rceil
= \left\lfloor \frac{m}{n} \right\rfloor + 1,
\end{equation}
and, consequently, 
\begin{equation}
\label{eq:floor}
\left\lfloor \frac{-1-m}{n} \right\rfloor
= -1 - \left\lfloor \frac{m}{n} \right\rfloor.
\end{equation}

\subsection{Greatest common divisor and least common multiple} \mbox{}

Let $m_1,\ldots,m_r,k \in \ZZ$, then
\begin{equation}\label{eq:gcd-gcd}
\gcd\left( \frac{m_1}{\gcd(k,m_1)},\ldots,\frac{m_r}{\gcd(k,m_r)} \right)
= \frac{\gcd(m_1,\ldots,m_r)}{\gcd(k,m_1,\ldots,m_r)}.
\end{equation}
If, in addition, $m \in \ZZ$ is a multiple of all the $m_i$'s, then
\begin{equation}\label{eq:gcd-lcm}
\lcm\left( \frac{m}{m_1},\ldots,\frac{m}{m_r} \right) = \frac{m}{\gcd(m_1,\ldots,m_r)}.
\end{equation}
Let $d_1,\dots,d_r\in\ZZ\setminus\{0\}$ such that there exist $q_1,\dots,q_r\in\ZZ\setminus\{0\}$, 
$\gcd(q_i,d_i)=1$, for which
\[
\sum_{i=1}^r\frac{q_i}{d_i}\in\ZZ;
\]
then
\begin{equation}\label{eq:gcd-lcm2}
\lcm( d_1, \ldots, d_r ) = \frac{\gcd \left( \left\{ \frac{d_1\cdot\ldots\cdot d_r}{d_i} \mid i = 1,\ldots,r \right\} \right)}
{\gcd \left( \left\{ \frac{d_1\cdot\ldots\cdot d_r}{d_i d_j} \mid i,j = 1,\ldots,r, \ i \neq j \right\} \right)}.
\end{equation}
The two first results are classical; let us proof this last one (the technique can be easily adapted by the reader for the other proofs).
The hypothesis implies that for any $i$, $d_i$ divides the $\lcm$ of all the $d_j$'s, but $d_i$.

Fix a prime number $p$ and consider the valuation $\nu_p$ associated to this prime. Let $n_i:=\nu_p(d_i)$. For the sake of simplicity, let us order the numbers $d_1,\dots,d_r$ such that $n_1\leq\dots\leq n_r$. The condition that $d_r$ divides 
$\lcm(d_1,\dots,d_{r-1})$ implies that $n_r\leq\max(n_1,\dots,n_{r-1})=n_{r-1}\leq n_r$, i.e., $n_r=n_{r-1}$.
The $p$-valuation of the left-hand side of \eqref{eq:gcd-lcm2} equals $\max(n_1,\dots,n_{r})=n_r$. 
The $p$-valuation of the numerator of the right-hand side equals
\[
\min_{1\leq i\leq r}\left((n_1+\dots+n_r)-n_i\right)=
(n_1+\dots+n_r)-\max_{1\leq i\leq r}\left(n_i\right)=n_1+\dots+n_{r-1}.
\]
The $p$-valuation of the denominator of the right-hand side equals
\[
\min_{1\leq i<j\leq r}\left((n_1+\dots+n_r)-n_i-n_j\right)=
(n_1+\dots+n_r)-\max_{1\leq i<j\leq r}\left(n_i+n_j\right)=n_1+\dots+n_{r-2}
\]
and we obtain that both sides have the same $p$-valuation. Since it is valid for all primes,
the equality follows.

\section{Coverings of the projective line}

The following proof has been inspired by the calculations in
\cite[p.~547]{Steenbrink77}.

\begin{lemma}\label{covering_proj_line}
Let $F: Y \to X$ be a cyclic branched covering of $e$ sheets between two orientable compact surfaces.
Denote by $\Delta$ the characteristic polynomial of the monodromy of $F$
acting on the cohomology groups. Assume $Y$ has $r$ connected components. Then,
$$
\Delta_{H^1(Y)}(t) = \frac{(t^r-1)^2 \cdot (t^e-1)^{-\chi(\check{X})}}{\prod_{Q \in \mathcal{R}} \left( t^{\m(Q)}-1 \right)},
$$
where $\mathcal{R}$ is the ramification set of $F$, $\check{X}:= X \setminus \mathcal{R}$, and $\m(Q)$ denotes
the number of preimages of $Q$.
\end{lemma}

\begin{proof}
Let us write $\check{Y}:= Y \setminus F^{-1}(\mathcal{R})$. Since $F_|: \check{Y} \to \check{X}$ is a topological covering
of $e$ sheets, one gets that $\Delta_{H^0(\check{Y})}(t) \cdot \Delta^{-1}_{H^1(\check{Y})}(t) = (t^e-1)^{\chi(\check{X})}$.
One the other hand, by virtue of Mayer-Vietoris, there is an exact sequence
$$
0 \longrightarrow H^1(Y) \longrightarrow H^1(\check{Y}) \longrightarrow
H^0(Y \setminus \check{Y}) \longrightarrow H^2(Y) \longrightarrow 0.
$$
Therefore
$$
\Delta_{H^1(Y)}(t) = \frac{\Delta_{H^0(\check{Y})}(t) \cdot (t^e-1)^{-\chi(\check{X})} \cdot \Delta_{H^2(Y)}(t)}
{\Delta_{H^0(Y \setminus \check{Y})}(t)}.
$$

The monodromy on $H^0(\check{Y}) \cong H^2(Y)$ is the rotation of the corresponding $r$ irreducible components.
Hence its characteristic polynomial is $t^r-1$. Finally every point $Q \in \mathcal{R}$ gives the factor
$t^{\m(Q)}-1$ in $\Delta_{H^0(Y \setminus \check{Y})}(t)$.
\end{proof}

\bibliographystyle{amsplain}
\providecommand\noopsort[1]{}
\providecommand{\bysame}{\leavevmode\hbox to3em{\hrulefill}\thinspace}
\providecommand{\MR}{\relax\ifhmode\unskip\space\fi MR }
\providecommand{\MRhref}[2]{%
  \href{http://www.ams.org/mathscinet-getitem?mr=#1}{#2}
}
\providecommand{\href}[2]{#2}

\end{document}